\documentclass[10pt]{amsart}
\usepackage{egastyle}

\usepackage{amssymb,amsmath,amscd,amsthm,amsfonts,wasysym,mathrsfs,amsxtra,stmaryrd,array}

\usepackage{graphicx,array,url}
\usepackage[vcentermath]{genyoungtabtikz}
\usepackage{tikz}
\usepackage[section]{placeins}
\usepackage{afterpage}
\usepackage{verbatim}
\usepackage{tikz-cd}

\input xy
\xyoption{all}

\newcommand{\RR}{\mathbb{R}}
\newcommand{\OO}{\mathcal{O}}

\newcommand{\image}{\textnormal{im}\,}
\newcommand{\kernel}{\textnormal{ker}\,}
\newcommand{\cokernel}{\textnormal{coker}\,}

\newcommand{\Hom}{\textnormal{Hom}}

\newcommand{\Ext}{\textnormal{Ext}}

\newcommand{\Ac}{\mathcal{A}}
\newcommand{\Bc}{\mathcal{B}}

\newcommand{\Fc}{\mathcal{F}}
\newcommand{\Tc}{\mathcal{T}}
\newcommand{\Dc}{\mathcal{D}}
\newcommand{\Hc}{\mathcal{H}}
\newcommand{\Pc}{\mathcal{P}}
\newcommand{\Uc}{\mathcal{U}}
\newcommand{\Coh}{\mathrm{Coh}}

\newcommand{\arinj}{\ar@{^{(}->}}
\newcommand{\arsurj}{\ar@{->>}}
\newcommand{\areq}{\ar@{=}}

\newcommand{\wh}{\widehat}

\newcommand{\Bl}{\mathcal{B}^l}

\newcommand{\ch}{\mathrm{ch}}

\newcommand{\Stab}{\mathrm{Stab}}

\newcommand{\Aut}{\mathrm{Aut}}

\newcommand{\whPhi}{{\wh{\Phi}}}

\newcommand{\wt}{\widetilde}

\newcommand{\olw}{{\overline{\omega}}}

\newcommand{\bsm}{\begin{smallmatrix}}
\newcommand{\esm}{\end{smallmatrix}}

\newcommand{\CLoovc}{{\mathbb{C}(\!(\tfrac{1}{v})\!)^c}}
\newcommand{\RLoovc}{{\mathbb{R}(\!(\tfrac{1}{v})\!)^c}}

\newcommand{\GLlp}{{\mathrm{GL}^{l,+}\!(2,\RLoovc)}}

\newcommand{\C}{\mathbb{C}}
\newcommand{\mC}{\mathcal{C}}
\newcommand{\A}{\mathcal{A}}
\newcommand{\B}{\mathcal{B}}
\newcommand{\T}{\mathcal{T}}
\newcommand{\F}{\mathcal{F}}

\newcommand{\HN}{\text{HN}}

\newcommand{\cP}{\mathcal{P}}
\newcommand{\cQ}{\mathcal{Q}}
\newcommand{\cO}{\mathcal{O}}
\newcommand{\NS}{\text{NS}}

\makeatletter
\newtheorem*{rep@theorem}{\rep@title}
\newcommand{\newreptheorem}[2]{%
\newenvironment{rep#1}[1]{%
 \def\rep@title{#2 \ref{##1}}%
 \begin{rep@theorem}}%
 {\end{rep@theorem}}}
\makeatother

\newenvironment{que}[1][]%
  {\begin{genthm}{Question}{true}{#1}{paragraph}}%
  {\end{genthm}}

  {\begin{genthm}{Question}{true}{#1}{equation}}%
  {\end{genthm}}

\newreptheorem{theorem}{Theorem}

\usepackage{scalerel,stackengine}
\stackMath
\newcommand\reallywidehat[1]{%
\savestack{\tmpbox}{\stretchto{%
  \scaleto{%
    \scalerel*[\widthof{\ensuremath{#1}}]{\kern-.6pt\bigwedge\kern-.6pt}%
    {\rule[-\textheight/2]{1ex}{\textheight}}
  }{\textheight}%
}{0.5ex}}%
\stackon[1pt]{#1}{\tmpbox}%
}
\newenvironment{nouppercase}{%
  \renewcommand{\uppercasenonmath}[1]{}}{}

\begin{document}

\title[Stability for line bundles and dHYM equation on elliptic surfaces]{Stability for line bundles and deformed Hermitian-Yang-Mills equation on some elliptic surfaces}

\author[Tristan C. Collins]{Tristan C. Collins}
\address{Department of Mathematics \\
Massachusetts Institute of Technology\\
77 Massachusetts Avenue\\
Cambridge MA 02139 \\
USA}
\email{tristanc@mit.edu}

\author[Jason Lo]{Jason Lo}
\address{Department of Mathematics \\
California State University, Northridge\\
18111 Nordhoff Street\\
Northridge CA 91330 \\
USA}
\email{jason.lo@csun.edu}

\author[Yun Shi]{Yun Shi}
\address{Department of Mathematics \\
Brandeis University\\
415 South Street\\
Waltham MA 02453\\
USA}
\email{yunshi@brandeis.edu}

\author[Shing-Tung Yau]{Shing-Tung Yau}
\address{Yau Mathematical Sciences Center\\
Tsinghua University\\
Haidian District, Beijing\\
China}
\email{styau@tsinghua.edu.cn}


\makeatletter
\@namedef{subjclassname@2020}{%
  \textup{2020} Mathematics Subject Classification}
\makeatother



\begin{abstract}
We study the twisted ampleness criterion due to Collins, Jacob and Yau on surfaces, which is equivalent to the existence of solutions to the deformed Hermitian-Yang-Mills (dHYM) equation.  When $X$ is a Weierstra{\ss} elliptic K3 surface, and $\omega$ an ample class such that $\omega$ lies in the span of a section class and the fiber class, we show that for a class of line bundles $L$ with fiber degree 1 and $\omega c_1(L)>0$, the twisted ampleness of $L$  respect to $\omega$ always implies the $\sigma_{\omega, 0}$-stability (Bridgeland stability) of $L$. This answers a question by Collins and Yau for a class of examples. 
\end{abstract}

\begin{nouppercase}
\maketitle
\end{nouppercase}


\tableofcontents
\section{Introduction}

In the 1980s, Donaldson and Uhlenbeck-Yau established the following fundamental result:

\begin{thm}\cite{DonaldsonASDYMC, UYHYMC}
    Let $(X,\omega)$ be a compact K\"{a}hler manifold, and $E \to X$ an irreducible holomorphic vector bundle.  Then $E$ admits a Hermitian metric solving the Hermitian-Yang-Mills equation if and only if $E$ is a Mumford-Takemoto stable vector bundle.
\end{thm}

With the development of the theory of deformed Hermitian Yang-Mills (dHYM) equation and the theory of Bridgeland stability conditions, the following question was asked by Collins and Yau:

\begin{que}\cite{CoShdHYM}\label{que:main-1}
Let $L$ be a holomorphic line bundle on $X$.  What is the relationship between the following statements?

(i) $L$ admits a metric solving the deformed Hermitian-Yang-Mills equation,

(ii) $L$ is a Bridgeland stable object in $D^b(\Coh(X))$.
\end{que}

By the work of Collins-Jacob-Yau \cite{CJY}, on a K\"{a}hler surface,  the existence of solutions to the dHYM equation for a $(1,1)$-form is equivalent to an algebraic, positivity condition, which we shall refer to as the ``twisted ampleness" condition.  Briefly, suppose the $L\rightarrow X$ is a line bundle.  Define complex numbers
\[
Z_{X}(L) = -\int_{X}e^{-\sqrt{-1}\omega}ch(L), \quad Z_{C}(L) = -\int_{C}^{-\sqrt{-1}\omega}ch(L)
\]
where $C\subset X$ is any curve.  Suppose that $\Im(Z_{X}(L))>0$.  Then we say that $L$ is {\em twisted ample} (with respect to $\omega$) if, for every curve $C\subset X$ there holds
\[
\Im \left(\frac{Z_{C}(L)}{Z_{X}(L)}\right) >0.
\]
We refer the reader to Section~\ref{sec: que} for more discussion.  Using this formulation, Collins and Shi \cite[Theorem 4.15]{CoShdHYM} answered Question ~\ref{que:main-1} on the blowup $\Bl_P(\mathbb{P}^2)$ of $\mathbb{P}^2$ at a point. The answer is a direct consequence of Arcara-Miles \cite{ARCARA20161655}.  Precisely, twisted ampleness for $L$ (equivalently, the existence of a solution to the dHYM equation) implies that $L$ is Bridgeland stable, but not conversely. In particular, there exist line bundles and K\"ahler classes which are Bridgeland stable, but are not ``twisted ample".  

The work of \cite{CoShdHYM} leaves open the possibility that twisted ampleness (or equivalently the solvability of the dHYM equation) always implies Bridgeland stability on surfaces.  In this article, we take up this question on Weierstra{\ss} elliptic surfaces $X$.  These surfaces differ from the blowup $\Bl_P(\mathbb{P}^2)$ in two fundamental ways: the Picard rank of $X$ could be strictly higher than 2, and there exists an autoequivalence of $D^b(\Coh (X))$ that is a relative Fourier-Mukai transform. In particular, when $X$ is also a K3 surface, we answer Question \ref{que:main-1} for a class of line bundles of fiber degree 1. Our first main result is:

\begin{thm}\label{thm:main-1}
    Let $X$ be a Weierstra{\ss} elliptic K3 surface.  Suppose $L$ is a line bundle on $X$ and $\omega$ is an ample class on $X$, such that $c_1(L), \omega$ both lie in the span of the section class $\Theta$ and fiber class $f$.  Then under the assumptions $fc_1(L)=1, \omega c_1(L)>0$, we have:
    \begin{itemize}
        \item If we further assume $c_1(L) \neq \Theta, \Theta + f$, then  the line bundle $L$ is stable with respect to the Bridgeland stability $\sigma_{\omega,0}$. 
        \item If $c_1(L)=\Theta, \Theta + f$, then whenever $L$ is twisted ample with respect to $\omega$, the line bundle $L$ is $\sigma_{\omega, 0}$-stable.
    \end{itemize} 
    In particular, under the assumptions $fc_1(L)=1, \omega c_1(L)>0$, we have
    \[
    \text{$L$ is twisted ample with respect to $\omega$} \Rightarrow \text{$L$ is $\sigma_{\omega, 0}$-stable}.
    \]
    Furthermore, there exist line bundles which are $\sigma_{\omega,0}$ stable, but which are not twisted ample.
\end{thm}

Note that the restriction $\omega c_1(L) >0$, which is equivalent to $\Im(Z_{X}(L))>0$ appears in the reformulation of the existence of solutions to the dHYM equation in terms of twisted ampleness; morally, this condition is needed to ensure that the line bundle $L$ lies in the heart of a Bridgeland stability condition (so that it makes sense to speak of the Bridgeland stability of $L$).

Writing $\alpha := c_1(L)=\Theta + (D+e)f$ where $e = -\Theta^2$, our proof of  Theorem \ref{thm:main-1} is divided into two main cases according to the sign of the parameter $D$, with very different approaches. 

When $D$ is negative, we begin by proving a new estimate for the Bridgeland walls for line bundles.  The twisted ampleness condition arises as a key part of this estimate - see Theorem \ref{thm:main1}.  By using a `going-up lemma' in conjunction with autoequivalences on the elliptic surface, we extend the region of stability for $L$ to the entire region $\omega \alpha >0$ except when the coefficient of $f$ in $\omega$ is strictly between $-e$ and $0$ (Theorem \ref{thm:Dalphaneg-main1}). Note that this part of the argument actually works for any Weierstra{\ss} ellipic surface, without the K3 assumption. 

When $D$ is nonegative, the above approach is not sufficient when $\omega$ is small.  In order to reach the region where $\omega$ is small, we are forced to consider the notion of weak stability conditions, which are limits (or `degenerations') of Bridgeland stability conditions.  We then consider the actions of autoequivalences on weak stability together with the going up lemma to finish the argument.

The going-up lemma is a key technical result that links the two approaches together - it can be thought of as a "preservation of stability" result for $L$ under a 1-dimensional deformation (along a vertical line).  When combined with the relative Fourier-Mukai transform on an elliptic surface, the going-up lemma becomes a 2-dimensional deformation result and thus more powerful.

More concretely, since the ample divisors $\omega$ we consider lie in the span of $\Theta$ and $f$, we can always write $\omega$ in the form \[
  \omega = R_\omega (\Theta + (D_\omega + e)f) 
\]
for some $R_\omega, D_\omega >0$.  If we also set
\[
  V_\omega = \frac{\omega^2}{2},
\]
then $\omega$ is completely determined by $D_\omega$ and $V_\omega$, allowing us to think of $\omega$ as a point on the $(D_\omega, V_\omega)$-plane.  The going-up lemma then roughly says the following: if a line bundle $L$ is stable with respect to the Bridgeland stability $\sigma_{\omega, 0}$, then $L$ is also stable with respect to the Bridgeland stability $\sigma_{\omega^\ast, 0}$ where 
\[
  D_{\omega^\ast}= D_\omega, \, \, V_{\omega^\ast}>V_\omega.
\]
That is, a line bundle $L$ remains stable when we "go up" vertically on the $(D_\omega, V_\omega)$-plane.  Since a Weierstra{\ss} elliptic surface $X$ comes with a relative Fourier-Mukai transform $\Phi$ that switches the $D_\omega$ and $V_\omega$ coordinates, the going up lemma implies that the stability of $L$ at $\sigma_{\omega, 0}$ also implies the stability of $L$ at $\sigma_{\omega^\bullet, 0}$ where
\[
 D_{\omega^\bullet}> D_\omega, \, \, V_{\omega^\bullet}=V_\omega.
\]
Overall, the stability of $L$ at $\sigma_{\omega, 0}$ implies the stability of $L$ at $\sigma_{\omega',0}$ where
\[
 D_{\omega'}> D_\omega, \, \, V_{\omega'}>V_\omega.
\]
In the case of $D <-e$, this is sufficient for us to extend the stability of $L$ from a region determined by the twisted ampleness criterion and an estimate for the outermost Bridgeland wall, to the entire region of $\omega c_1(L) >0$.  When $X$ is K3 (i.e.\ $e=2$), and when $-e \leq D <0$, i.e.\ when $D = -2, -1$, we need to assume that the twisted ampleness criterion already holds in order to conclude the stability of $L$ (Theorem \ref{thm:conj-Daotcase}).


In the case where $D \geq 0$, however, the above argument is no longer sufficient to prove the stability of $L$ in the entire region $\omega c_1(L) >0$, the issue being there is no obvious way to  reach the region on the $(D_\omega, V_\omega)$-plane near the origin.  As a result, we need to consider limits of Bridgeland stability conditions where  Bridgeland stability conditions degenerate to "weak stability conditions".  First, we identify a weak stability condition such that $\Psi L$, for some autoequivalence $\Psi$ of $D^b(\Coh(X))$, is stable with respect to a weak stability condition $\tau_R$.  Then, by studying the preimage of $\tau_R$ with respect to $\Psi$, we construct another weak stability condition $\tau_L$  such that $L$ itself is $\tau_L$-stable.  The important point is that we can informally think of $\tau_L$ as a weak stability condition corresponding to the origin of the $(D_\omega, V_\omega)$-plane.

Then, by generalising  the going up lemma to weak stability conditions, we obtain the stability of $L$ for all weak stability conditions on the $V_\omega$-axis (i.e.\ $D_\omega=0$ and $V_\omega >0$) - these weak stability conditions look like Bridgeland stability conditions, except that they are constructed using nef divisors rather than ample divisors.  Applying the relative Fourier-Mukai transform $\Phi$ then gives the stability of $\Phi L$ along the $D_\omega$-axis (i.e.\ $V_\omega=0$ and $D_\omega >0$).  If we then apply the generalised going-up lemma to $\Phi L$, we obtain the stability of $\Phi L$ for all $D_\omega, V_\omega>0$.  Finally, since $\Phi$ merely exchanges the $D_\omega$- and $V_\omega$-coordinates, we obtain the stability of $L$ itself for all $D_\omega, V_\omega >0$, proving Theorem \ref{thm:main-1} for the $D \geq 0$ case.

In the course of the proof we note that under the assumptions of Theorem \ref{thm:main-1}, there exist line bundles $L$ such that the region where $L$ is twisted ample does not cover the region $\omega c_1(L)>0$ in the first quadrant of the $D$, $V$ plane. Hence in the case considered in Theorem \ref{thm:main-1}, we find pairs $(L, \omega)$ such that $L$ is $\sigma_{\omega, 0}$ stable, but $L$ is not twisted ample with respect to $\omega$. 

As pointed out above, most of the results for $D<0$ actually works for any Weierstra{\ss} elliptic surface (See Section 7 for details). If we restrict $X$ to a Weierstrass elliptic K3 surface, we obtain the same result for a wider class of line bundles, which is our second main result:
\begin{thm}
Let $X$ be a Weierstra{\ss} elliptic K3 surface, and $\omega$ be an ample class on $X$. 
Let $L$ be a line bundle on $X$ such that $fc_1(L)=1$, $-\Theta c_1(L)+\ch_2(L)=1$ and $\Theta c_1(L)>-1$. 
Then $L$ is $\sigma_{\omega, 0}$-stable for all $\omega\in\langle\Theta, f\rangle$.
\end{thm}

The proof of Theorem 1.4 is completely built on the methods and results of the $D\geq 0$ case of Theorem 1.3.  
\paragraph[Outline of the article]
In Section \ref{sec:prelim}, we set up some notations for derived categories, elliptic surfaces, and Bridgeland stability conditions on surfaces. 
In Section \ref{sec: que}, we recall some notions and results from deformed Hermitian-Yang-Mills (dHYM) equation. In particular, we recall an algebraic criterion for the dHYM equation to admit a solution and state the main question of this paper. 
In Section \ref{sec:geom-goingup}, we show the stability of pure 1-dimensional sheaf with the same Chern character as $\cO_f$, which implies that the image of certain Bridgeland stability conditions remain geometric under the relative Fourier-Mukai transform. 
In Section \ref{sec:mainstabeq}, we state an equation of group action on Bridgeland stability conditions, which is used frequently throughout the article. 
In Section \ref{sec:astabcriforlbs}, we give a criterion for stability of line bundles. This result is completely general - it applies to any line bundle and ample class on an arbitrary projective surface. 

For the rest of the paper, we restrict to the case of Bridgeland stability conditions of the form $\sigma_{\omega, 0}$, where $\omega\in\langle\Theta, f\rangle$, and line bundles of fiber degree $1$. In Section \ref{sec:Dalphanegative}, we answer the main question for a subclass of line bundles of fiber degree $1$ when $D < 0$. 
For the $D \geq 0$ case, an important tool we use is a generalization of weak stability conditions  developed in \cite{CLSY1}. In Section \ref{sec:defwsc}, we recall the notion of weak stability conditions and the explicit examples of weak stability conditions which will be used in the later sections. 
In Section \ref{sec:wpsuga}, we prove correspondences of stability of objects in weak polynomial stability data under group actions. In Section \ref{sec:Dalphageqzero}, we pull together the results in Section \ref{sec:defwsc} and \ref{sec:wpsuga} to answer the main question for   $D\geq 0$ and prove Theorem 1.4. 

\paragraph[Acknowledgements]
JL was partially supported by NSF grant DMS-2100906.  TCC was partially supported by NSF CAREER grant DMS-1944952.  Part of this work was done when YS was a postdoc at CMSA, Harvard, she would like to thank CMSA for the excellent working environment.


\section{Preliminaries} \label{sec:prelim}



\paragraph[Notation] Let $X$ be a smooth projective variety $X$.  We will write $D^b(X)$ to denote $D^b(\Coh (X))$, the bounded derived category of coherent sheaves on $X$.  For any divisors $D_1, \cdots, D_m$ on $X$, we will write $\langle D_1, \cdots, D_m \rangle_\mathbb{R}$ to denote the set of all $\mathbb{R}$-divisors that are linear combinations of $D_1, \cdots, D_m$.

If $\Ac$ is the heart of a bounded t-structure on a triangulated category $\Tc$, then we will write $\Hc^i_\Ac$ to denote the $i$-th cohomology functor with respect to the t-structure.  For any integers $j<k$, we define  the full subcategory of $\Tc$ 
\[
 D^{[j, k]}_\Ac = \{ E \in \Tc : \Hc^i_\Ac (E) =0 \text{ for all } i \notin [j, k] \}.
\]

\paragraph[Weierstra{\ss} elliptic surface] \label{para:def-Wellsurf}  By a Weierstra{\ss} elliptic surface, we mean a flat morphism  $p : X\to Y$ of smooth projective varieties of relative dimension 1, where $X$ is a surface and
\begin{itemize}
    \item the fibers of $p$ are Gorenstein curves of arithmetic genus 1, and are geometrically integral;
    \item $p$ has a section $s : Y \to X$ such that its image $\Theta$ does not intersect any singular point of any singular fiber of $p$.
\end{itemize}
The definition we use here follows that of 
 \cite[Definition 6.10]{FMNT}.  Under our definition, the generic fiber of $p$ is a smooth elliptic curve, and the singular fibers of $p$ are either nodal or cuspidal curves.  We usually write $f$ to denote the class of a fiber for the fibration $p$, and write $e = -\Theta^2$.  Often, we simply refer to $X$ as a Weierstra{\ss} elliptic surface.  Note that when $Y=\mathbb{P}^1$, $X$  is K3 if and only if $e=2$ \cite[2.3]{LLM}.
 
\paragraph[An autoequivalence on elliptic surfaces]\label{para:WESauto}  Given a Weierstra{\ss} elliptic surface $p : X \to Y$, we can always construct an autoequivalence $\Phi$ of $D^b(X)$ satisfying the following properties:
\begin{itemize}
    \item There exists another autoequivalence $\whPhi$ of $D^b(X)$ such that $\whPhi \Phi \cong \mathrm{id}[-1] \cong \Phi \whPhi$.
 \item For any closed point $x \in X$, the structure sheaf $\OO_x$ is taken by $\Phi$
 to a rank-one torsion-free sheaf of degree 0 on the fiber of $p$ containing $x$.
 \end{itemize}
 In fact, $\Phi$ is a relative Fourier-Mukai transform whose kernel is given by a normalised Poincar\'{e} sheaf, and $\whPhi$ has a similar description.  The construction of $\Phi, \whPhi$ and their properties can be found in \cite[6.2.3]{FMNT}.  
 
\paragraph[Cohomological Fourier-Mukai transform] \label{para:cohomFMT} The autoequivalence $\Phi$ plays a central role in many of the computations in the remainder of this article.  Its corresponding  cohomological Fourier-Mukai transform is given as follows \cite[6.2.6]{FMNT}: Suppose $E \in D^b(X)$ and we write 
\begin{equation}\label{eq:ellipsurfchE}
\ch_0(E)=n,\text{\quad}  f\ch_1(E)=d, \text{\quad} \Theta \ch_1(E)=c,\text{\quad}  \ch_2(E)=s.
\end{equation}
Then we have the following formulas from \cite[(6.21)]{FMNT}
\begin{align*}
\ch_0(\Phi E) &= d, \\
\ch_1 (\Phi E) &= -\ch_1(E) +  (d-n)\Theta + (c+\tfrac{1}{2} ed+s)f, \\
\ch_2 (\Phi E) &= -c-de+\tfrac{1}{2}ne.
\end{align*}
In particular, 
\[
  f\ch_1(\Phi E) = -n, \text{\quad} \Theta \ch_1(\Phi E) = (s-\tfrac{e}{2}d)+ne.
\]


\paragraph[RDV coordinates for divisors] \label{para:RDVcoord} Suppose $X$ is a Weierstra{\ss} elliptic surface.  Given any divisor $M$ on $X$ of the form $M = a\Theta + bf$ where $a, b \in \mathbb{R}$ and $a \neq 0$, we can find real numbers $R_M \neq 0$ and $D_M$ such that
\begin{equation}\label{eq:RDVcoord-2}
  M = R_M (\Theta + (D_M + e)f).
\end{equation}
We also set
\[
  V_M = \frac{M^2}{2}=R_M^2(D_M+\tfrac{e}{2}).
\]
Note that when $D_M, V_M>0$ (e.g.\ when $M$ is ample), we can write
\[
  R_M = \sqrt{\frac{V_M}{D_M+\tfrac{e}{2}}}.  
\]

The coordinates $R_M, D_M, V_M$ for divisors $M$ are especially suited for computations on elliptic fibrations. For example, we have 
\begin{equation}\label{eq:RDVcoord-1}
  \Theta M=R_MD_M
\end{equation}
and, if $W$ is a divisor  written in the  form  \eqref{eq:RDVcoord-2}, then 
\[
 MW = R_M R_W (\Theta + (D_M+D_W + e)f)
\]
which is reminiscent of multiplication for complex numbers in polar coordinates.




\paragraph[WIT$_i$ objects] Suppose  $\Psi : \Tc \to \Uc$ is an exact equivalence of triangulated categories, and $\Ac, \Bc$ are the  hearts of t-structures on $\Tc, \Uc$, respectively.  We say an object $E \in \Tc$ is $\Psi_\Bc$-WIT$_i$ if $\Psi E \in \Bc [-i]$ for some $i \in \mathbb{Z}$.  We also define the full subcategories of $\Ac$ 
\[
  W_{i,\Psi, \Ac, \Bc} = \{ E \in \Ac : E \text{ is $\Psi_{\Bc}$-WIT$_i$}\}
\]
for $i \in \mathbb{Z}$.  When $\Tc=\Uc =D^b(X)$ for some smooth projective variety $X$ and $\Ac= \Bc=\Coh (X)$, we simply say $\Psi$-WIT$_i$ to mean $\Psi_{\Coh (X)}$-WIT$_i$.

If $\wh{\Psi} : \Uc \to \Tc$ is an exact equivalence such that $\wh{\Psi} \Psi \cong \mathrm{id}_{\Tc}$ and $\Psi \wh{\Psi} \cong \mathrm{id}_{\Uc}$, $\Psi \Ac \subset D^{[0,1]}_{\Bc}$, and $\Ac, \Bc$ are hearts of bounded t-structures, then whenever $E \in \Ac$ is $\Psi_\Bc$-WIT$_i$, the object $\Psi E [i]$ is necessarily $\wh{\Psi}_\Ac$-WIT$_{1-i}$ \cite[Lemam 3.10(a)]{Lo20}.

\paragraph[Bridgeland stability conditions]  Let $X$ be a smooth projective variety over $\mathbb{C}$.

\begin{defn}\label{defn: slicing}
A {\em slicing} $\cP$ of $D^{b}(X)$ is a collection of subcategories $\cP(\phi) \subset D^{b}(X)$ for all $\phi \in \mathbb{R}$ such that
\begin{enumerate}
\item $\cP(\phi)[1] = \cP(\phi+1)$,
\item if $\phi_1 > \phi_2$ and $A\in \cP(\phi_1)$, $B \in \cP(\phi_2)$, then ${\rm Hom}(A,B) =0$,
\item every $E\in D^{b}(X)$ admits a Harder-Narasimhan (HN) filtration by objects in $\cP(\phi_i)$ for some $1 \leq i \leq m$.
\end{enumerate}
\end{defn}
\begin{defn}\label{defn: BrStab}
A Bridgeland stability condition on $D^{b}(X)$ with central charge $Z$ is a slicing $\cP$ satisfying the following properties
\begin{enumerate}
\item For any non-zero $E\in \cP(\phi)$ we have
\[
Z(E) \in \mathbb{R}_{>0} e^{\sqrt{-1}\phi},
\]
\item
\[
C := \inf \left\{ \frac{|Z(E)|}{\|\ch(E)\|} : 0 \ne E \in \cP(\phi), \phi \in \mathbb{R} \right\} >0
\]
where $\| \cdot \|$ is any norm on the finite dimensional vector space $H^{\text{even}}(X, \mathbb{R})$.
\end{enumerate}
\end{defn}

Let $\mathbb{H}$ denote the upper-half complex plane together with the negative real axis
\[
  \{ re^{i\pi \phi} : r \in \mathbb{R}_{>0}, \phi \in (0,1]\},
\]
and let $\mathbb{H}_0 = \mathbb{H}\cup \{0\}$.

\begin{prop}[Bridgeland, \cite{StabTC}]\label{prop: BrStab}
	\label{pro1}
	A Bridgeland stability condition on $D^b(X)$ is equivalent to the following data: the heart $\mathcal{A}$ of a bounded t-structure on $D^b(X)$, and a central charge $Z: K(\mathcal{A})\rightarrow \mathbb{C}$ such that for every nonzero object $E\in \mathcal{A}$, one has 
 
 (i) $Z(E)\in \mathbb{H}$,
 
 (ii) $E$ has a finite filtration
	\begin{equation*}
	0=E_0\subset E_1\subset...\subset E_{n-1}\subset E_n=E
	\end{equation*} 
	such that $\HN_i(E)=E_i/E_{i-1}$'s are semistable objects in $\mathcal{A}$ with decreasing phase $\phi$. Furthermore, the central charge satisfies Definition \ref{defn: BrStab} (2).
\end{prop}

Let $\omega\in \mathrm{NS}_\mathbb{R}(X)$ be an ample class, and $B\in \mathrm{NS}_\mathbb{R}(X)$. For any object $E\in \Coh (X)$, we define 
\begin{equation*}
\mu_{\omega, B}(E)=\frac{\omega\cdot \ch_1^B(E)}{\ch_0(E)}.
\end{equation*}
If $E$ is a torsion sheaf, we set $\mu_{\omega, B}(E)=\infty$. Then for any torsion-free sheaf $E\in \Coh(X)$, $E$ admits a Harder-Narasimhan filtration:
\begin{equation*}
	0=E_0\subset E_1\subset E_2 ...\subset E_n=E
\end{equation*}
with $\mu_{\omega, B}(E_{i+1}/E_i)$ monotonically decreasing. The factors $E_{i+1}/E_i$ are called the Harder-Narasimhan factors of $E$, and are denoted by $\HN_i(E)$.
Consider the following subcategories of $\Coh(X)$:

\begin{center}
$\mathcal{T}_{\omega, B}=\{E\in \Coh(X)| \mu_{\omega, B}(\HN_i(E))>0 \text{ for any } i\}$, 
\end{center}
\begin{center}
$\mathcal{F}_{\omega, B}=\{E\in \Coh(X)| \mu_{\omega, B}(\HN_i(E))\leq 0 \text{ for any } i\}$, 
\end{center}
Define an abelian subcategory of $D^b(X)$ by tilting $\Coh(X)$ at the torsion pair $(\T_{{\omega, B}}, \F_{\omega, B})$:
\begin{equation*}
\Coh^{\omega, B}=\langle \F_{\omega, B}[1], \T_{\omega, B}\rangle
\end{equation*}
Equivalently, we can describe $\Coh^{\omega, B}$ as
\begin{equation*}
\Coh^{\omega, B}= \{E\in D^b(X)| H^0(E)\in \mathcal{T}_{\omega, B}, H^{-1}(E)\in \mathcal{F}_{\omega, B}, H^i(E)=0 \text{ for } i\neq -1, 0\}.
\end{equation*}

In this paper, we mainly work with Bridgeland stability conditions of the following forms. 

(i) $\sigma_{\omega, B}=(Z_{\omega, B}, \Coh^{\omega, B})$, where the central charge is defined by 
\begin{equation}
\label{Equ:CenCharNoTodd}
Z_{\omega, B}(E)=-\int e^{-i\omega}\ch^B(E).
\end{equation}

(ii) $\sigma^{td}_{\omega, B}=(Z^{td}_{\omega, B}, \Coh^{\omega, B})$, where the central charge is defined by 
\begin{equation}
\label{Equ:CenCharTodd}
Z^{td}_{\omega, B}(E)=-\int e^{-i\omega}\ch^B(E)\sqrt{\mathrm{td}(X)}.
\end{equation}

(iii) $\sigma_{V, D}=(Z_{V, D}, \Coh^{\omega, 0})$, where the central charge is defined by 
\begin{equation}
\label{Equ:CenCharVD}
Z_{V, D}(E)=-\ch_2(E)+V\ch_0(E)+i(\Theta+(D+e)f)\cdot \ch_1(E).
\end{equation}

\paragraph[Group actions on stability conditions] \label{para:Brimfdgroupactions} Given a triangulated category $\Tc$, there are two natural group actions on the space $\Stab (\Tc)$ of stability conditions on $\Tc$ \cite[Lemma 8.2]{StabTC}: a left action by the autoequivalence group $\Aut (\Tc)$, and a right action by the universal cover $\wt{\mathrm{GL}}^+\!(2,\mathbb{R})$ of $\mathrm{GL}^+(2,\mathbb{R})$.  

Given a stability condition $\sigma = (Z, \Ac) \in \Stab (\Tc)$ where $Z$ is the central charge and $\Ac$ the heart of a t-structure, the action of an element $\Phi \in \Aut (\Tc)$ on $\sigma$ is given by 
\[
  \Phi \cdot (Z, \Ac) = (Z \circ (\Phi^K)^{-1}, \Phi (\Ac))
\]
where $\Phi^K$ is the automorphism of the Grothendieck group $K(\Tc)$ induced by $\Phi$.  We will often simply write $\Phi^{-1}(E)$ instead of $(\Phi^K)^{-1}(E)$ for $E \in \Tc$.

Recall that we can describe elements of $\wt{\mathrm{GL}}^+\! (2,\mathbb{R})$ as pairs $(T,f)$ where $T \in \mathrm{GL}^+(2,\mathbb{R})$ and $f : \mathbb{R} \to \mathbb{R}$ is an increasing function satisfying $f(\phi+1)=f(\phi)+1$ for all $\phi$, such that the restriction of $f$ to $\mathbb{R}/2\mathbb{Z}$ coincides with the restriction of $T$ to $(\mathbb{R}^2\setminus \{0\})/\mathbb{R}_{>0}$.  Using this description of $\wt{\mathrm{GL}}^+\!(2,\mathbb{R})$, and supposing we write $\sigma = (Z,\Pc)$ where $\Pc$ is the slicing for $\sigma$, 
the action of an element $g =(T,f)$ on $\sigma$ is given by 
\[
 (Z,\Pc)\cdot g = (Z', \Pc')
\]
where $Z' = T^{-1}Z$ and $\Pc'(\phi) =\Pc (f(\phi))$ for all $\phi$.

\section{Statement of the question}\label{sec: que}

Let us recall some basic aspects of the deformed Hermitian-Yang-Mills (dHYM) equation.  Let $(X,\omega)$ be a compact K\"ahler manifold of dimension $n$ and let $\mathfrak{a} \in H^{1,1}(X,\mathbb{R})$ be a cohomology class.  The deformed Hermitian-Yang-Mills equation seeks a $(1,1)$-form $\alpha \in \mathfrak{a}$ such that
\[
\frac{(\omega+\sqrt{-1}\alpha)^n}{\omega^n}: X\rightarrow  \mathbb{R}_{>0}e^{\sqrt{-1}\hat{\theta}}
\]
for some constant $e^{\sqrt{-1}\hat{\theta}}\in S^1$.  The dHYM equation was discovered independently in \cite{MMMS, LYZ} as the equations of motion for a BPS brane (with abelian gauge group) in type IIB mirror symmetry.  In this vein it is worth recalling that the concept of $\Pi$-stability (a precursor to Bridgeland stability) was introduced by Douglas-Fiol-R\"omelsberger \cite{DFR} as an attempt to describe, in algebraic terms, BPS branes in type IIB mirror symmetry. Due to their shared physical origins, it is natural to ask how the existence of solutions to the dHYM equation is related to notions of stability in algebraic geometry, particularly Bridgeland stability.  Starting with the work of Collins-Jacob-Yau \cite{CJY}, and continuing in \cite{CollinsYaugeo, CollinsYau} the first and last named authors have pursued the following line of questioning: what are the necessary and sufficient algebraic conditions for the existence of a solution to the dHYM equation?  How are these conditions related to notions of Bridgeland stability?  We shall focus henceforth on the case when $\dim_{\mathbb{C}}X=2$, which is considerably simpler than the general case.  We refer the interested reader to the survey article \cite{CoShdHYM} and the references therein for further discussion of the higher dimensional case.  For progress relating algebraic conditions to existence of solutions to the dHYM equation, we refer the reader to \cite{GaoChen, ChuLeeTak, ChuLee, DatarPing} and the references therein.

When $\dim_{\mathbb{C}}X=2$ the following observation was made in \cite{CJY}, as a consequence of Yau's solution of the complex Monge-Amp\`ere equation \cite{YauMA} and the Nakai-Moishezon ampleness criterion (or more generally the Demailly-P\u{a}un theorem \cite{DP}):

\begin{lem}[\cite{CJY}, Proposition 8.5]\label{lem: dHYMSurf}
    Suppose $(X,\omega)$ is a K\"ahler surface, and let $C\subset X$ be a curve.  Define
    \[
    Z_{X} = -\int_{X}e^{-\sqrt{-1}(\omega+\sqrt{-1}\alpha)}, \quad Z_{C} = -\int_{C}e^{-\sqrt{-1}(\omega+\sqrt{-1}\alpha)}.
    \]
    Suppose that $\Im(Z_{X})>0$.  Then there exists a solution of the dHYM equation if and only if  for every curve $C\subset X$ there holds
    \begin{equation}\label{eq: GITObstruction}
     \Im\left(\frac{Z_{C}}{Z_{X}}\right) >0.
    \end{equation}
    
\end{lem}

We remark that Lemma~\ref{lem: dHYMSurf} gives a full description of the set of classes $\mathfrak{a}\in H^{1,1}(X,\mathbb{R})$ admitting a solution of the dHYM equation.  Indeed, if $\Im(Z_{X})<0$ then one can replace $\mathfrak{a} \mapsto -\mathfrak{a}$.  On the other hand, if $\Im(Z_{X})=0$ then the dHYM equation reduces to the Laplace equation, and hence $\mathfrak{a}$ always admits a solution of the dHYM equation.  Thus, Lemma~\ref{lem: dHYMSurf}  gives necessary and sufficient conditions for the existence of solutions to the dHYM equation, and is therefore a natural testing ground for any possible correspondence between existence of solutions and the dHYM equation.  Furthermore, the existence of Bridgeland stability conditions is well understood for complex surfaces \cite{ABL}. This was pursued by the first and third authors \cite{CoShdHYM}, who found the following result (direct consequence of Section 5 of Arcara-Miles \cite{ARCARA20161655})

\begin{prop}[\cite{CoShdHYM}, Theorem 4.15]\label{prop: GITimpliesBS}
    Let $X=\Bl_P(\mathbb{P}^2)$ be the blow-up of $\mathbb{P}^2$ in a point.  Let $\omega$ be any K\"ahler form and $L\rightarrow X$ any line bundle.  Then:
    \begin{itemize}
    \item If there is a solution of the dHYM equation on $(X,\omega)$ in the cohomology class $c_1(L)$, then $L$ is Bridgeland stable for a Bridgeland stability condition with central charge $Z=-e^{-\sqrt{-1}\omega}ch(L)$.
    
    \item Conversely, there exist K\"ahler forms $\omega$ and line bundles $L\rightarrow X$ that are Bridgeland stable, but do not satisfy~\eqref{eq: GITObstruction}, and hence cannot admit solutions of the dHYM equation.
    \end{itemize}
\end{prop}

Given Proposition~\ref{prop: GITimpliesBS}, one may (very optimistically) wonder whether~\eqref{eq: GITObstruction} always implies Bridgeland stability on K\"ahler surfaces.  Even more optimistically, one may wonder whether a similar phenomenon holds in all dimensions, after replacing~\eqref{eq: GITObstruction} with the higher dimensional obstructions from \cite{CollinsYau, CoShdHYM}.  From the point of view of Bridgeland stability, we are asking whether~\eqref{eq: GITObstruction} is an effective numerical criterion for Bridgeland stability of line bundles; it is worth recalling that determining whether or not a given line bundle is Bridgeland stable is highly non-trivial in general.  We therefore pose:

\begin{que}\label{que: motivating}
    Suppose $(X,\omega)$ is a K\"ahler surface and let $L\rightarrow X$ be a line bundle. Define
     \[
    Z_{X}(L) = -\int_{X}e^{-\sqrt{-1}\omega}ch(L), \quad Z_{C}(L) = -\int_{C}e^{-\sqrt{-1}\omega}ch(L).
    \]
    where $C\subset X$ is a curve.  Suppose that $\Im(Z_{X}(L))>0$ and suppose that for every curve $C\subset X$ we have
    \begin{equation}\label{eq: queIntersection}
    \Im\left(\frac{Z_{C}(L)}{Z_{X}(L)}\right)>0.
    \end{equation}
    Is $L$ Bridgeland stable?
\end{que}

We can reinterpret the intersection theoretic condition in Question~\ref{que: motivating} as a "twisted ampleness" condition.  To see this, suppose that $\Im(Z_{X}(L))\geq 0$ (otherwise replace $L$ with $-L$).  By the Hodge Index theorem we have $Z_{X}(L)\in \mathbb{R}_{>0}e^{\sqrt{-1}\hat{\theta}}$ for some $\hat{\theta}\in (0, \pi)$.  Then~\eqref{eq: queIntersection} is equivalent to
\begin{equation}\label{eq: twistAmple}
\int_{C}\cot(\hat{\theta})\omega+c_1(L) >0.
\end{equation}
It is not hard to show that if~\eqref{eq: twistAmple} holds for all curves $C\subset X$ then in fact the class $[(\cot(\hat{\theta})\omega+c_1(L)]$ is a K\"ahler class; see \cite[Section 8]{CJY}.  For the purpose of the current work, it will be useful to have some terminology (distinct from ``stability") to refer to this condition.

\begin{defn}
    Let $L\rightarrow(X,\omega)$ be a line bundle.  If the assumptions of Question~\ref{que: motivating} hold, then we shall say that $L$ is ``twisted ample" with respect to $\omega$, or satisfies the ``twisted ampleness criterion" with respect to $\omega$.
\end{defn}





\section{Geometricity of stability conditions}\label{sec:geom-goingup}


Let $X$ be an elliptic surface with a fixed section $\Theta$.  The goal of this section is to prove that any pure 1-dimensional sheaf with the same Chern character as $\OO_f$, the structure sheaf of a fiber, is Bridgeland stable.  This is needed in later sections, where we need to know that the images of Bridgeland stability conditions of the form $\sigma_{\omega, B}$ remain geometric under the autoequivalence $\Phi$ from \ref{para:WESauto}. We follow the idea and methods used in Arcara-Miles \cite{ARCARA20161655}.


\paragraph Let $\omega=\frac{1}{\sqrt{2D_\omega+e}}(\Theta+(D_\omega+e)f)$, $H=\frac{1}{\sqrt{2D_\omega+e}}(\Theta-D_\omega f)$. We have $\omega^2=1$, $H^2=-1$, $\omega\cdot H=0$. Let the $B$-field be $B=y\omega+zH$, $y$, $z\in \mathbb{R}$. Consider  stability conditions of the form $\sigma_{x\omega, B}$; we denote it by $\sigma_{x, y, z}$. Note that $\Coh^{x\omega, B}$ only depends on $y$, so we denote the heart $\Coh^{\omega, B}$ by $\Coh^y$. For $E\in D^b(X)$,
we denote the Chern character of $E$ by 
\begin{equation*}
	\ch(E)=(r, E_\omega\omega+E_HH+\alpha, d)
\end{equation*}
where $\alpha\in\langle \omega, H\rangle^\perp$.

\begin{defn} (\cite{ARCARA20161655, Mac14})
	A potential wall associated to $\cO_f$ is a subset of $\mathbb{R}^{>0}\times\mathbb{R}^2$ consisting of points $(x, y, z)$ such that 
	\begin{equation}
 \label{Equ:PotentialWall}
		\Re Z_{x, y, z}(E)\Im Z_{x, y, z}(\cO_f)=\Re Z_{x, y, z}(\cO_f)\Im Z_{x, y, z}(E)
	\end{equation}
	for some object $E$ of $D^b(X)$.
\end{defn}
The potential wall $W(E, \cO_f)$ in the $\mathbb{R}^3$ with coordinates $(x, y, z)$ are quadric surfaces:

\begin{equation*}
	\frac{1}{2}rx^2+\frac{1}{2}ry^2+\frac{1}{2}rz^2+ryz-E_\omega z-E_Hz-d=0.
\end{equation*}
Consider the intersection of $W(E, \cO_f)$ with the plane $x=0$. Since the discriminant is zero, this is a parabola,  a pair of parallel lines, or coincident lines in the $yz$-plane. Within the intersection of $W(E, \cO_f)$ with $z=z_0$, the wall is a semicircle with center $(x=0, y=-z_0)$.

Using the notations above, when $E\in \Coh^y(X)$, we have 
$$\phi_{x, y, z}(E)>\phi_{x, y, z}(\cO_f)$$
 when $(x, y, z)$ lies in the region bounded by the semicircle $W(E, \cO_f)$ and the $yz$-plane. 

Let $(x_0, y_0, z_0)\in \mathbb{R}^{>0}\times \mathbb{R}^2$.
Consider the following short exact sequence in $\Coh^{y_0}$:
\begin{equation*}
	0\rightarrow E\rightarrow \cO_f\rightarrow Q\rightarrow 0
\end{equation*}
with $\phi_{x_0, y_0, z_0}(E)>\phi_{x_0, y_0, z_0}(\cO_f)$.
\begin{prop}
The sheaf $E$ in the above short exact sequence is torsion-free.
\end{prop}
\begin{proof}
Assume $E$ is not torsion-free, and let 
$$0\rightarrow T\rightarrow E\rightarrow F\rightarrow 0$$
be a short exact sequence with $T$ a torsion sheaf and $F$ a torsion-free sheaf. 
Since $E$ destabilizes $\cO_f$, $E$ is not a torsion subsheaf of $\cO_f$, hence $E$ is not a torsion sheaf. 
Then we have either $T\simeq \cO_f$ or there is a short exact sequence: 
$$0\rightarrow T\rightarrow \cO_f\rightarrow Q'\rightarrow 0$$
for some torsion sheaf $Q'$ supported on dimension zero. 
By the snake lemma (or equivalently, applying the octahedral axiom to the composition $T \to E \to \OO_f$), we have the exact sequence:
$$0\rightarrow H^{-1}(Q)\rightarrow F\rightarrow Q'\rightarrow H^0(Q)\rightarrow 0.$$
Since $Q'$ and $H^0(Q)$ are both supported in dimension zero, we have $\ch_1(H^{-1}(Q))=\ch_1(F)$, and hence $\omega\ch_1(F) \leq 0$; since $E$ is a sheaf lying in  $\Coh^{y_0}$ with nonzero rank, however, we also have $\omega \ch_1(F) >0$ ($F$ being the torsion-free part of $E$), a contradiction.
\end{proof}

Denote the $\mu_\omega$-HN filtration of $E$ by 
\begin{equation*}
	0=E_0\subset E_1\subset E_2 ...\subset E_n :=E.
\end{equation*}
Denote the $\mu_\omega$-HN filtration of $H^{-1}(Q)$ by 
\begin{equation*}
	0=K_0\subset K_1\subset K_2 ...\subset K_m=K:=H^{-1}(Q).
\end{equation*}
Then $\mu_\omega(\mathrm{HN}_1(K))<y_0<\mu_\omega(\mathrm{HN}_n(E))$.

\begin{prop}
\label{Prop:SubobjOf}
If $E$ is a subobject of $\cO_f$ in $\Coh^{y_0}(X)$, then we have 

 (i) For each $1\leq i\leq m$, $E/K_i$ is a subobject of $\cO_f$ in $\Coh^{y_0}(X)$,
 
 (ii) For each $1\leq j\leq n-1$, $E_j$ is a subobject of $\cO_f$ in $\Coh^{y_0}(X)$.
\end{prop}
\begin{proof}
The proof of (i) follows simply from the definition of torsion class and torsion free class. The proof of (ii) is a consequence of $\mu_{\omega, B}(\text{HN}_1(\ker(E_j\rightarrow \cO_f)))\leq 0$.
\end{proof}
Recall the following version of Bertram's Lemma in \cite[Lemma 4.7]{ARCARA20161655}. Denote the intersection of $W(E, \cO_f)$ and the plane $\{z=z_0\}$ by $W(E, \cO_f)_{z_0}$.

\begin{prop}(\cite[Lemma 4.7]{ARCARA20161655}, Bertram's Lemma)
	\label{Prop:BerLem}
	Consider the stability condition $\sigma_{x_0, y_0, z_0}$. Assume that $E\subset \cO_f$ in $\Coh^{y_0}$ and $\ch_0(E)>0$. Then 
	
	(i) If $W(E, \cO_f)_{z_0}$ intersects $y=\mu_\omega(\HN_1(K))$ in the plane $z=z_0$, then the wall $W(E/K_1, \cO_f)_{z_0}$ is outside the wall $W(E, \cO_f)_{z_0}$.
	
	(ii) If $W(E, \cO_f)_{z_0}$ intersects $y=\mu_\omega(\HN_n(E))$ in the plane $z=z_0$,  then the wall $W(E_{n-1}, \cO_f)_{z_0}$ is outside the wall $W(E, \cO_f)_{z_0}$.
\end{prop}
\begin{proof}
The proof is the same as that of \cite[Lemma 4.7]{ARCARA20161655}, with $\cO_X$ replaced with $\cO_f$.
\end{proof}
Using the same ideas and methods as in Section 5 
 of \cite{ARCARA20161655}, we have the following Proposition.

\begin{prop}
	\label{Prop:AMOf}
	Given any $(x_0, y_0, z_0)\in \mathbb{R}^{>0}\times \mathbb{R}^2$, we have $\cO_f$ is $\sigma_{x_0, y_0, z_0}$-stable.
\end{prop}
\begin{proof}
	
	We prove the claim by contradiction. Assume there is a stability condition $\sigma_{x_0, y_0, z_0}$ 
	such that $\cO_f$ is not stable. Then there exists a short exact sequence 
	\begin{equation*}
		0\rightarrow E\rightarrow \cO_f\rightarrow Q\rightarrow 0
	\end{equation*}
	in $\Coh^{y_0}$  where $\phi(E)>\phi(\cO_f)$. We can assume that for any object $F$ with $\mathrm{rk}(F)<\mathrm{rk}(E)$, $F$ destabilizes $\cO_f$ in $\sigma_{x_0, y_0, z_0}$, the wall $W(F, \cO_f)$ is inside $W(E, \cO_f)$.
	
	
	If $W(E, \cO_f)$ is a parabola, we call $W(E, \cO_f)$ a type I wall if $W(E, \cO_f)=\emptyset$ when $z\ll 0$. We call $W(E, \cO_f)$ a type II wall if $W(E, \cO_f)=\emptyset$ when $z\gg 0$. 
	
	First we consider wall of type I. Since $\mu_\omega(\HN_1(K))<y_0<\mu_\omega(\HN_n(E))$, and $y_0$ lies inside the semicircle $W(E, \cO_f)_{z_0}$ intersecting the $yz$-plane, we have the intersection of $W(E, O_f)_z$ and $y=\mu_\omega(\HN_1(K))$ is not empty for some $z$. 

	Let $(x_1, y_1, z_1)$ be the intersection point of $y=\mu_\omega(\HN_1(K))$, $y+z=0$ and $W(E, \cO_f)$. 
	By Bertram's Lemma, we have the wall $W(E/K_1, \cO_f)_{z_1}$ is outside $W(E, \cO_f)_{z_1}$. On the other hand, by Proposition \ref{Prop:SubobjOf} we have 
	$W(E/K_1, \cO_f)_{z_0}$ is inside $W(E, \cO_f)_{z_0}$.
	Since $W(E/K_1, \cO_f)$ and $W(E, \cO_f)$ are not the same. From the equation of the numerical walls, they can only intersect at at most one $z$-value. Then $W(E/K_1, \cO_f)_z$ is outside $W(E, \cO_f)_z$ for all $z>z_1$. In particular, $W(E/K_1, \cO_f)$ is also of type I. 
	
	Since $E/K_i$ are subobjects of $\cO_f$ in $\Coh^{y_0}$ for all $i$. We can repeat the above argument for all $W(E/K_{i}, \cO_f)$. At the $i$-th step, let $(x_i, y_i, z_i)$ be the intersection point of $y=\mu_\omega(\HN_i(K))$, $y+z=0$ and $W(E/K_{i-1}, \cO_f)$. 
	By Bertram's Lemma, we have $W(E/K_{i}, \cO_f)_{z_i}$ is outside $W(E/K_{i-1}, \cO_f)_{z_i}$, and hence outside $W(E, \cO_f)_{z_i}$. By Proposition \ref{Prop:SubobjOf}, we have $W(E/K_{i}, \cO_f)_{z_0}$ is inside $W(E, \cO_f)_{z_0}$, hence $W(E/K_{i}, \cO_f)_z$ is outside $W(E, \cO_f)_z$ for $z>z_i$. Finally we have $W(E/K_{m-1}, \cO_f)_z$ is outside $W(E, \cO_f)_z$ for all $z>z_{m-1}$. Then $W(E/K_{m-1}, \cO_f)$ intersects $y=\mu_\omega(\HN_m(K))$ implies
	$$\phi_{x, y, z}(E/K)>\phi_{x, y, z}(\cO_f)$$ at some $\sigma_{x, y, z}$. But $E/K$ is a subsheaf of $\cO_f$, and can never destabilize $\cO_f$. We obtain a contradiction.
	
	We are left with considering wall of type II. Then we have the intersection of $W(E, \cO_f)$ and $y=\mu_\omega(\HN_i(E))$ is not empty. 
	Let $(x^n, y^n, z^n)$ be the intersection point of $y=\mu_\omega(\HN_n(E))$, $y+z=0$ and $W(E, \cO_f)$. 
	By Bertram's Lemma, we have $W(E_{n-1}, \cO_f)_{z_n}$ is outside $W(E, \cO_f)_{z_n}$. By Proposition \ref{Prop:SubobjOf}, we have $W(E_{n-1}, \cO_f)_{z_0}$ is inside $W(E, \cO_f)_{z_0}$, hence $W(E_{n-1}, \cO_f)_{z}$ is outside the wall $W(E, \cO_f)_{z}$  
	for all $z<z^{n}$. In particular $W(E_{n-1}, \cO_f)$ is also of type II. Applying the same argument repeatedly, finally we have $W(E_1, \cO_f)_z$ is outside wall $W(E, \cO_f)_z$ for all $z<z^2$. This implies $W(E_1, \cO_f)$ intersects $y=\mu_\omega(\HN_1(E))$. Let $(x^1, y^1, z^1)$ be the intersection point of $y=\mu_\omega(\HN_1(E))$, $y+z=0$ and $W(E_1, \cO_f)$. In the plane $z=z^1$, let $y\rightarrow \mu_\omega(\HN_1(E))^-$ along $W(E_1, \cO_f)_{z_1}$, we have $\phi_{x, y, z_1}(E_1)\rightarrow -\infty$, but $\phi_{x, y, z_1}(E_1)=\phi_{x, y, z_1}(\cO_f)$, we obtain a contradiction.
	
	The proof of the case $W(E, O_f)$ is a pair of parallel lines is similar to either the type I or type II case. 
\end{proof}

\begin{rem}
	\label{Rk:Geom}
	Note that the proof only used that $\cO_f$ is pure (in the proof of $E$ is torsion free by snake lemma and $E/K$ does not destabilize $\cO_f$ in the proof of Proposition \ref{Prop:AMOf}), and $\ch(\cO_f)$ (in the computation of $W(E, \cO_f)$). Hence we can replace $\cO_f$ by any pure sheaf $F\in \Coh(X)$ such that $\ch(F)=\ch(\cO_f)$, and the statement of Proposition \ref{Prop:AMOf} still holds.  
\end{rem}

\begin{prop} \label{Prop:Geo}
	Let $X$ be a Weierstra{\ss} elliptic surface, and $\Phi$ the autoequivalence as in \ref{para:WESauto}. 
	For any $D_\omega>0$, $V_{\omega}>0$, $B\in\langle \Theta, f\rangle_{\mathbb{R}}$, the stability condition 
	$\Phi \cdot \sigma_{\omega, B}$ is geometric. 
\end{prop}
\begin{proof}
	Let $\sigma_{\omega, B}$ be a stability condition such that $D_{\omega}>0$, $V_{\omega}>0$. Let $F$ be a 1-dimensional pure sheaf such that $\ch(F)=\ch(\cO_f)$. By Proposition \ref{Prop:AMOf} and Remark \ref{Rk:Geom} we know that $F$ is stable in $\sigma_{\omega, B}$. This implies that $\Phi(F)$ is stable in $\Phi\cdot\sigma_{\omega, B}$. Since any skyscraper sheaf $\C(x)$ is of the form $\Phi(F)$ for some pure sheaf $F$ with $\ch(F)=\ch(\cO_f)$, we have $\C(x)$ is stable in $\Phi\cdot\sigma_{\omega, B}$ for any $x\in X$. 
\end{proof}


\section{An equation for Bridgeland stability conditions under group actions}\label{sec:mainstabeq}



When we apply an autoequivalence $\Phi$ of the derived category to a Bridgeland stability condition of the form $\sigma_{\omega, B}$, we often want explicit descriptions of the central charge and the heart of the resulting stability condition $\Phi \cdot \sigma_{\omega, B}$.  We are able to obtain such explicit descriptions as long as we know $\Phi \cdot \sigma_{\omega, B}$ is a geometric stability condition, such as in the setting of Proposition \ref{Prop:Geo}.

For any $a \in \mathbb{R}$ and $\mathbb{R}$-divisors $\omega, B$, let us write
\[
  Z_{a,  \omega, B} = -\ch_2^B + a\ch_0^B + i \omega \ch_1^B.
\]

\begin{prop}\label{prop:AG52-20-1}
Let $p : X \to Y$ be a Weierstra{\ss} elliptic surface, and suppose $\omega, B \in \langle \Theta, f \rangle_{\mathbb{R}}$ with $\omega$  ample.  Then $\Phi \cdot \sigma_{\omega, B}$ is a geometric stability condition, and  there exist $a'\in \mathbb{R}$, $\mathbb{R}$-divisors $\omega', B'$ on $X$ with $\omega'$  ample, and $\wt{g} \in \wt{\mathrm{GL}}^+\! (2,\mathbb{R})$   such that 
\begin{equation}\label{eq:AG52-21-1}
   \Phi \cdot \sigma_{\omega, B} \cdot \wt{g} = (Z_{a', \omega', B'}, \Coh^{\omega', B'}).
\end{equation}
\end{prop}


\begin{proof}
From the computations in \cite[Appendix A]{LM2}, we know that we can find $a' \in \mathbb{R}$ and $\omega', B' \in \langle \Theta, f\rangle_{\mathbb{R}}$ with $\omega'$ ample, and $g \in \mathrm{GL}^+\! (2,\mathbb{R})$, such that 
\[
  Z_{\omega, B}(\Phi^{-1}(-)) = gZ_{a', \omega', B'}(-).
\]
Note that $\Phi \cdot \sigma_{\omega, B}$ being geometric  follows from Proposition \ref{Prop:Geo}.

Now, suppose $\wt{g} \in \wt{\mathrm{GL}}^+\! (2,\mathbb{R})$ is a lift of $g$ such that with respect to  $\Phi \cdot \sigma_{\omega, B} \cdot \wt{g}$, which has central charge
\[
g^{-1}Z_{\omega, B}(\Phi^{-1}(-))=Z_{a', \omega', B'}(-),
\]
all the skyscraper sheaves $\OO_x$ have phase 1.  Then by \cite[Lemma 3.7]{LM2} (also see \cite[Lemma 6.20]{MSlec}), we know that if $\mathcal{P}$ denotes the slicing of  $\Phi \cdot \sigma_{\omega, B} \cdot \wt{g}$, then $\Pc (0,1] = \Coh^{\omega', B'}$ as wanted.
\end{proof}

\paragraph \label{para:DVsymmetricBzero} In Proposition \ref{prop:AG52-20-1}, the parameters $a', \omega', B'$ and $\wt{g}$ can all be solved explicitly in terms of $\omega, B$ (see \cite[Section 10]{Lo20}).  In particular, when we assume that $B, B'$ are of the form $B=lf, B'=qf$ for $l, q \in \mathbb{R}$,  we can solve for $a', \omega', B'$ using the following simple, symmetric relations (see also \cite[(10.2.2)]{Lo20}):
\begin{equation}\label{eq:AG50-118-1}
  l = \tfrac{e}{2} + q, \text{\quad} V_{\omega} = D_{\omega'}, \text{\quad} D_{\omega} = V_{\omega'} \text{\quad and $a'=V_{\omega'}$},
\end{equation}
and \eqref{eq:AG52-21-1} reads
\begin{equation}\label{eq:AG52-21-2}
  \Phi \cdot \sigma_{\omega, B} \cdot \wt{g} = \sigma_{\omega', B'}.
\end{equation}

\begin{rem}
    The symmetry in \eqref{eq:AG50-118-1} between the $V$- and the $D$-coordinates for $\omega$ and $\omega'$ in the case $B \in \mathbb{Z}f$ was first noticed in the second author's joint work with Martinez and Liu \cite{LLMII}.
\end{rem}

\begin{rem}\label{rem:omegaintremoval}
The original form of Proposition \ref{prop:AG52-20-1} was \cite[Theorem 11.7]{Lo20}, where we required $D_\omega \in \mathbb{Z}$ and $B \in \mathbb{Z}f$.  The integrality conditions were used to ensure that $\whPhi \OO_x$ is Bridgeland stable for any skyscraper sheaf $\OO_x$.  In \cite[Proposition 5.1]{LM2}, the integrality conditions were replaced with an estimate for the outermost Bridgeland wall for torsion sheaves.  By focusing on only torsion sheaves with the same Chern character as $\whPhi \OO_x$ in Proposition \ref{Prop:AMOf}, we are able to improve \cite[Theorem 11.7]{Lo20} to the form of  Proposition \ref{prop:AG52-20-1}.
\end{rem}



\section{A criterion for stability of line bundles}\label{sec:astabcriforlbs}

In this section, we give a region in the ample cone in which twisted ampleness of $L$ always implies the stability of $L$.

\textbf{Notation.} When the $B$-field is zero, we will drop the subscript $B=0$ in notations such as $\mu_{\omega, 0}, \sigma_{\omega, 0}, \Ac_{\omega,0}, Z_{\omega,0}, \phi_{\omega, 0}$,  and write $\mu_\omega, \sigma_\omega, \Ac_\omega, Z_\omega, \phi_\omega$, respectively.

\begin{thm}\label{thm:main1}
Let $X$ be a smooth projective surface.  Suppose $L$ is a line bundle on $X$, and $\omega$ is an ample divisor on $X$.  Let  $\alpha = c_1(L)$, and suppose that the inequalities
\begin{align}
  \omega \alpha &>0 \label{eq:ineq00}\\
  \omega^2 &\geq \omega \alpha  \label{eq:ineq2-b}
\end{align}
hold.  Then when $\alpha$ is twisted ample with respect to $\omega$, the line bundle $L$ is $\sigma_{\omega}$-stable.
\end{thm}

Note that when we assume  \eqref{eq:ineq00}, the inequality \eqref{eq:ineq2-b} can be rewritten as 
\begin{equation}\label{eq:ineq2}
 \omega^2 - \frac{(\omega\alpha)^2}{\omega^2}\geq 0.
\end{equation}
Also, when we assume both \eqref{eq:ineq00} and \eqref{eq:ineq2-b}, we have  
\[
(\omega^2)^2 \geq (\omega \alpha)^2 \geq (\omega^2)(\alpha^2)
\]
where the second inequality follows from the Hodge Index Theorem.  This gives 
\begin{equation}\label{eq:ineq1}
     \omega^2 - \alpha^2 \geq 0
\end{equation}
 which is equivalent to $\Re Z_\omega (L)\geq 0$, i.e.\ $\phi_\omega (L) \in (0,\tfrac{1}{2}]$.


\begin{proof}[Proof of Theorem \ref{thm:main1}]
Suppose $L$ is strictly $\sigma_\omega$-semistable, so that we have an $\Ac_\omega$-short exact sequence
\begin{equation}\label{eq:AG50-98-1}
0 \to A \to L \to Q \to 0
\end{equation}
where $Q$ is $\sigma_\omega$-semistable and $\phi_\omega (L)=\phi_\omega (Q)$.    Let $Q^i$ denote $H^i(Q)$ for $i=-1, 0$.

Note that $Q^0$ is a torsion sheaf by \cite[Lemma 4.1]{ARCARA20161655}. Then the canonical map $Q \to Q^0$ is an $\Ac_\omega$-surjection, and so the $\sigma_\omega$-semistability of $Q$ implies $\phi_\omega (Q) \leq \phi_\omega (Q^0)$, i.e.\
\begin{equation}
  \frac{\ch_2(Q) - \tfrac{\omega^2}{2}\ch_0(Q)}{\omega\ch_1(Q)} \leq \frac{\ch_2(Q^0)}{\omega \ch_1(Q^0)}.
\end{equation}
On the other hand, from the equality $\phi_\omega (L) = \phi_\omega (Q)$ we have
\begin{equation}\label{eq:AG50-99-2}
\frac{\alpha^2 - \omega^2}{2\omega \alpha} = \frac{\ch_2(L)-\tfrac{\omega^2}{2}\ch_0(L)}{\omega \ch_1(L)} = \frac{\ch_2(Q)-\tfrac{\omega^2}{2}\ch_0(Q)}{\omega \ch_1(Q)}.
\end{equation}
We now give a lower bound for $\ch_2(Q)=\ch_2(Q^0)-\ch_2(Q^{-1})$.

From \cite[Lemma 4.1]{ARCARA20161655} again, we know that in the long exact sequence of sheaves for \eqref{eq:AG50-98-1}, the kernel of the surjective morphism of sheaves $L \to Q^0$  is of the form $L \otimes I_Z(-C)$ for some effective curve $C$ and $0$-dimensional subscheme $Z \subset X$, so that we have  a short exact sequence of sheaves
\[
0 \to L \otimes I_Z(-C) \to L \to Q^0 \to 0
\]
where $I_Z$ is the ideal sheaf of some $0$-dimensional subscheme $Z$ of $X$.  Moreover, since \eqref{eq:AG50-98-1} is a destabilising sequence, we know $C$ is a negative curve from \cite[Proposition 5.5]{ARCARA20161655}.  Let $n$ denote the length of $Z$.  Then $\ch(I_Z)=(1,0,-n)$ and
\begin{equation}\label{eq:AG50-98-2}
\ch (Q^0)=\ch(L)-\ch (L \otimes I_Z(-C)) = ( 0, C+\alpha, -\tfrac{C^2}{2}+n+\alpha C ).
\end{equation}

As for $Q^{-1}$, let $J_1, \cdots, J_m$ denote the $\mu_\omega$-HN factors of $Q^{-1}$ with decreasing slopes.  Then for each $i$ we have $0 \leq \omega \ch_1(J_i[1]) < \omega \alpha$.  Also, since $J_i$ is $\mu_\omega$-semistable, the Bogomolov-Gieseker inequality gives
\[
  \ch_2 (J_i) \leq \frac{ \ch_1(J_i)^2}{2\ch_0(J_i)}
\]
while the Hodge Index Theorem gives
\[
  \ch_1(J_i)^2 \leq \frac{ (\omega \ch_1(J_i))^2}{\omega^2}.
\]
Putting all these together gives
\[
  \ch_2(J_i) < \frac{ (\omega \alpha)^2}{2 \omega^2 \ch_0(J_i)}
\]
for each $i$.  Consequently,
\begin{equation}\label{eq:AG50-99-1}
  \ch_2(Q^{-1}) = \sum_{i=1}^m \ch_2 (J_i) <(-s) \frac{ (\omega\alpha)^2}{2 \omega^2}
\end{equation}
where $s$ denotes the rank of $Q$ (which equals $-\ch_0(Q^{-1})$) and so
\begin{align}
   \ch_2 (Q) - \tfrac{\omega^2}{2}\ch_0(Q) & >  \ch_2 (Q^0) + s \frac{ (\omega \alpha)^2}{2 \omega^2} - s \frac{\omega^2}{2} \notag \\
   &= \ch_2(Q^0) + \frac{s}{2} \left( \frac{ (\omega \alpha)^2}{ \omega^2} - \omega^2 \right). \label{eq:AG50-99-3}
\end{align}

Since  \eqref{eq:ineq1} holds, from \eqref{eq:AG50-99-2} we see that the right-hand side of  \eqref{eq:AG50-99-3} is non-positive.  Then, since $0 < \omega C \leq \omega \ch_1(Q)$, it follows that
\[
  \frac{ \ch_2 (Q) - \tfrac{\omega^2}{2}\ch_0(Q)}{\omega \ch_1(Q)} >  \frac{\ch_2(Q^0) + \frac{s}{2} \left( \frac{ (\omega \alpha)^2}{ \omega^2} - \omega^2 \right)}{\omega C}.
\]
Rearranging, we obtain
\begin{equation}\label{eq:AG50-99-5}
\left( \frac{\omega^2-\alpha^2}{2\omega \alpha} \omega + \alpha\right)C < \frac{C^2}{2} -n + \frac{s}{2}\left( \omega^2 - \frac{ (\omega \alpha)^2}{\omega^2}\right).
\end{equation}
Since $C$ is a negative, effective curve and $n \geq 0$, when \eqref{eq:ineq2} holds and the twisted ampleness criterion holds, the left-hand side of \eqref{eq:AG50-99-5} is positive while the right-hand side of \eqref{eq:AG50-99-5} is negative, which is a contradiction.  This shows that in the region of the ample cone given by \eqref{eq:ineq00} and \eqref{eq:ineq2-b}, there cannot be any Bridgeland walls for the object $L$.  

Note that if $\wt{\omega}$ is any ample class satisfying  \eqref{eq:ineq00} and \eqref{eq:ineq2-b}, then $t\wt{\omega}$ also satisfies \eqref{eq:ineq00} and \eqref{eq:ineq2-b} for any $t\geq 1$.  From \cite[Lemma 4.8]{LM2}, we know that   for any ample class $\wt{\omega}$, there exists $t_0>1$ depending on $L$ and $\wt{\omega}$ such that $L$ is $\sigma_{t\wt{\omega}}$-stable for all $t>t_0$, it follows that $L$ must be $\sigma_\omega$-stable whenever both  \eqref{eq:ineq00} and \eqref{eq:ineq2-b} are satisfied.
\end{proof}

\section{Answer to the question for $D_\alpha <0$}\label{sec:Dalphanegative}

In Theorem \ref{thm:main1}, we identified a region in the ample cone where, if $L$ is twisted ample with repsect to $\omega$, then $L$ is Bridgeland stable.  In this section, we restrict to elliptic surfaces and begin  using autoequivalences of the derived category to extend this region of stability.

\begin{lem}\label{lem:AG50-132-1}
Let $p : X \to Y$ be an elliptic surface with a section $\Theta$.  Then a divisor of the form $\Theta + af$ where $a \in \mathbb{R}_{>0}$ is ample if and only if $a > -\Theta^2$.
\end{lem}

\begin{proof}
By the Nakai-Moishsezon criterion, $\Theta + af$ is ample if and only if $(\Theta + af)^2=2a-e>0$ and $(\Theta + af)C>0$ for every irreducible curve $C \subset X$.

Suppose $C$ is an irreducible curve on $X$ that is distinct from $\Theta$.  Then $\Theta C\geq 0$ and $fC \geq 0$.  If $fC=0$, then $C$ is a vertical divisor, in which case $\Theta C>0$ and $(\Theta +af)C>0$.  If $fC>0$, then $(\Theta +af)C>0$ as well.

Now suppose  $C=\Theta$.  Then $(\Theta + af)C=a+\Theta^2$.  Since $a>0$, it now follows that $\Theta + af$ is ample if and only if $a>-\Theta^2$.
\end{proof}


The following result is a special case of \cite[Lemma 4.2(d)]{Martinez2017} (see also \cite[Lemma 6.2]{ARCARA20161655}):

\begin{lem}\label{lem:ArcMil-1}
Let $X$ be a smooth projective surface and $\omega$ an ample divisor on $X$.  If $L$ is a line bundle on $X$ such that $\omega \ch_1(L)>0$, then $L$ is $\sigma_\omega$-stable if and only if $L^\vee [1]$ is $\sigma_\omega$-stable.
\end{lem}


\begin{lem}\label{lem:lbfdonetr}
Let $p : X \to Y$ be a Weierstra{\ss} elliptic surface, and $\Phi$ the autoequivalence of $D^b(X)$ as in \ref{para:WESauto}.  
\begin{itemize}
\item[(i)] If $L$ is a line bundle on $X$ satisfying $f\ch_1(L)=1$, then the transform $\Phi L$ is also a line bundle sitting at degree $0$ in $D^b(X)$. 
\item[(ii)] If $L$ is a line bundle on $X$ satisfying $f\ch_1(L) \leq 0$, then $L$ is $\Phi$-WIT$_1$.  In particular, $\Phi \OO_X$ is a torsion sheaf.
\item[(iii)] If $L$ is a line bundle on $X$ satisfying $f\ch_1(L)=-1$, then $\Phi L$ is  a line bundle sitting at degree $1$ in $D^b(X)$.
\end{itemize}
\end{lem}



\begin{proof}
Let $L$ be a line bundle on $X$.

(i) Assume  $f\ch_1(L)=1$.  For any closed point $y \in Y$, the restriction $L|_{p^{-1}(y)}$ is a line bundle of degree $1$ on the fiber $p^{-1}(y)$, and so is $\Phi_y$-WIT$_0$ by \cite[Proposition 6.38]{FMNT}.  Then, by \cite[Lemma 3.6]{Lo11}, it follows that $L$ itself is $\Phi$-WIT$_0$.  By \cite[Corollary 5.4]{Lo7}, the transform $\Phi L$ is a locally free sheaf.  Since $\ch_0(\Phi L)=f\ch_1(L)=1$, the transform of $L$ is a line bundle.

(ii) Assume $f\ch_1(L) \leq 0$.   For a general closed point $y \in Y$, we know by  \cite[Corollary 3.29]{FMNT} that  the restriction $\OO_X|_{p^{-1}(y)}$  is $\Phi_y$-WIT$_1$; since $\OO_X$ has no fiber subsheaves, it follows by \cite[Lemma 3.18]{Lo11} that $\OO_X$ itself is $\Phi$-WIT$_1$.  The second part of (ii)  follows easily from the formula for the cohomological Fourier-Mukai transform for $\Phi$.

(iii) In this case, $L^\vee$ is a line bundle satisfying $f\ch_1(L^\vee)=1$, and so $\Phi (L^\vee)$ is a line bundle by part (i).  From \cite[Theorem 2.3]{LZ3}, we have
\[
  (\Phi L)^\vee \cong \iota^\ast ( \Phi (L^\vee)) \otimes p^\ast M [1]
\]
for some involution $\iota : X \to X$ and some line bundle $M$ on $Y$.  This implies that $(\Phi L)^\vee$ is a line bundle sitting at degree $-1$ in $D^b(X)$, and the claim follows.
\end{proof}


\paragraph \label{para:spcase1} Suppose $p : X \to Y$ is a Weierstra{\ss} elliptic surface, and $L$ is a line bundle on $X$ satisfying $fc_1(L)=1$.  Then by Lemma \ref{lem:lbfdonetr}, $L$ is $\Phi$-WIT$_0$ and the transform $\wh{L} = \Phi L$ is again a line bundle.  In addition, if we write $c=\Theta c_1(L)$ and $s=\ch_2(L)$, then
\[
  c_1(\wh{L}) = -c_1(L)+(c+\tfrac{e}{2}+s)f
\]
from the cohomological FMT formulas in \ref{para:cohomFMT}.  Now suppose $\omega, B, \omega', B' \in \langle \Theta, f \rangle_{\mathbb{R}}$ are divisors satisfying the relations \eqref{eq:AG50-118-1}, where $\omega, \omega'$ are ample and $B=0$ (so $l=0$ and $B'=-\tfrac{e}{2}f$).  Then  we have
\begin{align}
 \text{$L$ is $\sigma_{\omega,0}$-stable } &\Leftrightarrow \text{ $\wh{L}[1]$ is $\sigma_{\omega',B'}$-stable} \text{ by \eqref{eq:AG52-21-2}} \notag\\
 &\Leftrightarrow \text{ $(\wh{L}\otimes \OO_X(\tfrac{e}{2}f))[1]$ is $\sigma_{\omega',0}$-stable} \text{ by \cite[(13.4.1)]{Lo20}}. \label{eq:Lstabequiv}
\end{align}

Let us write $\wt{L} = \wh{L} \otimes \OO_X (\tfrac{e}{2}f)$.  Then
\[
  c_1(\wt{L}) = c_1(\wh{L}) + \tfrac{e}{2}f = -c_1(L)+(c+s+e)f.
\]
Let us also write $\alpha = c_1(L)$ and $\wt{\alpha} = c_1(\wt{L})$.  If we  specialise to the case where $\alpha$ is of the form
\[
  \alpha = \Theta + (D_\alpha + e)f
\]
for some $D_\alpha \in \mathbb{R}$, then $\Theta c_1(L) =D_\alpha, \ch_2(L)=V_\alpha=D_\alpha + \tfrac{e}{2}$ and
\[
  \wt{\alpha} = -\Theta + (D_\alpha + \tfrac{e}{2}) f.
\]

\paragraph \label{para:Dapos-comp1} Assume the setting of \ref{para:spcase1}.  To make our computations more concrete, let us assume that $\omega \in \langle \Theta, f \rangle_{\mathbb{R}}$ so  it can be written in the form
\[
 \omega = R_\omega (\Theta + (D_\omega + e)f)
\]
for some $R_\omega, D_\omega  \in \mathbb{R}_{>0}$. Then 
\[
  \omega \alpha = R_\omega (D_\omega + D_\alpha +e), \text{\quad} R_\omega^2 = \frac{V_\omega}{D_\omega + \tfrac{e}{2}},
\]
and the inequality \eqref{eq:ineq2}, which is
\[
\omega^2 - \frac{(\omega\alpha)^2}{\omega^2}\geq 0,
\]
can be rewritten as
\begin{equation}\label{eq:ineq2RVD}
2V_\omega - \frac{ (D_\omega + D_\alpha + e)^2}{2(D_\omega + \tfrac{e}{2})} \geq 0.
\end{equation}


Next, recall that the twisted ampleness criterion for $\alpha$ with respect to $\omega$ is equivalent to
\[
  \frac{\omega^2 - \alpha^2}{2\omega \alpha}\omega +\alpha
\]
being ample.  Since
\begin{equation}\label{eq:algobs}
\frac{\omega^2 - \alpha^2}{2\omega \alpha}\omega +\alpha = \left( \frac{V_\omega - D_\alpha - \tfrac{e}{2}}{D_\omega + D_\alpha + e} +1\right)\Theta + \left( \frac{V_\omega - D_\alpha - \tfrac{e}{2}}{D_\omega + D_\alpha + e}(D_\omega + e) + (D_\alpha + e) \right)f,
\end{equation}
the twisted ampleness criterion for $\omega, \alpha$ corresponds to
\begin{equation}\label{eq:ineq3RVD}
\left( \frac{V_\omega - D_\alpha - \tfrac{e}{2}}{D_\omega + D_\alpha + e}(D_\omega + e) + (D_\alpha + e) \right) > e\left( \frac{V_\omega - D_\alpha - \tfrac{e}{2}}{D_\omega + D_\alpha + e} +1\right)
\end{equation}
by Lemma \ref{lem:AG50-132-1}. Under our assumption $\omega \alpha >0$, which is equivalent to $D_\omega+ D_\alpha + e>0$, the above inequality can be rewritten as
\begin{equation}\label{eq:ineq3RVD-v2}
  D_\omega (V_\omega - \tfrac{e}{2}) + D_\alpha (D_\alpha + e) > 0.
\end{equation}

\paragraph It can also be useful to  consider the ``transforms'' of the conditions in Theorem \ref{thm:main1} that ensure the stability of a line bundle.  Let us continue with the setup in \ref{para:Dapos-comp1}. Note that
\[
  \omega' \wt{\alpha} = R_\omega'(\Theta + (D_{\omega'}+e)f)(-\Theta + (D_\alpha+\tfrac{e}{2}) f) = R_{\omega'}(D_\alpha + \tfrac{e}{2} - V_\omega),
\]
and so $\omega' \wt{\alpha}>0$ if and only if $V_\omega < D_\alpha + \tfrac{e}{2}$. Writing $(\wt{\alpha})^\vee = c_1 ( (\wt{L})^\vee)$, we have $\omega' (\wt{\alpha})^\vee >0$ if and only if
\begin{equation}\label{eq:w'a'pos}
  V_\omega > D_\alpha+ \tfrac{e}{2}.
\end{equation}
Since $V_\omega$ is positive, this suggests that we should apply Theorem \ref{thm:main1} to $\omega'$ and $(\wt{L})^\vee$, which means we need to consider the analogues of  \eqref{eq:ineq2} and the twisted ampleness criterion for $\omega'$ and $(\wt{L})^\vee$.  

The analogue of \eqref{eq:ineq2} for $\omega', (\wt{L})^\vee$ is
\[
  (\omega')^2 - \frac{ (\omega' \wt{\alpha})^2}{ (\omega')^2} \geq 0
\]
which is equivalent to
\begin{equation}\label{eq:ineq2p}
4D_\omega \geq  \frac{ (V_\omega - D_\alpha- \tfrac{e}{2})^2}{(V_\omega + \tfrac{e}{2})} .
\end{equation}
Lastly, the twisted ampleness criterion for $\omega', (\wt{L})^\vee$ corresponds to the ampleness of
\[
  -\frac{ (\omega')^2-(\wt{\alpha})^2}{2\omega' \wt{\alpha}} \omega' - \wt{\alpha}
  = \left( \frac{D_\omega + D_\alpha + e}{V_\omega - D_\alpha- \tfrac{e}{2}} + 1\right)\Theta + \left( \frac{D_\omega + D_\alpha + e}{V_\omega - D_\alpha- \tfrac{e}{2}} (V_\omega + e) - (D_\alpha+ \tfrac{e}{2})\right)f,
\]
which in turn corresponds to
\begin{equation}\label{eq:algobstrans}
\left( \frac{D_\omega + D_\alpha + e}{V_\omega - D_\alpha- \tfrac{e}{2}} (V_\omega + e) - (D_\alpha+ \tfrac{e}{2})\right) > e\left( \frac{D_\omega + D_\alpha + e}{V_\omega - D_\alpha- \tfrac{e}{2}} + 1\right).
\end{equation}

\begin{rem}
In the computations above, switching from $\omega, L$ to $\omega', (\wt{L})^\vee$ amounts to  making the  change of variable
\begin{align*}
  D_\omega &\mapsto V_\omega \\
  V_\omega &\mapsto D_\omega \\
  D_\alpha &\mapsto -D_\alpha -\tfrac{3e}{2}
\end{align*}
in the sense that $\omega \alpha >0$ is transformed to $\omega'(\wt{\alpha})^\vee>0$, $\omega^2 - \frac{(\omega \alpha)^2}{\omega^2} \geq 0$ is transformed to $(\omega')^2 - \frac{ (\omega' \wt{\alpha})^2}{ (\omega')^2} \geq 0$, and $\frac{\omega^2-\alpha^2}{2\omega \alpha}\omega + \alpha$ is transformed to $ -\frac{ (\omega')^2-(\wt{\alpha})^2}{2\omega' \wt{\alpha}} \omega' - \wt{\alpha}$.
\end{rem}

\begin{thm}\label{thm:conj-Daotcase}
Let $p : X \to Y$ be a Weierstra{\ss} elliptic surface.  Suppose $e=2$, that  $\alpha = \Theta +f$ or $\Theta$ (i.e.\ $D_\alpha = -1$ or $-2$) and $\omega \in \langle \Theta, f \rangle_{\mathbb{R}}$ is ample.  Then for any line bundle $L$ with $c_1(L)=\alpha$,  the twisted ampleness criterion for $\omega, \alpha$ implies the $\sigma_\omega$-stability of  $L$.
\end{thm}

\begin{proof}
Assume that the twisted ampleness criterion holds for $\omega, \alpha$, i.e.\ \eqref{eq:ineq3RVD} holds.  Note that $D_\omega >0$ by Lemma \ref{lem:AG50-132-1}.

\textbf{Case 1: $e=2, D_\alpha=-1$.}  In this case, $\omega \alpha =R_\omega (D_\omega +1)$  is positive.    The inequality $\omega^2 - \frac{(\omega \alpha)^2}{\omega^2} \geq 0$ reduces to
\begin{equation}\label{eq:egA-eq3}
4V_\omega \geq D_\omega + 1,
\end{equation}
while the twisted ampleness criterion for $\omega, \alpha$, equation \eqref{eq:ineq3RVD}, reduces to
\begin{equation}\label{eq:eg1-3}
  V_\omega  > 1 + \frac{1}{D_\omega}.
\end{equation}
If \eqref{eq:egA-eq3} holds, then $L$ is $\sigma_\omega$-stable by Theorem \ref{thm:main1}.  So let us assume that \eqref{eq:egA-eq3} fails from here on, i.e.\ we assume
\begin{equation}\label{eq:ineqfail-1}
 4V_\omega < D_\omega +1.
\end{equation}
We will show that $L$ is still $\sigma_\omega$-stable under this assumption,

Since $\omega' (\wt{\alpha})^\vee = R_{\omega'}V_\omega  >0$, we will consider $\omega'$ and the line bundle $(\wt{L})^\vee$ instead.  The inequality $(\omega')^2 - \frac{(\omega' \wt{\alpha})^2}{(\omega')^2} \geq 0$  reduces to
\begin{equation}\label{eq:egA-eq3p}
4D_\omega \geq \frac{V_\omega^2}{V_\omega + 1},
\end{equation}
while the twisted ampleness criterion for $\omega', (\wt{\alpha})^\vee$ reduces to  $D_\omega >1$.

From \eqref{eq:eg1-3} we see $V_\omega >1$, and so by \eqref{eq:ineqfail-1} we have $D_\omega +1>4$, i.e.\ $D_\omega >3$.  This implies  the twisted ampleness criterion for $\omega', (\wt{\alpha})^\vee$.  Lastly, \eqref{eq:ineqfail-1} also gives $4D_\omega > 16V_\omega - 4$.  Using $V_\omega >1$, it is easy to check that
\[
  16V_\omega -4 \geq \frac{V_\omega^2}{V_\omega + 1}.
\]
The inequality \eqref{eq:egA-eq3p} then follows.  We can now apply  Theorem \ref{thm:main1} to the ample divisor $\omega'$ and the line bundle $(\wt{L})^\vee$ to conclude that $(\wt{L})^\vee$ is $\sigma_{\omega'}$-stable.  By Lemma \ref{lem:ArcMil-1}, this means $\wt{L}[1]$ is $\sigma_{\omega'}$-stable, which in turn is equivalent to the $\sigma_\omega$-stability of $L$ by \eqref{eq:Lstabequiv}.

\textbf{Case 2: $e=2, D_\alpha = -2$.} In this case, $\omega \alpha = R_\omega D_\omega$ is positive.  Also,  \eqref{eq:ineq2RVD} reduces to $4V_\omega (D_\omega + 1) \geq D_\omega^2$.  Let us assume that the twisted ampleness criterion for $\omega, \alpha$ holds; from  \eqref{eq:ineq3RVD}, this means
\[
  \frac{V_\omega + 1}{D_\omega}(D_\omega + 2) > 2\left( \frac{V_\omega +1}{D_\omega} + 1\right)
\]
which is equivalent to $V_\omega >1$.  As in Case 1, it suffices to consider the scenario where \eqref{eq:ineq2RVD} fails, so we assume
\begin{equation}\label{eq:ineq-fail-2}
  4V_\omega (D_\omega + 1) < D_\omega ^2
\end{equation}
from now on.  Note that this implies $D_\omega^2 > 4V_\omega D_\omega$, i.e.\ $D_\omega > 4V_\omega$.

Since $\omega' (\wt{\alpha})^\vee =R_{\omega'}(V_\omega+1)>0$ in this case, we will prove that $(\wt{L})^\vee$ is $\sigma_{\omega'}$-stable.  As in Case 1, this will follow once we establish the inequalities  \eqref{eq:ineq2p} and  \eqref{eq:algobstrans}.  In this case, the inequality \eqref{eq:algobstrans} reduces to
\[
  \frac{D_\omega }{V_\omega + 1}(V_\omega + 2)+1 > 2\left(\frac{D_\omega}{V_\omega + 1} + 1\right)
\]
which is equivalent to $D_\omega >1+\tfrac{1}{V_\omega}$.  Since $D_\omega >4V_\omega$ and $V_\omega >1$ from the previous paragraph, \eqref{eq:algobstrans} certainly holds.   Also,
\eqref{eq:ineq2p} reads
\[
4D_\omega  \geq (V_\omega + 1)
\]
in this case.  Since $V_\omega >1$, we have
\[
  4D_\omega > 16V_\omega >15V_\omega + 1>V_\omega + 1
\]
and so \eqref{eq:ineq2p} holds.  Hence $(\wt{L})^\vee$ is $\sigma_{\omega'}$-stable, and the same argument as in Case 1 shows $L$ itself is $\sigma_\omega$-stable.
\end{proof}

On a Weierstra{\ss} elliptic surface $X$ with $e=2$, in order to answer  Question \ref{que:main-1} beyond the case of $D_\alpha = -1, -2$ as in Theorem \ref{thm:conj-Daotcase}, we need to use autoequivalences of the derived category in a more systematic way.  We begin with an observation that the stability of a line bundle is preserved as the volume of the ample divisor increases. 

\paragraph Let $X$ be a projective surface. Let $\omega$ be an ample divisor in $\NS_\mathbb{R}(X)$ such that $\omega^2=1$. Let $H$ be an arbitrary divisor in $\NS_\mathbb{R}(X)$ such that $H\cdot\omega=0$ and $H^2=-1$. 
Let $L$ be a line bundle on $X$. Let $E\in D^b(X)$, we denote the Chern character of $E$ by $(r, c, d)$, where $c=E_\omega \omega+E_H H+\alpha$ for $\alpha\in\langle \omega, H\rangle^\perp$.
Similarly, we denote the Chern character of $L$ by $(1, l, e)$, where $l=l_\omega \omega+l_H H+\beta$ for $\beta\in \langle \omega, H\rangle^\perp$. We consider stability conditions of the form $\sigma_{x\omega, B}$, where $B=y\omega+zH$. 
Then the potential wall $W(E, L)$ defined by  \eqref{Equ:PotentialWall} is a quadratic surface in $\mathbb{R}^{>0}\times\mathbb{R}^2$:
\begin{equation}
\label{walle}
\begin{split}
\frac{1}{2}x^2(rl_\omega-E_\omega)+\frac{1}{2}y^2(rl_\omega-E_\omega)+\frac{1}{2}z^2(rl_\omega-E_\omega)-yz(rl_H-E_H)\\
+y(d-er)-z(E_Hl_\omega-l_HE_\omega)-(dl_\omega-eE_\omega)=0.
\end{split}
\end{equation}
For any $z$, the intersection $W(E, L)_{z}$ is either empty or a semicircle.

\begin{prop}[Going-up Lemma, Version 1]
	\label{Prop:StabvarVol}
	If $L$ is stable at $\sigma_{x_0\omega, B}$ at some $x_0$, $B=y_B\omega+z_BH$, then $L$ is stable for $\sigma_{x\omega, B}$ for all $x>x_0$.
\end{prop}
\begin{proof}
 Assume that $L$ is unstable at $\sigma_{x_1\omega, B}$ for some $x_1>x_0$.  Then there exists a short exact sequence in $\Coh^{y_B}$:
 \begin{equation}
 \label{Equ:GoingupLem}
0\rightarrow E\rightarrow L\rightarrow Q\rightarrow 0
\end{equation}
 such that $E$ destabilizes $L$ at $\sigma_{x_1\omega, B}$. Then in the plane $z=z_B$, the point $(x_1, y_B, z_B)$ is inside the region bounded by the semicircle $W(E, L)_{z_B}$ and the $z$-axis. Then the points $(x, y_B, z_B)$ are all in the region bounded by the semicircle $W(E, L)_{z_B}$ and the $z$-axis for $0<x<x_1$. Since the short exact sequence \ref{Equ:GoingupLem} is a short exact sequence in $\Coh^{x\omega, B}$ for all $x<x_1$, we have $$\phi_{x, y_B, z_B}(E)>\phi_{x, y_B, z_B}(L)$$ for all $0<x<x_1$, contradicting $L$ is stable at $\sigma_{x_0\omega, B}$. 
\end{proof}

The following lemma is a  variation of Proposition \ref{Prop:StabvarVol} on a smooth projective surface.  It says that for a twisted $\omega$-Gieseker stable sheaf $E$ (such as a $\mu_\omega$-stable torsion-free sheaf), if there exists a mini-wall on  the ray corresponding to $\{t\omega\}_{t>0}$, then $E$ must be Bridgeland stable on the far side of this wall, while unstable on the near   side (the side closer to the origin) of this wall.

\begin{lem}[Going-up Lemma, Version 2]\label{lem:uniqueminiwall}
Let $X$ be a smooth projective surface.  Suppose $H, B$ are $\mathbb{R}$-divisors on $X$ with $H$ ample. Consider the 1-parameter family of Bridgeland stability conditions
\[
  \sigma_{H,B,t} = (Z_{H,B,t}, \Coh^{H,B})
\]
where
\[
  Z_{H,B,t} = -\ch_2^B + t\ch_0^B + iH\ch_1^B.
\]
Suppose $E$ is a torsion-free, $B$-twisted $H$-Gieseker stable sheaf on $X$, and $t=t_0$ is the outermost mini-wall on the ray $\{\sigma_{H,B,t}\}_{t>0}$ with respect to the Chern character $\ch(E)$.  Then $E$ is stable (resp.\ unstable) with respect to  $\sigma_{H,B,t}$ for all $t>t_0$ (resp.\ $0<t<t_0$).
\end{lem}

Even though the central charge of $\sigma_{H, B, t}$ is not of the standard form $Z_{\omega, B}$ as in \eqref{Equ:CenCharNoTodd}, $\sigma_{H, B, t}$ is still a Bridgeland stability condition for all $t>0$  \cite[Lemma 3.3]{LM2}.

\begin{proof}
Consider any $\Coh^{H,B}$-short exact sequence of the form
\begin{equation*}
0 \to A \to E \to Q \to 0.
\end{equation*}
Then $A$ is a sheaf sitting at degree $0$, and we have the long exact sequence of sheaves
\[
0 \to H^{-1}(Q) \to A \overset{\beta}{\to} E \to H^0(Q) \to 0.
\]
The rank of $A$ must be positive, for otherwise $H^{-1}(Q)$ would be zero, and $A$ would be a torsion subsheaf of $E$, contradicting the assumption that $E$ is torsion-free.

For any object $M \in D^b(X)$ such that $\ch_0(M), H\ch_1^B(M)\neq 0$, write
\[
  g_M(t) := -\frac{\Re Z_{H,B,t}(M)}{\Im Z_{H,B,t}(M)}=\frac{\ch_2^B(M)-t\ch_0^B(M)}{H\ch_1^B(M)}.
\]
Now choose $A$ to be a destabilising subobject of $E$ with respect to $\sigma_{H,B,t_0}$ for some $t_0>0$, so that $g_A(t_0)=g_E(t_0)$.  The twisted Gieseker stability assumption on $E$  ensures that  $E$ is $\sigma_{H,B,t}$-semistable for $t \gg 0$ \cite[Remark 4.3]{LM2}, which  means $g_A(t) \leq g_E(t)$ for $t \gg 0$, which then implies
\[
  -\frac{1}{\mu_{H,B}(A)} \leq -\frac{1}{\mu_{B,A}(E)}.
\]
Since $\mu_{H,B}(A), \mu_{H,B}(E)$ are both positive, this corresponds to $\mu_{H,B}(A) \leq \mu_{H,B}(E)$.  If we are in the case $\mu_{H,B}(A) = \mu_{H,B}(E)$, then since $g_A(t_0)=g_E(t_0)$, $E$ is a strictly $B$-twisted $H$-Gieseker semistable sheaf by \cite[Lemma 4.6]{LM2}, contradicting the assumption that $E$ is $B$-twisted $H$-Gieseker stable.  Hence we must have
\begin{equation}\label{eq:AG51-168-1}
\mu_{H,B}(A) < \mu_{H,B}(E).
\end{equation}

On the other hand, \eqref{eq:AG51-168-1} means
  \[
   \frac{d}{dt}g_A(t) < \frac{d}{dt} g_E(t) <0
  \]
  Since $g_A(t), g_E(t)$ are both linear functions in $t$, it follows that $g_A(t)> g_E(t)$ for all $t \in (0,t_0)$, i.e.\ $E$ is unstable   with respect to $\sigma_{H,B,t}$ for all $t \in (0,t_0)$, with $A$  destabilising $E$ in $\Coh^{H,B}$.  (Here, we are using the fact that $A$ is a subobject of $E$ in the heart $\Coh^{H,B}$ for all values of $t>0$, since the construction of $\Coh^{H,B}$ is independent of $t$.)
  
  Since $t=t_0$ is assumed to be the outer-most mini-wall, the above argument also implies that $E$ cannot be strictly semistable with respect to $\sigma_{H,B,t}$ for any $t>t_0$, i.e.\ $E$ must be stable with respect to $\sigma_{H,B,t}$ for all $t>t_0$.
\end{proof}

\begin{lem}\label{lem:Yun-goingup}
Let $p : X \to Y$ be a Weierstra{\ss} elliptic surface, and let $L$ be a line bundle on $X$ such that $c_1(L) \in \langle \Theta , f\rangle_{\mathbb{R}}$.  Whenever $L$ lies in $\Coh^{\omega, B}$ and is $\sigma_{\omega,B}$-stable  for some $\mathbb{R}$-divisors $\omega, B \in \langle \Theta , f\rangle_{\mathbb{R}}$ where $\omega$ is ample, it follows that $L$ is also $\sigma_{\omega^\ast, B}$-stable for any ample class $\omega^\ast \in \langle \Theta, f \rangle_{\mathbb{R}}$ such that
\[
D_{\omega^\ast}=D_\omega \text{\quad and \quad} V_{\omega^\ast}>V_{\omega}.
\]
\end{lem}


\begin{proof}
Let us write  $\omega = R_\omega (\Theta + (D_\omega + e)f)$,  $H = \Theta + (D_\omega + e)f$, and consider the element  $\wt{g}=(T,f) \in \wt{\mathrm{GL}}^+\! (2,\mathbb{R})$ where $T=\begin{pmatrix} 1 & 0 \\ 0 & R_\omega \end{pmatrix}$  and $f$ fixes all values in $\tfrac{1}{2}\mathbb{Z}$.  Then it is easy to check that $\sigma_{\omega, B}=\sigma_{H,B,V_\omega}\cdot \wt{g}$ (where we use the notation in  Lemma \ref{lem:uniqueminiwall}).

Since $L$ is $\mu_{\omega, B}$-stable and hence $\mu_{H,B}$-stable, it is $B$-twisted $H$-Gieseker stable.  Our assumption that $L$ is $\sigma_{\omega,B}$-stable then implies  $L$ is $\sigma_{H,B,V_\omega}$-stable.  Then by Lemma \ref{lem:uniqueminiwall}, $L$ must be $\sigma_{H,B,t}$-stable for all $t>V_\omega$.  Up to a  $\wt{\mathrm{GL}}^+\!(2,\mathbb{R})$-action, it follows that  that $L$ is $\sigma_{\omega^\ast,B}$-stable for any ample class $\omega^\ast \in \langle \Theta, f \rangle_{\mathbb{R}}$ such that $V_{\omega^\ast}>V_\omega$ and $D_{\omega^\ast} = D_\omega$.
\end{proof}

Informally, Lemma \ref{lem:Yun-goingup} says the following: if a line bundle $L$ is stable with respect to a point $(D_{\omega_0}, V_{\omega_0})$ on the $(D_\omega, V_\omega)$-plane, then we can always ``move up'' from the point $(D_{\omega_0}, V_{\omega_0})$ - meaning keeping the $D_\omega$-coordinate fixed while increasing the $V_\omega$-coordinate - and maintain the stability of $L$.  The following lemma says, that via the Fourier-Mukai transform $\Phi$ which essentially interchanges the $D_\omega$= and $V_\omega$=coordinates, we can also ``move to the right'' on the $(D_\omega, V_\omega)$-plane and preserve the stability of $L$.

\begin{lem}\label{lem:Yun-goingup-transf}
Let $p : X \to Y$ be a Weierstra{\ss} elliptic surface.  Suppose $\alpha = \Theta + (D_\alpha +e)f$ for some integer  $D_\alpha \leq -\tfrac{e}{2}$, and $L$ is a line bundle with $c_1(L)=\alpha$.   Whenever $L$ lies in $\Coh^{\omega,0}$ and is $\sigma_\omega$-stable for some ample class $\omega \in\langle \Theta, f \rangle_{\mathbb{R}}$, it follows that $L$ is $\sigma_{\omega^\ast}$-stable for any ample class $\omega^\ast$ such that 
\[
D_{\omega^\ast}>D_\omega \text{\quad and \quad} V_{\omega^\ast}=V_{\omega}.
\]
\end{lem}

\begin{proof}
Suppose $\omega$ is an ample class on $X$ such that $L$ is $\sigma_\omega$-stable.  By \eqref{eq:Lstabequiv}, we have that $\wt{L}[1]$ is $\sigma_{\omega'}$-stable where $\wt{L} = (\Phi L) \otimes \OO_X(\tfrac{e}{2}f)$ and $(D_{\omega'}, V_{\omega'})=(V_\omega, D_\omega)$.  Let $\wt{\alpha} = c_1(\wt{L})$.  Note that
\[
 \omega' \wt{\alpha} = R_{\omega'}(D_\alpha + \tfrac{e}{2} - V_\omega).
\]
Under our assumption $D_\alpha \leq -\tfrac{e}{2} \leq 0$, and since $V_\omega >0$, we  have
\[
\omega' c_1((\wt{L})^\vee) = -\omega'\wt{\alpha} >0
\]
meaning $(\wt{L})^\vee$ lies in the category $\Coh^{\omega',0}$; we also know $(\wt{L})^\vee$ is $\sigma_{\omega'}$-stable by Lemma \ref{lem:ArcMil-1}.  Now by Lemma \ref{lem:Yun-goingup}, it follows that $(\wt{L})^\vee$ is also $\sigma_{\omega''}$-stable for any ample class $\omega''$ such that $D_{\omega''} = D_{\omega'}$ and $V_{\omega''} > V_{\omega'}$.  For any such $\omega''$, we have $\omega'' c_1( (\wt{L})^\vee)>0$ and
\begin{align*}
  (\wt{L})^\vee \text{ is $\sigma_{\omega''}$-stable } &\Leftrightarrow \wt{L}[1] \text{ is $\sigma_{\omega''}$-stable} \text{\quad by Lemma \ref{lem:ArcMil-1}} \\
  &\Leftrightarrow L \text{ is $\sigma_{\omega^\ast}$-stable} \text{\quad by \eqref{eq:Lstabequiv}}
\end{align*}
where $\omega^\ast$ is an ample class with
\begin{align*}
  D_{\omega^\ast} &= V_{\omega''} > V_{\omega'} = D_\omega, \\
  V_{\omega^\ast} &= D_{\omega''} = D_{\omega'} = V_\omega.
\end{align*}
\end{proof}

Combining Lemmas  \ref{lem:Yun-goingup} and \ref{lem:Yun-goingup-transf}, we obtain a result that says that whenever a line bundle $L$ is $\sigma_\omega$-stable for some ample class $\omega$, it must also be $\sigma_{\omega^\ast}$-stable for any $\omega^\ast$ that is farther away from the origin than $\omega$ on the $(D_\omega, V_\omega)$-plane with respect to the taxicab metric.

\begin{cor}\label{cor:goingawaylemma}
Let $p : X \to Y$ be a Weierstra{\ss} elliptic surface.  Suppose $\alpha = \Theta + (D_\alpha +e)f$ for some integer  $D_\alpha \leq -\tfrac{e}{2}$, and $L$ is a line bundle with $c_1(L)=\alpha$.   Suppose $L$ lies in $\Coh^{\omega,0}$ and is $\sigma_\omega$-stable for some ample class $\omega \in\langle \Theta, f \rangle_{\mathbb{R}}$.  Then  $L$ is also $\sigma_{\omega^\ast}$-stable for any ample class $\omega^\ast$ such that 
\[
D_{\omega^\ast} \geq D_\omega \text{\quad and \quad} V_{\omega^\ast} \geq V_{\omega}.
\]  
\end{cor}

\begin{thm}\label{thm:Dalphaneg-main1}
Let $p : X \to Y$ be a Weierstra{\ss} elliptic surface.  Suppose $\alpha = \Theta + (D_\alpha + e)f$ where $D_\alpha < -e$ is an integer.  Then for any ample class $\omega \in \langle \Theta, f \rangle_{\mathbb{R}}$ and any line bundle $L$ on $X$ with $c_1(L)=\alpha$ satisfying $\omega \alpha > 0$, the line bundle $L$ is $\sigma_\omega$-stable.  In particular, we have
\[
\text{ $L$ is twisted ample with respect to $\omega$ } \Rightarrow \text{ $L$ is $\sigma_\omega$-stable}.
\]
\end{thm}


\begin{proof}
Our strategy is as follows: 
\begin{enumerate}
  \item[(i)] We show that $\omega \alpha >0$ defines a half-plane in the $(D_\omega, V_\omega)$-plane, and is itself contained in the half-plane $D_\omega >0$.  Let $P$ denote the intersection of the $D_\omega$-axis with  the straight line $\omega \alpha = 0$.
  \item[(ii)] We show that the intersection of the regions given by $\omega \alpha >0$ and  $\omega^2 - \frac{(\omega \alpha)^2}{\omega^2} >0$ is contained in the half-plane $V_\omega >0$, and its boundary is tangent to the $D_\omega$-axis at the point $P$.
  \item[(iii)] We show that $P$ has an open neighhourbood $U$ contained within the region given by the twisted ampleness criterion, i.e.\ the region given by \eqref{eq:ineq3RVD}.
\end{enumerate}
Once we establish (i) through (iii), given any ample divisor $\omega^\ast$ satisfying $\omega^\ast \alpha >0$, we will be able to find a point $(D'_\omega, V'_\omega)$ such that
\begin{itemize}
\item[(a)] $D'_\omega < D_{\omega^\ast}$ and $V'_\omega < V_{\omega^\ast}$.
\item[(b)] If $\omega'$ denotes the  ample divisor satisfying $(D_{\omega'}, V_{\omega'})=(D'_\omega, V'_\omega)$, then $\omega'$ satisfies $\omega'\alpha>0$, $(\omega')^2 - \frac{(\omega' \alpha)^2}{(\omega')^2} >0$, and the twisted ampleness criterion for $\omega', \alpha$ holds.
\end{itemize}
By Theorem \ref{thm:main1}, condition (b) ensures that $L$ is $\sigma_{\omega'}$-stable.  Then by Corollary \ref{cor:goingawaylemma}, condition (a) ensures that $L$ is also $\sigma_{\omega^\ast}$-stable, thus proving the theorem.

(i) Recall that
\[
\omega \alpha = R_\omega (D_\omega + D_\alpha +e)
\]
and so, under our assumption $D_\alpha + e < 0$, the inequality $\omega \alpha > 0$ is equivalent to
\[
  D_\omega > -(D_\alpha + e)
\]
and the claim follows.  The straight line $\omega \alpha = 0$ intersects the $D_\omega$-axis at the point $P : (D_\omega, V_\omega) = (-(D_\alpha + e), 0)$.

(ii) We can rewrite $\omega^2 - \frac{(\omega \alpha)^2}{\omega^2} \geq 0$ as $(\omega^2)^2 \geq (\omega \alpha)^2$ or equivalently $\omega^2 \geq \omega \alpha$ since $\omega$ is ample and $\omega \alpha>0$ by assumption.  The boundary of this region is given by $\omega^2 = \omega \alpha$, i.e.\
\[
2\sqrt{V_\omega} = \frac{(D_\omega + D_\alpha + e)}{\sqrt{D_\omega + \tfrac{e}{2}}}.
\]
Restricting to the half plane $\omega \alpha \geq 0$, i.e.\ $D_\omega + D_\alpha + e\geq 0$, we see that the above boundary equation can further be rewritten as
\begin{equation}\label{eq:regioncri1boundary}
  4V_\omega = \frac{ (D_\omega + D_\alpha + e)^2}{D_\omega + \tfrac{e}{2}}
\end{equation}

Regarding $V_\omega$ as a function in $D_\omega$ using \eqref{eq:regioncri1boundary}, we obtain
\[
\frac{d}{d D_\omega} V_\omega = \frac{1}{4}\frac{(D_\omega + D_\alpha + e)}{(D_\omega + \frac{e}{2})^2}(D_\omega - D_\alpha),
\]
which equals zero exactly when $D_\omega + D_\alpha + e=0$, i.e.\ $D_\omega = -(D_\alpha + e)$, and is positive (resp.\ negative) when $D_\omega > -(D_\alpha+e)$ (resp.\ $D_\omega < -(D_\alpha + e)$) in the region $D_\omega>0$.  This shows that within the region $\omega \alpha \geq 0$ and $V_\omega \geq 0$, the curve $\omega^2=\omega\alpha$ is tangent to the $D_\omega$-axis at the point $P$.


(iii) Consider the region on the $(D_\omega, V_\omega)$-plane given by the twisted ampleness criterion, i.e.\ \eqref{eq:ineq3RVD-v2}, which has boundary given by
\begin{equation}\label{eq:algobsboundary}
g(D_\omega, V_\omega) := D_\omega (V_\omega - \tfrac{e}{2}) + D_\alpha (D_\alpha + e)=0.
\end{equation}
Since $g$ is a continuous function $\mathbb{R}^2 \to \mathbb{R}$, and
\[
g(-(D_\alpha + e), 0) = (D_\alpha + e)(D_\alpha + \tfrac{e}{2}) > 0
\]
under our assumption that $D_\alpha < -e$, it follows that the point $P$ has an open neighbourhood $U$  that is contained entirely within the region \eqref{eq:ineq3RVD-v2} given by the twisted ampleness criterion.
\end{proof}

\section{Definition of weak stability condition} \label{sec:defwsc}
In Section \ref{sec:Dalphageqzero}, we will answer Question \ref{que:main-1} for the case $D_\alpha\geq 0$. Our strategy is to first study the stability of line bundles with respect to  ``degenerations'' of Bridgeland stability conditions, i.e.\ when specific parameters of Bridgeland stability conditions approach certain limits. For this purpose, we use a generalized notion of weak stability conditions defined in the predecessor to this paper \cite{CLSY1}. In this section, we recall the definition and examples of weak stability conditions that will be used in Section \ref{sec:Dalphageqzero}. 


\begin{defn} (Definition 3.2, \cite{CLSY1})
		\label{Def:Weak}
		A weak stability condition on $\mathcal{D}$ is a triple $$\sigma=(Z, \A, \{\phi(K)\}_{K\in \ker(Z)\cap \A}),$$ where $\mathcal{A}$ is the heart of a bounded t-structure on $\mathcal{D}$, and $Z: K(\mathcal{D})\rightarrow \mathbb{C}$ a group homomorphism satisfying:
		
		(i) For any $E\in \mathcal{A}$, we have $Z(E)\in\mathbb{H}\cup\mathbb{R}_{\leq 0}$. For any $K\in \ker(Z)\cap\A$, we have $0<\phi(K)\leq 1$. 
  
		(ii) (Weak see-saw property) For any short exact sequence 
		\begin{equation*}
			0\rightarrow K_1\rightarrow K\rightarrow K_2\rightarrow 0
		\end{equation*}
		in $\ker(Z)\cap\A$, we have $\phi(K_1)\geq \phi(K)\geq \phi(K_2)$ or $\phi(K_1)\leq \phi(K)\leq \phi(K_2)$.

		 For any object $E\notin \ker(Z)$, define the phase of an object $E\in\A$ by $$\phi(E)=(1/\pi) \mathrm{arg}\, Z(E)\in (0, 1].$$ We further require 
   
   (iii) The phase function satisfies the Harder-Narasimhan (HN) property. 
	\end{defn}

\subparagraph	\label{para:ssdefforweakstabcon}   We define the slope of an object with respect to a weak stability condition by 
$\rho(E)=-\text{cot}(\pi\phi(E))$
Following \cite{piyaratne2015moduli}, an object $E\in \A$ is (semi)-stable if for any nonzero subobject $F\subset E$ in $\A$, we have 
\[
\rho(F)<(\leq)\rho(E/F),
\]
or equivalently
\[
\phi(F)<(\leq)\phi(E/F).
\]

The notions of slicing can be defined in an analogous way as Bridgeland stability conditions, for more details we refer to \cite{CLSY1}. Then we have the following result:
\begin{prop} \cite[Proposition 3.5]{CLSY1}
A weak stability condition can  equivalently be defined by 
$$\sigma=(Z, \cP, \{\phi(E)\}_{E\in \ker(Z)\cap\cP((0, 1])}),$$
where $\cP$ is a slicing with properties as in Definition 3.3 in \cite{StabTC}, and the central charge $Z: K(\mathcal{D})\rightarrow \mathbb{C}$ is a group homomorphism satisfying:

(i) For $E\in \cP(\phi)$,

\begin{itemize}
    \item if $E\notin \cP(\phi)\cap \ker(Z)$, then $Z(E)=m(E)e^{i\pi\phi}$ where $m(E)\in \mathbb{R}_{> 0}$;
    \item if $E\in \cP(\phi)\cap \ker(Z)$, then $\phi(E)=\phi$.
\end{itemize}

(ii) (weak see-saw property) For any short exact sequence 
		\begin{equation*}
			0\rightarrow K_1\rightarrow K\rightarrow K_2\rightarrow 0
		\end{equation*}
		in $\ker(Z)\cap\cP((0, 1])$, we have $\phi(K_1)\geq \phi(K)\geq \phi(K_2)$ or $\phi(K_1)\leq \phi(K)\leq \phi(K_2)$.
\end{prop}

 \begin{rem}
 \label{Rem:stableness}
 Above, we  defined what it means for  an object  in the abelian category $\A=\cP((0, 1])$ to be  $\sigma$-stable. For a general $t \in \mathbb{R}$, we say an object is stable in $\cP(t)$ if it is a shift of a stable object in $\cP(t')$ for some $0<t'\leq 1$.	
 \end{rem}
 
 Recall in the original definition of weak stability in \cite{piyaratne2015moduli}, that the phases of the objects in the kernel is set to be $1$. In this generalized definition, we allow the phases of the objects in the kernel to be arbitrary. With this flexibility, for the weak stability conditions we will be using, we define the phases of the objects in the kernel to be the limit of their phases at Bridgeland stability conditions. This way, the weak stability conditions behave like limits of Bridgeland stability conditions, and allow us to reduce the problem of stability of line bundles at Bridgeland stability conditions to the problem of  their stability at weak stability conditions. 

\paragraph In Section \ref{sec:Dalphageqzero}, we will be switching among the different weak stability conditions defined in 
\cite{CLSY1}. For convenience, we 
recall here the notation and description of kernel for each of these four types of weak stability conditions.  

\begin{itemize}
    \item (Weak stability conditions at the origin, $\sigma_b^L$)   For any $b \in \mathbb{P}^1$, we have a weak stability condition
\[
 \sigma^L_b=\left(Z_{H}, \B^0_{H, -2}, \{\phi_b(E)\}|_{E\in \ker (Z_{H})\cap\B^0_{H, -2}}\right)
\]
where 
\[
  Z_H = -\ch_2 + i H\ch_1
\]
with $H=\Theta + ef$.  The heart $\Bc^0_{H,-2}$ is obtained by tilting $\Coh(X)$ twice -  we recall the construction here. Let 
\begin{align*}
  \Tc^0_H &= \langle E \in \Coh (X) : E \text{ is $\mu_H$-semistable}, \mu_{H, \mathrm{min}}(E)>0 \rangle \\
  \Fc^0_H &= \langle E \in \Coh (X) : E \text{ is $\mu_H$-semistable}, \mu_{H, \mathrm{max}}(E) \leq 0 \rangle.
\end{align*}
We obtain an abelian category by tilting $\Coh(X)$ at the torsion pair above: 
\[
\A^0_H=\langle \F^0_H[1], \T^0_H\rangle.
\]
This is similar to the standard tilted heart, but with respect to a nef divisor. Set \[
  \Fc^0_{H,-2} = \langle \OO_\Theta (i) : i \leq -2 \rangle
\]
where $\OO_\Theta (i)$ denotes $i_\ast \OO_{\mathbb{P}^1}(i)$, and set
\[
  \Tc^0_{H,-2} = \{ E \in \Ac^0_{H} : \Hom (E, \OO_\Theta (i)) = 0 \text{ for all } i \leq -2 \}.
\]
Performing a second tilt at this torsion pair we obtain the heart
\[
  \Bc^0_{H,-2} = \langle \Fc^0_{H,-2}[1], \Tc^0_{H,-2} \rangle.
\]

Here, 
\[
\kernel (Z_H) \cap \Bc^0_{H,-2} = \langle \OO_X [1], \OO_\Theta (-1)\rangle, 
\]
with phases $\phi_b^L(\cO_X[1])=1$, and $\phi_b^L(\cO_\Theta(-1))=\frac{1}{2}$. See \cite[Section 5.1]{CLSY1} for details. 
\item (Weak stability conditions on $V$-axis, $\sigma_{V,H}$)  For any $V>0$, we have a weak stability condition
\[
\sigma_{V, H}=(Z_{V, H}, \B^0_{H, -2}, \{\phi_{V, H}(K)\}_{K\in\ker(Z_{V, H})\cap\B^0_{H, -2}})
\]
where 
\[
Z_{V,H} = -\ch_2 + V\ch_0 + iH\ch_1
\]
with $H$ and $\B^0_{H, -2}$ as above, and 
\[
\kernel (Z_{V,H})\cap \Bc^0_{H,-2}=\langle \OO_\Theta (-1)\rangle.  
\]
We have $\phi_{V,H}(\OO_\Theta (-1))=\frac{1}{2}$.  See  \cite[Section 5.13]{CLSY1} for details. 
\item (Weak stability conditions on $D$-axis, $\sigma_D$)  For any rational number $D>0$, we have a weak stability condition
\[
\sigma_D=(Z_{D}, \Coh^{\omega, 0}, \{ \phi_D(K) \}_{K \in \kernel (Z_D) \cap \Coh^{\omega,0}} )
\]
where
\[
Z_D = -\ch_2 + i \omega \ch_1
\]
with $\omega = \Theta + (D+e)f$.  Here, $\Coh^{\omega, 0}$ is the standard tilted heart as in Bridgeland stability conditions, and 
\[
\kernel (Z_D) \cap \Coh^{\omega, 0}=\langle \OO_X[1]\rangle
\]
with $\phi_D(\OO_X [1])=1$.  See 
 \cite[Section 5.16]{CLSY1} for details.

\item (Weak stability conditions after an autoequivalence) For any $a \in \mathbb{P}^1$, we have a weak stability condition 
\[
\sigma^R_a=(Z'_0, \B, \{\phi^R_a(K)\}_{K\in \ker(Z'_0)\cap\B})
\]
where 
\[
Z_0' = Z^{td}_{\omega_0', B_0'} = -\ch_2^{B_0'} + (V_{\omega_0'}-1)\ch_0^{B_0'} + i \omega_0' \ch_1^{B_0'}
\]
with $\omega_0' = \tfrac{1}{2}(\Theta + 3f)$ and $B_0' = \tfrac{1}{2}(\Theta + (-2D_\alpha -3)f)$. The heart $\Bc$ is obtained by tilting $\Coh(X)$ twice. Let $\Coh^{\omega_0', B_0'}$ be the standard tilted heart. Set $L_0 = \OO_X (-(D_\alpha+1)f)$ and $L_1=\OO_X (\Theta -(D_\alpha +2)f)$. Consider the subcategories in $\Coh^{\omega'_0, B'_0}$:
\[
\begin{split}
\T_\A&=\langle L_1[1]\rangle,\\
\F_\A&=\{F\in\Coh^{\omega'_0, B'_0} : \Hom(L_1[1], F)=0\}.
\end{split}
\]
Performing the second tilt at this torsion pair and shifting by $-1$, we obtain the heart 
\[
\B=\langle\F_\A, \T_\A[-1]\rangle.
\]
Here,
\[
\kernel (Z_0') \cap \Bc = \langle L_0[1], L_1\rangle
\]
We have $\phi_a^R(L_0[1])=\frac{3}{4}$ and $\phi_a^R (L_1)=\frac{1}{4}$. Intuitively, these weak stability conditions are the "images" of the weak stability conditions $\sigma_b^L$'s under the relative Fourier Mukai transform.  See 
\cite[Section 5.19]{CLSY1} for details.
\end{itemize}

\section{Weak polynomial stability  under group actions}\label{sec:wpsuga}


 The notion of weak polynomial stability condition, as defined in \cite{Lo20}, encompasses Bridgeland stablity condition, polynomial stability condition in the sense of Bayer \cite{BayerPBSC}, and very weak pre-stability condition in the sense of  Piyaratne and Toda \cite{piyaratne2015moduli}.  In this section, we define a notion of \emph{weak polynomial stability data}.  This notion includes the notion of weak stability conditions defined in the predecessor to this paper \cite{CLSY1}.  We also extend some of  the results on preservation of weak polynomial stability under equivalences of triangulated categories in \cite{Lo20}, so that they can be applied to weak stability conditions in Section \ref{sec:Dalphageqzero}.  Compared to the results in \cite{Lo20}, the results in this section need more care, since  objects in the kernel of the central charge are now assigned phases in a very specific manner.
 

\paragraph[Weak polynomial stability functions] We recall some notions on weak polynomial stability from \cite[Section 4]{Lo20}.  For $K=\mathbb{R}$ or $\mathbb{C}$, let $K(\!(w)\!)^c$ denote the field of  Laurent series  in $w$  over $K$ that are convergent on some punctured disk of positive radius centered at the origin.  Given any $0 \neq f(v) \in \CLoovc$, if we let $v \to \infty$ along the positive real axis, then there is a continuous function germ $\phi : \mathbb{R}_{>0} \to \mathbb{R}$ such that 
\[
 f(v) \in \mathbb{R}_{>0}e^{i \pi \phi (v)} \text{ for $v \gg 0$}.
\]
By abuse of notation, we refer to $\phi (v)$ as a polynomial phase function of $f(v)$.  Given two polynomial phase functions $\phi_2, \phi_2$, we write $\phi_1 \prec \phi_2$  (resp.\ $\phi_1 \preceq \phi_2$) if $\phi_1 (v) < \phi_2 (v)$ (resp.\ $\phi_1(v) \leq \phi_2(v)$) for all $v \gg 0$.  As shown in \cite[Lemma 4.3]{Lo20}, the set of all polynomial phase functions is totally ordered with respect to $\preceq$.

\subparagraph \label{para:defpolyphasefunction} Given a triangulated category $\Dc$ and the heart $\Ac$ of a bounded t-structure on $\Dc$, we say a group homomorphism
\[
Z : K(\Dc) \to \CLoovc
\]
(where we consider $\CLoovc$ as a group under addition) is a \emph{weak polynomial stability function} on $\Ac$ with respect to $(\phi_0, \phi_0 +1]$, for some polynomial phase function $\phi_0$, if for all $0 \neq E \in \Ac$ we have
\[
  Z(E)(v) \in \mathbb{R}_{\geq 0}\cdot e^{i\pi (\phi_0(v), \phi_0(v)+1]} \text{\qquad for $v \gg 0$}.
\]
When $Z$ is a weak polynomial stability function on $\Ac$, we set
\[
\Ac_{\kernel Z} = \{ E \in \Ac : Z(E)=0\}
\]
which is a Serre subcategory of $\Ac$.  


\subparagraph[Phase functions and stability] Given a weak polynomial stability function $Z$ on $\Ac$ with respect to $(\phi_0, \phi_0+1]$, we define the (polynomial) phase function $\phi_Z(E)$ of any $E \in \Ac \setminus \Ac_{\kernel Z}$ via the relation
\[
 Z(E)(v) \in \mathbb{R}_{>0}    e^{i \pi \phi_{Z} (E)(v)} \text{ for $v \gg 0$, such that } \phi_0 \prec  \phi_Z(E) \preceq \phi_0+1.
\]

In order to define what it means for an object in $\Ac$ to be semistable with respect to $Z$, however, we also need to define phase functions of objects in $\Ac_{\kernel Z}$.  
\begin{defn}
Given a polynomial phase function $\phi_0$, a \emph{weak polynomial stability data} on a triangulated category $\Dc$ with respect to $(\phi_0, \phi_0+1]$ is a triple $(Z,\Ac,S)$ where
\begin{itemize}
\item[(i)] $\Ac$ is the heart of a bounded t-structure on $\Dc$.
\item[(ii)] $Z : K(\Dc) \to \CLoovc$ is a weak polynomial stability function on $\Ac$ with respect to $(\phi_0, \phi_0+1]$ for some polynomial phase function $\phi_0$.
\item[(iii)] $S$ is an indexed set of the form $\{ \phi_k \}_{0 \neq k \in \Ac_{\kernel Z}}$ where each $\phi_k$ lies in  $(\phi_0, \phi_0+1]$.  That is, for each nonzero object $k$ in $\Ac_{\kernel Z}$, we specify a phase function in $(\phi_0, \phi_0+1]$.
\end{itemize}
\end{defn}

Given a weak polynomial stability data $(Z, \Ac, S)$ as above, we say a nonzero object $E \in \Ac$ is $Z$-semistable (resp.\ $Z$-stable) if, for every $\Ac$-short exact sequence
\[
  0 \to M \to E \to N \to 0
\]
where $M, N\neq 0$, we have
\[
  \phi_Z (M) \preceq \phi_Z (N) \text{\quad (resp.\ $\phi_Z (M) \prec \phi_Z (N)$)}.
\]
More generally, we say that an object $E \in \Dc$ is $Z$-semistable (resp.\ $Z$-stable) if $E[j]$   is a $Z$-semistable (resp.\ $Z$-stable) object in $\Ac$  for some $j \in \mathbb{Z}$, in which case we define $\phi_Z(E)$ via the relation $\phi_Z (E) + j = \phi_Z (E[j])$.  We say $Z$ satisfies the Harder-Narasimhan (HN) property on $\Ac$ if every nonzero object $E$ in $\Ac$ admits a filtration
\begin{equation}\label{eq:HNfilt}
 0=E_0 \subsetneq E_1 \subsetneq \cdots \subsetneq  E_m = E
\end{equation}
in $\Ac$ for some $m \geq 1$, where each quotient $E_i /E_{i-1}$ ($1 \leq i \leq m$) is a $Z$-semistable object and
\[
  \phi_Z (E_1/E_0) \succ \phi_Z (E_2/E_1) \succ \cdots \succ \phi_Z (E_m/E_{m-1}).
\]

\begin{rem}
Once we have a weak polynomial stability function $Z$ on a heart $\Ac$ with respect to some interval $(\phi_0, \phi_0+1]$, we can define the phases of objects in $\Ac \setminus \Ac_{\kernel Z}$.  In order to define which objects in $\Ac$ (including objects in $\Ac_{\kernel Z}$) are $Z$-semistable, however, it becomes necessary to define phases of nonzero objects in $\Ac_{\kernel Z}$, hence the need for the notion of weak polynomial stability data.
\end{rem}

\paragraph[Relations to other types of stability] Suppose $(Z, \Ac, S)$ is a weak polynomial stability data  with respect to $(\phi_0, \phi_0+1]$.  If $Z$ has the HN property on $\Ac$, then we say $(Z,\Ac)$ is a weak polynomial stability condition with respect to $(\phi_0, \phi_0+1]$. 

Suppose  $(Z,\Ac)$ is a weak polynomial stability condition with respect to $(\phi_0, \phi_0+1]$.  
\begin{itemize}
\item When $\phi_Z (k)=\phi_0+1$ for every $0 \neq k \in \Ac_{\kernel Z}$, the triple $(Z, \Ac, S)$ is equivalent to a weak polynomial stability condition as defined in \cite{Lo20}. 
\item When $Z$ factors through the  inclusion $\mathbb{C} \subset \CLoovc$, $\phi_0=0$, and  $\phi_Z(k)=1$ for every $0\neq k \in \Ac_{\kernel Z}$, the triple $(Z,\Ac, S)$ is equivalent to  a very weak pre-stability condition in the sense of Piyaratne and Toda \cite{piyaratne2015moduli}.
\item When $Z$ factors through the inclusion $\mathbb{C} \subset \CLoovc$, $\phi_0=0$, and the weak see-saw property holds (see \ref{para:weakseesawprop}), the triple $(Z,\Ac,S)$ is a weak stability condition in the sense of Section \ref{sec:defwsc}.
\item When $\Ac_{\kernel Z}=0$, the set $S$ is empty, and the triple of data $(Z,\Ac,S)$ is equivalent to the pair $(Z,\Ac)$, which  is precisely a polynomial stability condition with respect to $(\phi_0, \phi_0+1]$ in the sense of Bayer \cite{BayerPBSC}.
\item When $\Ac_{\kernel Z}=0$, and $Z$ factors through the canonical inclusion $\mathbb{C} \subset \CLoovc$ with $\phi_0=0$, the pair $(Z,\Ac)$ is precisely a Bridgeland stability condition.
\end{itemize}


\paragraph[Group actions on weak polynomial stability functions]  Recall from \ref{para:Brimfdgroupactions}  that we have two commuting group actions on the Bridgeland stability manifold of a triangulated category $\Dc$: a left-action by the group of autoequivalences of $\Dc$, and a right action by the group $\wt{\mathrm{GL}}^+\! (2,\mathbb{R})$.  In \cite[Section 5]{Lo20}, we defined an analogue of the $\wt{\mathrm{GL}}^+\! (2,\mathbb{R})$-action  on  polynomial stability conditions.  We recall some of the basics here.  

Given any element $T \in \mathrm{GL}(2,\RLoovc)$, we can think of $T$ as a function 
\[
[v_0, \infty)  \to \mathrm{GL}(2,\mathbb{R}) : v \mapsto T_v:=T(v)
\]
for some fixed $v_0 \in \mathbb{R}_{>0}$.  Then we can define
\[
  \GLlp = \{ T \in \mathrm{GL}(2,\RLoovc) : T_v \in \mathrm{GL}^+(2,\mathbb{R}) \text{\quad for $v \gg 0$}\}.
\]

Now for any $T \in \GLlp$, there exists $v_1 \in \mathbb{R}_{>0}$ such that we have a path
\[
  \mathrm{Pa}(T) : [v_1, \infty) \to \mathrm{GL}^+(2,\mathbb{R}) : v \mapsto T_v
\]
in $\mathrm{GL}^+(2,\mathbb{R})$, which we can lift to a path in the universal covering space $\wt{\mathrm{GL}}^+\! (2,\mathbb{R})$
\begin{equation}\label{eq:liftedpath}
  \wt{\mathrm{Pa}} (T) : [v_1,\infty) \to \wt{\mathrm{GL}}^+\!(2,\mathbb{R}) : v \mapsto (T_v,g_v).
\end{equation}
If we  write $P$ to denote the set of all the germs of polynomial phase functions ordered by $\prec$, then $T$ induces an order-preserving bijection 
\[
\Gamma_T : P \to P : \phi \mapsto \phi' 
\]
where
\[
\phi'(v) := g_v(\phi (v)) \text{ for $v \gg 0$}
\]
which depends on the choice $\{g_v\}_v$ (see \cite[5.2]{Lo20} for details).  Note that $\Gamma_T (\phi+1)=\Gamma_T(\phi)+1$ for all $\phi$.

We can  regard elements of $\CLoovc$ as 2 by 1 column vectors with entries from $\RLoovc$.  Then for any triangulated category $\Dc$ and any  weak polynomial stability function $Z : K(\Dc) \to \CLoovc$ on a heart $\Ac$ with respect to $(\phi_0, \phi_0+1]$, and any  $T \in \GLlp$, we can define the group homomorphism $TZ : K(\Dc) \to \CLoovc : E\mapsto T\cdot Z(E)$, and then $TZ$ is a weak polynomial stability function on $\Ac$ with respect to $(\Gamma_T(\phi_0), \Gamma_T (\phi_0)+1]$.

\paragraph[Configuration III-w]  Following the terminology in \cite[5.3]{Lo20}, we will say we are in Configuration III-w  if we have the following setup:
\begin{itemize}
\item[(a)] $\Phi : \Dc \to \Uc$ and $\Psi : \Uc \to \Dc$ are exact equivalences between triangulated categories $\Dc, \Uc$ satisfying $\Psi \Phi \cong \mathrm{id}_\Dc [-1]$ and $\Phi \Psi \cong \mathrm{id}_\Uc [-1]$.
\item[(b)] $\Ac, \Bc$ are hearts of bounded t-structures on $\Dc, \Uc$, respectively, such that $\Phi \Ac \subset D^{[0,1]}_\Bc$.
\item[(c)] There exist weak polynomial stability functions $Z_\Ac : K(\Dc) \to \CLoovc$ and $Z_{\Bc} : K(\Uc) \to \CLoovc$ on $\Ac, \Bc$  with respect to $(a,a+1], (b,b+1]$, respectively, together with an element $T \in \GLlp$ such that
    \begin{equation}\label{eq:AG45-108-17}
 Z_{\Bc} (\Phi E) = T   Z_{\Ac}(E) \text{\quad for all $E\in \Dc$}
\end{equation}
    and a choice of an induced map $\Gamma_T$ such that
    \begin{equation}\label{eq:AG47-23-1}
    a \prec \Gamma_T^{-1}(b)  \prec a+1.
    \end{equation}
\end{itemize}
Writing $\wh{T}=-T^{-1}$, we know from \cite[5.5]{Lo20} that Configuration III-w is symmetric in the pairs $\Phi$ and $\Psi$, $\Ac$ and $\Bc$, $Z_\Ac$ and $Z_\Bc$, $a$ and $b$, $T$ and $\wh{T}$ and also $\Gamma_T$ and $\Gamma_{\wh{T}}$ for appropriate choices of $\Gamma_{\wh{T}}$ and $\Gamma_{\wh{T}}$.  Note that conditions (a) and (b) together imply $(W_{0,\Phi, \Ac, \Bc}, W_{1, \Phi, \Ac, \Bc})$ forms a torsion pair in $\Ac$ \cite[Lemma 3.10(ii)]{Lo20}.

In \cite[Theorem 5.6]{Lo20}, we proved a result on preservation of stability under an autoequivalence when we are in Configuration III-w.  In proving that result, however, we assumed that nonzero objects in $\Ac_{\kernel Z_\Ac}$ (and similarly for $\Bc_{\kernel Z_\Bc}$) have maximal phase.  In Proposition \ref{prop:weakstabconcorr-var1}  below, we prove a similar result for  weak polynomial stability data in general.  We will need the next two lemmas:

\begin{lem}\cite[Theorem 5.6(i)]{Lo20}\label{lem:Lo20Thm5-6part-i}
    Assume Configuration III-w.  Then for any $\Phi_\Bc$-WIT$_0$ object $E$  and any $\Phi_\Bc$-WIT$_1$ object $F$ in $\Ac$ such that  $Z_\Ac (E), Z_\Ac (F) \neq 0$, we have
    \[
      \phi_{\Ac} (F) \preceq \Gamma_T^{-1}(b)  \prec \phi_\Ac (E).
    \]
\end{lem}


\begin{lem}\label{lem:AG52-34-1}\label{lem:seesawmiddlenonzero}
Let $\Dc$ be a triangulated category, and $(Z, \Ac, S)$ a weak polynomial stability data with respect to $(\phi_0, \phi_0+1]$.    Then for any short exact sequence in $\Ac$
\[
0 \to M \to E \to N \to 0
\]
where $Z(E)\neq 0$, we have:
\begin{itemize}
\item Whenever $\phi_Z(M) \preceq \phi_Z (N)$, we have $\phi_Z(M) \preceq \phi_Z (E) \preceq \phi_Z(N)$.
\item Whenever $\phi_Z(M) \succeq \phi_Z (N)$, we have $\phi_Z(M) \succeq \phi_Z (E) \succeq \phi_Z(N)$.
\end{itemize}
\end{lem}

\begin{proof}
Since $Z(E)\neq 0$, $M$ and $N$ cannot both lie in $\Ac_{\kernel Z}$.   Suppose $Z(M)=0$.  Then $Z(E)=Z(N)\neq 0$, and the claims follow easily.  Similarly for the case where $Z(N)=0$.

If $Z(M), Z(N)$ are both nonzero, then the claim follows from the usual seesaw property.
\end{proof}

\begin{prop}\label{prop:weakstabconcorr-var1}
Assume Configuration III-w, and suppose $(Z_\Ac, \Ac, S_\Ac), (Z_\Bc, \Bc, S_\Bc)$ are weak polynomial stability data for some $S_{\Ac}, S_{\Bc}$.  Also suppose $E \in \Ac \setminus \Ac_{\kernel Z_{\Ac}}$  is $\Phi_\Bc$-WIT$_i$ for $i=0$ or $1$, and that  the following vanishing condition holds:
\begin{itemize}
\item For any $\Bc$-short exact sequence $0 \to M \to \Phi E[i]\to N \to 0$     where $M, N \neq 0$, if $M$ or $N$ lies in $\Bc_{\kernel Z_\Bc}$, then we have $\phi_\Bc (M) \preceq  ( \text{resp.} \prec ) \phi_\Bc (N)$.
\end{itemize}
Then
\[
  \text{$E$ is $Z_\Ac$-semistable (resp.\ -stable) }\Rightarrow \text{ $\Phi E$ is $Z_\Bc$-semistable (resp.\ -stable)}.
\]
\end{prop}

\begin{proof}
We prove only the statement for `semistable', as the same argument works for the `stable' case with the obvious modifications.

(I) Suppose $E$ is $\Phi_\Bc$-WIT$_0$, and consider a $\Bc$-short exact sequence
\[
0 \to M \to \Phi E \to N \to 0
\]
where $M, N \neq 0$.  Since $\Phi E$ is $\Psi_\Ac$-WIT$_1$, so is $M$.  Hence we obtain  the long exact sequence in $\Ac$
\[
0 \to \Psi^0_\Ac N \to \wh{M} \overset{\gamma}{\to} E \to \Psi^1_\Ac N \to 0
\]
where $\wh{M} := \Psi M [1]$.  If $\image\gamma = 0$, then $\Psi^0_\Ac N \cong \wh{M}$.  Since $\Psi^0_\Ac N$ is $\Phi_\Bc$-WIT$_1$ while $\wh{M}$ is $\Phi_\Bc$-WIT$_0$, it follows that $\Psi^0_\Ac N \cong \wh{M} = 0$, i.e.\ $M=0$, a contradiction.  Hence $\image \gamma \neq 0$.

Now suppose $Z_\Ac (\image \gamma)=0$. Then $Z_\Ac (\Psi^0_\Ac N) = Z_\Ac (\wh{M})$.  If $Z_\Ac (\wh{M}) \neq 0$, then by  Lemma \ref{lem:Lo20Thm5-6part-i} we have $\phi_\Ac (\Psi^0_\Ac N) \prec \phi_\Ac (\wh{M})$, contradicting $\Psi^0_\Ac N, \Psi M$ having the same nonzero value under $Z_\Ac$.  So we must have $Z_\Ac (\wh{M} )=0$, in which case $Z_\Bc (M)=0$ by \eqref{eq:AG45-108-17}, i.e.\ $M \in \Bc_{\kernel Z_\Bc}$.  The vanishing condition in the hypotheses then ensures that $\Phi E$ is not destabilised by $M$.

Suppose instead that $Z_\Ac (\image \gamma ) \neq 0$.  We divide further into two cases:

\noindent\textbf{Case 1}: $Z_\Ac (\Psi^0_\Ac N)=0$.  In this case, we have $Z_\Ac (\wh{M}) = Z_\Ac (\image \gamma)\neq 0$ and hence $\phi_\Ac (\wh{M}) = \phi_\Ac (\image \gamma)$.  On the other hand, the $Z_\Ac$-semistability of $E$ gives $\phi_\Ac (\image \gamma) \preceq \phi_\Ac (\Psi^1_\Ac N)$, and so we have
\[
  \phi_\Ac (\wh{M}) \preceq \phi_\Ac (\Psi^1_\Ac N).
\]

If $Z_\Ac (\Psi^1_\Ac N)=0$, then $N \in \Bc_{\kernel Z_{\Bc}}$, and the vanishing condition ensures $\Phi E$ is not destabilised.  So let us assume  $Z_\Ac (\Psi^1_\Ac N)\neq 0$, in which case $Z_\Bc (\Phi (\Psi^1_\Ac N))=Z_\Bc (N) \neq 0$, and so the inequality above  implies (since $\wh{M}, \Psi^1_\Ac N$ are both $\Phi_\Bc$-WIT$_0$ objects in $\Ac$ with nonzero $Z_\Ac$)
\[
  \phi_\Bc (M) \preceq \phi_\Bc (\Phi (\Psi^1_\Ac N)) = \phi_\Bc (N).
\]
That is, $\Phi E$ is not destabilised in this case.

\noindent\textbf{Case 2}: $Z_\Ac (\Psi^0_\Ac N) \neq 0$.  In this case, we have that $\image \gamma$ is $\Phi_\Bc$-WIT$_0$ since $\wh{M}$ is so.  Also, since $Z_\Ac (\Psi^0_\Ac N)$ and $Z_\Ac (\image \gamma)$ are both nonzero, we have $Z_\Ac (\wh{M})\neq 0$. Then 
\begin{align*}
  &\phi_\Ac (\Psi^0_\Ac N) \prec \phi_\Ac (\image \gamma) \text{ by Lemma \ref{lem:Lo20Thm5-6part-i}} \\
  \Rightarrow \, &\phi_\Ac (\wh{M}) \prec \phi_\Ac (\image \gamma) \preceq \phi_\Ac (E) \text{ by the usual seesaw property and Lemma \ref{lem:seesawmiddlenonzero}} \\
  \Rightarrow \, &\phi_\Ac (\wh{M}) \prec \phi_\Ac (E) \\
  \Rightarrow \, &\phi_\Bc (M) \prec \phi_\Bc (\Phi E). \text{ since $\wh{M}, E$ are both $\Phi_\Bc$-WIT$_0$ objects with nonzero $Z_\Ac$}
\end{align*}
  As a result,  all of $M, \Phi E$ and $N$ have nonzero values under $Z_\Ac$ and hence under $Z_\Bc$, and so the last  inequality implies  $\phi_\Bc (M) \prec \phi_\Bc (N)$.

(II) Now we turn to the case where  $E$ is $\Phi_\Bc$-WIT$_1$, and consider a $\Bc$-short exact sequence
\[
0 \to M \to \Phi E [1]\to N \to 0
\]
where $M, N \neq 0$.  Since $\Phi E[1]$ is $\Psi_\Ac$-WIT$_0$, so is $N$, and we obtain  the long exact sequence in $\Ac$
\[
0 \to \Psi^0_\Ac M \to E \overset{\gamma}{\to} \Psi N \to \Psi^1_\Ac M \to 0.
\]
As in the  case where $E$ is $\Phi_\Bc$-WIT$_0$, we must have $\image \gamma \neq 0$, or else $\Psi N \cong \Psi^1_\Ac M$ which forces $\Psi N=0$, a contradiction.

Suppose $Z_\Ac (\image \gamma) =0$.  In this case, we have $Z_\Ac (\Psi N) = Z_\Ac (\Psi^1_\Ac M)$.    If $Z_\Ac (\Psi N)=0$, then $N \in \Bc_{\kernel Z_\Bc}$, and the vanishing condition in the hypotheses  ensures $\Phi E [1]$ is not destabilised by $N$.  If $Z_\Ac (\Psi N) \neq 0$, then  $Z_\Ac (\Psi^1_\Ac M)  =Z_\Ac (\Psi N) \neq 0$.  Since $\Psi N$ is $\Phi_\Bc$-WIT$_1$ while $\Psi^1_\Ac M$ is $\Phi_\Bc$-WIT$_0$,  by Lemma \ref{lem:Lo20Thm5-6part-i} we have $\phi_\Ac (\Psi N) \prec \phi_\Ac (\Psi^1_\Ac M)$, contradicting $\Psi N, \Psi^1_\Ac M$ having the same nonzero value under $Z_\Ac$.

Now suppose $Z_\Ac (\image \gamma) \neq 0$.  Again, we divide further into two cases:

\noindent\textbf{Case 1:} $Z_\Ac (\Psi^1_\Ac M)=0$.  In this case, we have $0 \neq Z_\Ac (\image \gamma) = Z_\Ac (\Psi N)$ and so $\phi_\Ac (\image \gamma) = \phi_\Ac (\Psi N)$.  On the other hand, the $Z_\Ac$-semistability of $E$ gives $\phi_\Ac (\Psi^0_\Ac M) \preceq \phi_\Ac (\image \gamma)$. 

If $Z_\Ac (\Psi^0_\Ac M)=0$, then $M \in \Bc_{\kernel Z_\Bc}$, and the vanishing condition in the hypotheses says $\Phi E [1]$ is not destabilised.  So suppose $Z_\Ac (\Psi^0_\Ac M) \neq 0$ from here on.  Then $Z_\Ac (\Psi M) \neq 0$ and $\phi_\Ac (\Psi^0_\Ac M) = \phi_\Ac (\Psi M)$, and from the previous paragraph we have
\begin{align*}
  & \phi_\Ac (\Psi M) \preceq \phi_\Ac (\Psi N) \\
  \Rightarrow \,\, &\phi_\Bc (M) \preceq \phi_\Bc (N) \text{ since $\Psi M, \Psi N$ both have nonzero values under $Z_\Ac$}.
\end{align*}

\noindent\textbf{Case 2:} $Z_\Ac (\Psi^1_\Ac M) \neq 0$.  In this case, we have
\[
  \phi_\Ac (E) \preceq \phi_\Ac (\image \gamma) \prec \phi_\Ac (\Psi^1_\Ac M).
\]
Here,  the first inequality follows from the $Z_\Ac$-semistability of $E$ and Lemma \ref{lem:seesawmiddlenonzero}.  The second inequality  follows from $\image \gamma$ being $\Phi_\Bc$-WIT$_1$ (since $\Psi N$ is so) with nonzero $Z_\Ac$, and $\phi_\Ac (\Psi^1_\Ac M)$ being $\Phi_\Bc$-WIT$_0$ with nonzero $Z_\Ac$, and Lemma \ref{lem:Lo20Thm5-6part-i}.  Note that $Z_\Ac (\image \gamma)\neq 0$ and $Z_\Ac (\Psi^1_\Ac M)\neq 0$ together imply $Z_\Ac (\Psi N) \neq 0$, and so the second inequality above also implies $\phi_\Ac (\image \gamma) \prec \phi_\Ac (\Psi N)$.  Thus 
\begin{align*}
  &\phi_\Ac (E) \prec \phi_\Ac (\Psi N) \\
  \Rightarrow \, &\phi_\Bc (\Phi E[1]) \prec \phi_\Bc (N) \text{ since $E, \Psi N$ are both $\Phi_\Bc$-WIT$_1$ with nonzero $Z_\Ac$}.
\end{align*}
Since $Z_\Ac (M) \neq 0$ and hence $Z_\Bc (M) \neq 0$, all of $M, \Phi E[1]$ and $N$ have nonzero values under $Z_\Bc$, and it follows that $\phi_\Bc (M) \prec \phi_\Bc (N)$.
\end{proof}

\paragraph[Weak seesaw property] \label{para:weakseesawprop} Given a weak polynomial stability data $(Z, \Ac, S)$ on a triangulated category $\Dc$ with respect to $(\phi_0, \phi_0+1]$, we say it has the weak seesaw property if, for every short exact sequence of nonzero objects
\begin{equation}\label{eq:AG52-34-ses}
0 \to E' \to E \to E'' \to 0
\end{equation}
in the heart $\Ac$, we have either 
\begin{equation}\label{eq:seesaw1}
\phi_Z(E') \preceq \phi_Z (E) \preceq \phi_Z (E'')
\end{equation}
or 
\begin{equation}\label{eq:seesaw2}
\phi_Z (E') \succeq \phi_Z (E) \succeq \phi_Z(E'').
\end{equation}

Given a weak polynomial stability data as above, note that 
\begin{itemize}
    \item If $Z(E'), Z(E), Z(E'')$ are all nonzero, then the usual seesaw property holds, and we have either \eqref{eq:seesaw1} or \eqref{eq:seesaw2}.
    \item If $Z(E)\neq 0$, but either $Z(E')=0$ or $Z(E'')=0$, then we also have either \eqref{eq:seesaw1} or \eqref{eq:seesaw2}.
\end{itemize}
Therefore, requiring a weak polynomial stability data to have the weak seesaw property is equivalent to requiring only short exact sequences \eqref{eq:AG52-34-ses} in $\Ac_{\kernel Z_\Ac}$ to have the weak seesaw property (cf.\ Definition \ref{Def:Weak}(ii)).

The following is essentially  \cite[Theorem 5.6(ii)]{Lo20}:

\begin{lem}\label{lem:Lo20-Thm5p6ii-gen}
Assume Configuration III-w,  and suppose $(Z_\Ac, \Ac, S_\Ac), (Z_\Bc, \Bc, S_\Bc)$ are weak polynomial stability data for some $S_{\Ac}, S_{\Bc}$.   Suppose $E \in \Ac \setminus \Ac_{\kernel Z_\Ac}$ is a $Z_\Ac$-semistable object such that
\begin{itemize}
\item $E$ has no nonzero $\Ac$-subobjects in $\Ac_{\kernel Z_\Ac} \cap W_{0, \Phi, \Ac, \Bc}$.
\item $E$ has no nonzero $\Ac$-quotients in $\Ac_{\kernel Z_\Ac} \cap W_{1, \Phi, \Ac, \Bc}$.
\end{itemize}
Then $E$ is $\Phi_\Bc$-WIT$_i$ for $i=0$ or $1$.
\end{lem}

\begin{proof}
By \cite[Lemma 3.10(ii)]{Lo20}, there is an $\Ac$-short exact sequence
\[
0 \to E_0 \to E \to E_1 \to 0
\]
where $E_i$ is $\Phi_\Bc$-WIT$_i$ for $i=0,1$.  Suppose $E_0, E_1$ are both nonzero.  Then by our assumptions on $E$, neither $E_0$ nor $E_1$ lies in $\Ac_{\kernel Z_{\Ac}}$.   Lemma \ref{lem:Lo20Thm5-6part-i} then gives $\phi_\Ac (E_0) \succ \phi_\Ac (E_1)$, contradicting the $Z_\Ac$-semistability of $E$.
\end{proof}

\begin{lem}\label{lem:tech1}
Suppose $(Z, \Ac, S)$ is a weak polynomial stability data on a triangulated category $\Dc$ that satisfies the weak seesaw property.    Then for any nonzero $Z$-semistable objects $E, F$ in $\Ac$ such that $\phi_Z (E) \succ \phi_Z (F)$, we have $\Hom(E,F)=0$.
\end{lem}


\begin{proof}
Take any morphism $\alpha : E \to F$ and assume $\image \alpha \neq 0$.

If $\kernel \alpha = \cokernel \alpha = 0$, then $E \cong F$ and we have a contradiction. 

Suppose $\kernel \alpha$ and $\cokernel \alpha$ are both nonzero.  The semistability of $E$ together with the weak seesaw property imply $\phi_Z (\kernel \alpha) \preceq \phi_Z(E) \preceq \phi_Z (\image \alpha)$, whereas the semistability of $F$ and the weak seesaw property together give $\phi_Z (\image \alpha) \preceq \phi_Z(F) \preceq \phi_Z (\cokernel \alpha)$.  Overall, we have $\phi_Z (E) \preceq \phi_Z (F)$ which is a contradiction.

Next, suppose $\kernel \alpha =0$ but $\cokernel \alpha \neq 0$.  Then the semistability of $F$ implies $\phi_Z (E) \preceq \phi (\cokernel \alpha)$.  Now, if $Z(F) \neq 0$, then $\phi_Z (E) \preceq \phi_Z(F) \preceq \phi_Z (\cokernel \alpha)$ by Lemma \ref{lem:seesawmiddlenonzero}, giving us a contradiction.  On the other hand, if $Z(F)=0$, then $E$ and $\cokernel \alpha$ also lie in $\Ac_{\kernel Z}$, and by the weak seesaw property along with the assumption $\phi_Z(E) \succ \phi_Z (F)$, we must have
\[
\phi_Z (E) \succ \phi_Z (F) \succeq \phi_Z (\cokernel \alpha),
\]
contradicting the semistability of $F$.  

The argument for the case where $\kernel \alpha\neq 0$ and $\cokernel \alpha =0$ is dual to the argument in the previous paragraph.
\end{proof}

\begin{prop}\label{Prop:HNcorr}
Assume Configuration III-w, and suppose $(Z_\Ac, \Ac, S_\Ac), (Z_\Bc, \Bc, S_\Bc)$ are weak polynomial stability data for some $S_{\Ac}, S_{\Bc}$, such that they both satisfy the weak seesaw property.  In addition, assume we have
\begin{itemize}
\item[(d)] $\Phi (\Ac_{\kernel Z_\Ac}) = \Bc_{\kernel Z_\Bc}$.
\item[(e)] Every nonzero object $F$ in $\Ac_{\kernel Z_\Ac}$ satisfies $\phi_\Ac (F) \succ \Gamma^{-1}_T (b)$, while every nonzero object $F'$ in $\Bc_{\kernel Z_\Bc}$ satisfies $\phi_\Bc (F') \preceq \Gamma_T(a)+1$.
\item[(f)] For any objects $F', F''$ in $\Ac$ that are both $\Phi_\Bc$-WIT$_i$ for some $i$, we have
    \[
      \phi_\Ac (F') \prec (\text{resp.}\, \preceq) \phi_\Ac (F'') \Rightarrow \phi_\Bc (\Phi F'[i]) \prec (\text{resp.}\, \preceq) \phi_\Bc (\Phi F'' [i]),
    \]
    and that the same statement holds when we switch the roles of $\Ac, \Phi, \phi_\Ac$ with $\Bc, \Psi, \phi_\Bc$, respectively.
\item[(g)] $\Ac_{\kernel Z_\Ac}$ is generated via direct sum by a finite number of $Z_\Ac$-semistable objects, and the $Z_\Ac$-semistable objects in $\kernel Z_\Ac$ are taken by $\Phi$ to $Z_\Bc$-semistable objects in $\Bc$.  The analogous statement holds when we switch the roles of $\Ac, Z_\Ac, \Phi$ with $\Bc, Z_\Bc, \Psi$, respectively.
\end{itemize}
Then whenever $Z_\Ac$ has the HN property on $\Ac$, $Z_\Bc$ also has the HN property on $\Bc$.
\end{prop}

\begin{proof}
Suppose $Z_\Ac$ has the HN property on $\Ac$.  Take any $0 \neq E\in \Bc$, and consider the filtration in
\[
0 \to E_0 \to E \to E_1 \to 0
\]
in $\Bc$ where $E_i$ is $\Psi_\Ac$-WIT$_i$ for $i=0,1$.  This is taken by $\Psi [1]$ to the exact triangle
\[
  \Psi^0_\Ac E_0 [1] \to \Psi E [1] \to \Psi^1_\Ac E_1 \to \Psi^0_\Ac E_0 [2]
\]
where $\Psi^0_\Ac E_0$ is $\Phi_\Bc$-WIT$_1$ and $\Psi_\Ac^1 E_1$ is $\Phi_\Bc$-WIT$_0$.  Since $Z_\Ac$ has the HN property on $\Ac$, we have $Z_\Ac$-HN filtrations
\begin{align}
  0= M_0 \subsetneq M_1 \subsetneq \cdots \subsetneq M_m &= \Psi^0_\Ac E_0, \label{eq:filt1}\\
  0= N_0 \subsetneq N_1 \subsetneq \cdots \subsetneq N_n &= \Psi^1_\Ac E_1. \label{eq:filt2}
\end{align}
We will show that the concatenation of these two filtrations is taken by $\Phi$ to the HN filtration of $E$.  To do this, we will first show that each $M_i/M_{i-1}$ (resp.\ $N_j/N_{j-1}$) is $\Phi_\Bc$-WIT$_1$ (resp.\ $\Phi_\Bc$-WIT$_0$), and then show that $\Phi$ takes each of the $M_i/M_{i-1}[1]$,  $N_j/N_{j-1}$ to a $Z_\Bc$-semistable object in $\Bc$.

\textbf{Step 1.} Suppose $E_0 \neq 0$.  In this step, we show that each $M_i/M_{i-1}$ is $\Phi_\Bc$-WIT$_1$.  To begin with, note that $M_1$ is a nonzero subobject of $\Psi^0_\Ac E_0$ in $\Ac$, and so is $\Phi_\Bc$-WIT$_1$.  By condition (d), $M_1$ cannot lie in $\Ac_{\kernel Z_\Ac}$, and so by Lemma \ref{lem:Lo20Thm5-6part-i}, we have $\phi_\Ac (M_1) \preceq \Gamma^{-1}_T(b)$, and hence
\begin{equation}\label{eq:ONC26-1}
  \phi_\Ac (M_i/M_{i-1}) \preceq \Gamma^{-1}_T(b)
\end{equation}
for all $1 \leq i \leq m$.  By condition (e), this also means that none of the $M_i/M_{i-1}$ lies in $\Ac_{\kernel Z_\Ac}$.

Since $\Ac_{\kernel Z_\Ac}$ is a Serre subcategory of $\Ac$, every $Z_\Ac$-HN factor of any object in $\Ac_{\kernel Z_\Ac}$ has $\phi_\Ac \succ \Gamma^{-1}_T(b)$ by condition (e).  Since $M_i/M_{i-1}$ is $Z_\Ac$-semistable, it follows from Lemma \ref{lem:tech1} that $$\Hom (\Ac_{\kernel Z_\Ac}, M_i/M_{i-1})=0$$ for all $1 \leq i \leq m$.  Also, we have $\Ac_{\kernel Z_\Ac}\cap W_{1,\Phi,\Ac, \Bc}=0$ by condition (d).  Then by  Lemma \ref{lem:Lo20-Thm5p6ii-gen}, we know that $M_i/M_{i-1}$ is $\Phi_\Bc$-WIT$_j$ for either $j=0$ or $1$, for all $i$.  From \eqref{eq:ONC26-1} and Lemma \ref{lem:Lo20Thm5-6part-i}, we know $j$ must be $1$.

\textbf{Step 2.} Suppose $E_1 \neq 0$.  In this step, we show that each $N_i/N_{i-1}$ is $\Phi_\Bc$-WIT$_0$.  Since $\Psi^1_\Ac E_1$ is $\Phi_\Bc$-WIT$_0$, it follows that $N_n/N_{n-1}$ is also $\Phi_\Bc$-WIT$_0$.    We will now prove $N_i/N_{i-1}$ is also $\Phi_\Bc$-WIT$_0$ for $1 \leq i < n$.

Take any  $1 \leq i < n$ and set  $\wt{N} = N_i/N_{i-1}$.  Consider the $\Ac$-short exact sequence
\[
0 \to \wt{N}_0 \to \wt{N} \to \wt{N}_1 \to 0
\]
where $\wt{N}_i$ is $\Phi_\Bc$-WIT$_i$ for $i=0, 1$.  

Suppose $\wt{N}_1 \neq 0$.  By condition (d), we must have $Z_\Ac (\wt{N}_1) \neq 0$.   Then by condition (e) together with Lemma \ref{lem:Lo20Thm5-6part-i}, regardless of whether $\wt{N}_0$ lies in $\Ac_{\kernel Z_\Ac}$, we have
\[
  \phi_\Ac (\wt{N}_0) \succ \Gamma^{-1}_T(b) \succeq \phi_\Ac (\wt{N}_1),
\]
which contradicts the $Z_\Ac$-semistability of $\wt{N}$ unless $\wt{N}_0=0$.  Hence $\wt{N}_0$ must vanish, and $\wt{N}= \wt{N}_1$ is $\Phi_\Bc$-WIT$_1$. By condition (d), $\wt{N}_1$ cannot lie in $\Ac_{\kernel Z_\Ac}$.  Then  by Lemma \ref{lem:Lo20Thm5-6part-i} again, we have $\phi_\Ac (\wt{N}_1) \preceq \Gamma^{-1}_T(b)$.  

On the other hand, since $N_n/N_{n-1}$ is $\Phi_\Bc$-WIT$_0$, regardless of whether it is in $\Ac_{\kernel Z_\Ac}$, we must have 
\[
\Gamma^{-1}_T(b) \prec \phi_\Ac (N_n/N_{n-1}),
\]
by   Lemma \ref{lem:Lo20Thm5-6part-i} and condition (e).  Thus we have 
\[
 \phi_\Ac (\wt{N})= \phi_\Ac (\wt{N}_1) \preceq \Gamma^{-1}_T(b) \prec \phi_\Ac (N_n/N_{n-1})
\]
overall, contradicting the property of an HN filtration.  Hence $\wt{N}_1$ must vanish, i.e.\ $\wt{N}$ must be $\Phi_\Bc$-WIT$_0$.

\textbf{Step 3.} In this step, we show that $\Phi (M_i/M_{i-1}) [1]$ is a $Z_\Bc$-semistable object in $\Bc$ for all $1 \leq i \leq m$.  Set $G = M_i/M_{i-1}$ for some $i$.  We know  $G$ is $\Phi_\Bc$-WIT$_1$  and $Z_\Ac (G)\neq 0$ (hence $Z_\Bc (\Phi G [1])\neq 0$) from Step 1.  We will use  Proposition \ref{prop:weakstabconcorr-var1} to check the semistability of $\Phi G[1]$.

Take any  $\Bc$-short exact sequence of the form
\[
0 \to P \to \Phi G [1] \to Q \to 0
\]
where $P, Q \neq 0$. In the case where $Q \in \Bc_{\kernel Z_\Bc}$, we have $Q$ is $\Psi_\Ac$-WIT$_1$ by condition (d) while $\Phi G [1]$ is $\Psi_\Ac$-WIT$_0$, and so $\Hom (\Phi G [1], Q)=0$, meaning the above short exact sequence cannot exist.  

In the case where $P \in \Bc_{\kernel Z_\Bc}$, we have from condition (e) that $\phi_\Bc (P) \preceq \Gamma_T(a)+1$.  On the other hand, that $Z_\Bc (\Phi G [1])\neq 0$ implies $Z_\Bc (Q)\neq 0$.  That $\Phi G[1]$ is $\Psi_\Ac$-WIT$_0$ then implies $Q$ is $\Psi_\Ac$-WIT$_0$.  From Lemma \ref{lem:Lo20Thm5-6part-i} and the symmetry in Configuration III-w, we now have $\phi_\Bc (Q) \succ \Gamma_T (a)+1$.  Overall, we have
\[
\phi_\Bc (P) \preceq \Gamma_T(a)+1 \prec \phi_B (Q).
\]
Thus by Proposition \ref{prop:weakstabconcorr-var1}, $\Phi G[1]$ is $Z_\Bc$-semistable.

\textbf{Step 4.} In this step, we show that $\Phi (N_i/N_{i-1})$ is a $Z_\Bc$-semistable object in $\Bc$ for all $1 \leq i \leq n$.  Let $G = N_i/N_{i-1}$ for any $1 \leq i \leq n$.  We know from Step 2 that $G$ is $\Phi_\Bc$-WIT$_0$.  If $G \in \Ac_{\kernel Z_\Ac}$, then $\Phi G$ is  $Z_\Bc$-semistable by condition (g), so let us assume from now on that $G \notin \Ac_{\kernel Z_\Ac}$.  

Consider any $\Bc$-short exact sequence of the form
\[
0 \to P \to \Phi G \to Q \to 0
\]
where $P, Q \neq 0$.  Suppose $Q \in \Bc_{\kernel Z_\Bc}$.  Then $Q$ is $\Psi_\Ac$-WIT$_1$ by condition (d).  Also, that $\Phi G$ is $\Psi_\Ac$-WIT$_1$ implies $P$ is $\Psi_\Ac$-WIT$_1$.  Thus the above short exact sequence is taken by $\Psi[1]$ to the short exact sequence in $\Ac$
\[
0 \to \Psi P [1] \to G \to \Psi Q [1] \to 0.
\]
By the semistability of $G$ and Lemma \ref{lem:seesawmiddlenonzero}, it follows that 
\[
\phi_\Ac (\Psi P [1]) \preceq \phi_\Ac (G) \preceq \phi_\Ac (\Psi Q [1])
\]
which, by condition (f), implies 
\[
\phi_\Bc ( P ) \preceq \phi_\Bc (\Phi G) \preceq \phi_\Bc (Q ).
\]

Next, suppose  $P \in \Bc_{\kernel Z_\Bc}$.  Since $\Bc_{\kernel Z_\Bc}$ is a Serre subcategory of $\Bc$, condition (g) implies that $P$ is a direct sum of $Z_\Bc$-stable objects in $\Bc_{\kernel Z_\Bc}$.  Let $P'$ be any such $Z_\Bc$-stable direct summand.  Then by condition (g) again, $\Psi P'[1]$ is $Z_\Ac$-semistable in $\Ac$.  Let $\gamma$ denote the composite injection $P' \to P \to \Phi G$ in $\Bc$.  Then $\gamma$  is taken by $\Psi [1]$ to a nonzero map $\Psi \gamma [1] : \Psi P' [1] \to G$, and so we must have $\phi_\Ac (\Psi P'[1]) \preceq \phi_\Ac (G)$ by Lemma \ref{lem:tech1}.  Then, by condition (f), we have $\phi_\Bc (P') \preceq \phi_\Bc (\Phi G)$.  Since this holds for any $Z_\Bc$-stable direct summand $P'$ of $P$, it follows  by the weak seesaw property that $\phi_\Bc (P) \preceq \phi_\Bc (\Phi G)$.  We will now show   $\phi_\Bc (P) \preceq \phi_\Bc (Q)$.  Since $P \in \Bc_{\kernel Z_\Bc}$ while $\Phi G \notin \Bc_{\kernel Z_\Bc}$, it follows that $Q \notin  \Bc_{\kernel Z_\Bc}$ and that  $\phi_\Bc (\Phi G) = \phi_\Bc (Q)$.  Now we have two cases:
\begin{itemize}
\item If $\phi_\Bc (P) = \phi_\Bc (\Phi G)$, then $\phi_\Bc (P) = \phi_\Bc (Q)$.
\item If $\phi_\Bc (P) \prec \phi_\Bc (\Phi G)$, then $\phi_\Bc (P) \prec \phi_\Bc (Q)$.
\end{itemize}
Overall, we have $\phi_\Bc (P) \preceq \phi_\Bc (Q)$, and by Proposition \ref{prop:weakstabconcorr-var1}, $\Phi G$ is $Z_\Bc$-semistable.

\textbf{Step 5.} We can now construct the $Z_\Bc$-HN filtration of $E$.  By Steps 1 and 3, the filtration \eqref{eq:filt1} is taken by $\Phi [1]$ to a filtration of $E_0$ in $\Bc$
\begin{equation}\label{eq:filt1transf}
0 = \wh{M_0} \subsetneq \wh{M_1} \subsetneq \cdots \subsetneq \wh{M_m} = E_0
\end{equation}
where each factor $\wh{M_i}/\wh{M_{i-1}}$ is $Z_\Bc$-semistable; similarly, by Steps 2 and 4,   \eqref{eq:filt2} is taken by $\Phi$ to the filtration
\begin{equation}\label{eq:filt2transf}
0 = \wh{N_0} \subsetneq \wh{N_1} \subsetneq \cdots \subsetneq \wh{N_n} = E_1
\end{equation}
where each factor $\wh{N_i}/\wh{N_{i-1}}$ is $Z_\Bc$-semistable.

By condition (f) and the fact that \eqref{eq:filt1} and \eqref{eq:filt2} are both HN filtrations, we have
\[
\phi_\Bc (\wh{M_1}) \succ \phi_\Bc (\wh{M_2}/\wh{M_1}) \succ \cdots \succ \phi_\Bc (\wh{M_m}/\wh{M_{m-1}})
\]
and
\[
\phi_\Bc (\wh{N_1}) \succ \phi_\Bc (\wh{N_2}/\wh{N_1}) \succ \cdots \succ \phi_\Bc (\wh{N_n}/\wh{N_{n-1}}).
\]
By condition (e) and Lemma \ref{lem:Lo20Thm5-6part-i}, regardless of whether $\wh{N_1}$ lies in $\Bc_{\kernel Z_\Bc}$, we have
\[
\phi_\Bc (\wh{M_m}/\wh{M_{m-1}}) \succ \Gamma_T(a)+1 \succeq \phi_\Bc (\wh{N_1}).
\]
Therefore, the concatenation of \eqref{eq:filt1transf} and \eqref{eq:filt1transf} is the HN filtration of $E$.
\end{proof}

\section{Answer to the question for $D_\alpha \geq 0$ and proof of Theorem 1.4}\label{sec:Dalphageqzero}

In this section, we study Question \ref{que:main-1} for a line bundle $L=\cO_X(\alpha)$, where $$\alpha=\Theta+(D_\alpha+e)f$$ with $D_\alpha\geq0$.   {When working with weak stability condition, we will follow the notation in Section \ref{sec:defwsc} and more generally in \cite{CLSY1}.
Let $\sigma_{\omega, B}=(Z_{\omega, B}, \Coh^{\omega, B})$ be a Bridgeland stability condition of standard form, with the central charge defined by
	\begin{equation}
		\label{Equ:CenCharNoTd}
		Z_{\omega,B}(E)=-\int e^{-i\omega} \ch^B(E).
	\end{equation}
	Consider $\omega\in \langle \Theta, f \rangle$.
In this Section we prove the following claim: $\cO_X(\alpha)$ is $\sigma_{\omega, 0}$-stable for any $\omega\in \langle \Theta, f\rangle$ with $D_\omega>0$ and $V_\omega>0$.
    
    The main idea of the proof is applying Proposition \ref{prop:weakstabconcorr-var1} to weak stability conditions under the action of autoequivalences $$\Psi:=\Phi(\_\otimes \cO_X(-\alpha))$$ and $$\Upsilon:=\Phi(\_)\otimes \cO_X(f).$$
    
    In particular, we first show that $\Psi(\cO_X(\alpha))$ is $\sigma_a^R$-semistable for all $a$. Then Proposition \ref{prop:weakstabconcorr-var1} implies that $\cO_X(\alpha)$ is $\sigma_b^L$-semistable. Using the going up lemma for weak stability conditions, we conclude $\cO_X(\alpha)$ is $\sigma_{V, H}$-stable for all $V$. 
    
    Then we show that $\Upsilon(\cO_X(\alpha))$ is $\sigma_D$-stable for all $D$. Using the going up lemma along the ray of a fixed $D$ and $V\geq0$, we conclude $\Upsilon(\cO_X(\alpha))$ is $\sigma_{\omega, 0}$-stable for any $\omega\in \langle\Theta, f \rangle$ with $D_\omega>0$ and $V_\omega>0$. Then the Claim follows from Proposition \ref{prop:AG52-20-1} (specifically, the form in \eqref{eq:AG52-21-2}).
    
    We begin by studying the image of some objects under the relative Fourier-Mukai transform $\Phi$. Let $X$ be a Weierstra{\ss} elliptic K3 surface with a fixed section $\Theta$. Let $p:X\rightarrow B$ be the projection to the base $B$. Denoting the line bundle $R^1p_*\cO_X\simeq (p_*\omega_{X/B})^*$ by $w$. 
	Recall the relative Fourier-Mukai transform $\Phi$ has kernel:
	\begin{equation*}
		P=I_\Delta\otimes \pi_1^*\cO_X(\Theta)\otimes \pi_2^*\cO_X(\Theta)\otimes \rho^*w^{-1}.
	\end{equation*}
	Here $\pi_i: X\times_B X\rightarrow X$ is the projection to the $i$-th factor, and $\rho: X\times_BX\rightarrow B$. Let $e=\Theta^2=-2$.
 
	\begin{prop}
		\label{Prop:ImageO}
  We have
 
	(i)		$\Phi(\cO_X)=\cO_{\Theta}(-2)[-1]$,

        (ii) $\Phi(\cO_X(-\Theta))=\cO_X(\Theta)[-1]$.
	\end{prop}
	\begin{proof}
 		First note that $$\rho^*w^{-1}\simeq \pi_2^*\omega_{X/B}\simeq \pi_2^*\cO_X(2f).$$ For (i), apply the functor $R\pi_{1*}(\pi_2^*\cO_X(\Theta+2f)\otimes\_)$ to the exact sequence:
		\begin{equation*}
			0\rightarrow I_\Delta\rightarrow \cO_{X\times_BX}\rightarrow \cO_\Delta\rightarrow 0.
		\end{equation*}
		We get the following exact triangle:
		\begin{equation}\label{Equ:Phi(OX)}
			R\pi_{1*}(\pi_2^*\cO_X(\Theta+2f)\otimes I_\Delta)\rightarrow R\pi_{1*}(\pi_2^*\cO_X(\Theta+2f))\xrightarrow{f} R\pi_{1*}(\pi_2^*\cO_X(\Theta+2f)\otimes \cO_\Delta)\xrightarrow{+1}.
		\end{equation}
		We have $R\pi_{1*}(\pi_2^*\cO_X(\Theta+2f)\otimes \cO_\Delta)\simeq \cO_X(\Theta+2f)$. Also 
		\begin{equation*}
			\begin{split}
				R\pi_{1*}\pi_2^*\cO_X(\Theta+2f)&\simeq p^*Rp_{*}\cO_X(\Theta+2f)\\
				&\simeq \cO_X(2f)\otimes p^*(Rp_*\cO_X(\Theta)).
			\end{split}
		\end{equation*}
		Using the exact sequence 
		\begin{equation*}
			0\rightarrow \cO_X\rightarrow \cO_X(\Theta)\rightarrow \cO_\Theta(-e)\rightarrow 0,
		\end{equation*}
		by base change we have $R^1p_*\cO_X(\Theta)=0$, hence we have
		\begin{equation*}
			0\rightarrow p_*\cO_X\rightarrow p_*\cO_X(\Theta)\rightarrow \cO_B(-e)\rightarrow R^1p_*\cO_X\rightarrow 0.
		\end{equation*}
		 Since $R^1p_*\cO_X\simeq \cO_B(-2)$, we have $p_*\cO_X(\Theta)\simeq p_*\cO_X\simeq \cO_B$. 
		
 Then $R\pi_{1*}(\pi_2^*\cO_X(\Theta+2f)\otimes I_\Delta)\simeq \text{Cone}(\cO_X(2f)\to \cO_X(\Theta+2f))[-1]$. By Lemma \ref{lem:lbfdonetr} (ii), we have $\Phi(\cO_X)[1]$ is a torsion sheaf. Hence $R\pi_{1*}(\pi_2^*\cO_X(\Theta+2f)\otimes I_\Delta)\simeq \cO_\Theta(\Theta+2f)[-1]$, and
		\begin{equation*}			\Phi(\cO_X)=\cO_X(\Theta)\otimes \cO_\Theta(\Theta+2f)[-1]\simeq \cO_\Theta(2\Theta+2f)[-1]\simeq \cO_\Theta(-2)[-1].
		\end{equation*}
For (ii), by Lemma \ref{lem:lbfdonetr}, we know that $\Phi\cO_X(-\Theta)[1]$ is a line bundle. Then by the Chern character formula for relative Fourier-Mukai transform, see \cite[(6.21)]{FMNT}, we have 
  \[
 \ch_1(\Phi(\cO_X(-\Theta)))=-\Theta, 
 \]
 \[
 \ch_2(\Phi(\cO_X(-\Theta)))=1\cdot [\text{pt}].
 \]
Hence $\Phi(\cO_X(-\Theta))=\cO_X(\Theta)[-1]$.
 \end{proof}

	\begin{thm}
		\label{Thm:OTheta(-2)stable}
  Given $a\in\mathbb{P}^1$, if $D_\alpha+1+m\geq -1$,  we have $\cO_{\Theta}(m)$ is $\sigma^R_a$-stable.
	\end{thm}

Recall from 5.23 in \cite[4.23]{CLSY1}
 that 
we denote $L_0=\cO(-(D_\alpha+1)f)$ and $L_1=\cO(\Theta-(D_\alpha+2)f)$.

	\begin{proof}
		Consider the following number in \cite[Exercise 6.19]{MSlec}:
		$$C_{\omega'_0}=\text{min}\{\ch_1^{B'_0}(E)\cdot \omega'_0|\ch_1^{B'_0}(E)\cdot \omega'_0>0\}.$$
		We have $C_{\omega'_0}=\frac{1}{2},$ and 
		\begin{equation*}
			\Im(Z'_0(\cO_\Theta(m)))=\frac{1}{2}.
		\end{equation*} 
		
		Assume there is a short exact sequence in $\B$:
		\begin{equation*}
			0\rightarrow E\rightarrow \cO_\Theta(m)\rightarrow Q\rightarrow 0.
		\end{equation*}
		Then $E$ is a coherent sheaf on $X$, and $H^{-1}(Q)\in \F_{\omega'_0, B'_0}$. 
		
		\textbf{Case 1}: $\Im(Z'_0(E))=\ch_1^{B'_0}(E)\cdot \omega'_0>0$.  In this case,  $\Im(Z'_0(E))\geq \frac{1}{2}$.  If $\Im(Z'_0(E))> \frac{1}{2}$, 	then
		\begin{equation*}
			\ch_1^{B'_0}(H^{-1}(Q))\cdot \omega'_0=\ch_1^{B'_0}(E)\cdot \omega'_0-\Theta\cdot \omega'_0>0
		\end{equation*} 
		contradicting  $H^{-1}(Q)\in \mathcal{F}_{\omega'_0, B'_0}$.  So we must have  $\Im(Z'_0(E))=\frac{1}{2}$, in which case $\Im Z_0'(Q)=0$, which implies  $\Im(Z'_0(H^{-1}(Q)))=0= \Im (Z'_0 (H^0(Q)))$. 
		By Serre duality, we have $$\Ext^1(\cO_\Theta(m), L_0)\simeq H^1(\cO_\Theta(D_\alpha+1+m))^*=0$$ for $D_\alpha+1+m\geq -1$. Now we further divide into  two cases:
		
		(i) $H^{-1}(Q)\notin \kernel (Z'_0)$.  In this case, $Q$ itself does not lie in $\kernel(Z'_0)$, and so the observation $\Im Z_0'(Q)=0$  above implies $\phi_{\sigma_a}(Q)=1$.
		
		(ii) $H^{-1}(Q)\in \kernel (Z'_0)$.  In this case we must have  $H^0(Q)\neq 0$.  For, if $H^0(Q)=0$, then $Q = H^{-1}(Q)[1]$ is an object of $\B_{\kernel (Z_0')}$ which implies $Q \cong (\oplus^m L_0)[1]$ for some $m>0$.  Then the Ext-vanishing above means $E \cong \OO_\Theta (-2) \oplus (\oplus^m L_0)$.  Since $\Hom (\OO_\Theta (m),L_1)=0=\Hom (L_0,L_1)$, we have $\Hom (E, \T_\A [-1])=0$ and hence $E \in \F_\A$, i.e.\ $E \in \A \cap \Coh (X)=\T_{\omega_0', B_0'}$.  Hence any quotient sheaf of $E$ must also lie in $\T_{\omega_0', B_0'}$; in particular, it follows that $L_0$ lies in $\T_{\omega_0', B_0'}$ which is false. 

         Now that we know $H^0(Q)\neq 0$, 
by the long exact sequence of cohomology, we have $H^0(Q)$ is a quotient sheaf of $\cO_\Theta(m)$, hence is supported in dimension $0$. Thus in (ii) we also have $\phi_{\sigma_a}(Q)=1$. In either (i) or (ii), $E$ does not destabilize $\cO_\Theta(m)$. 
  

		\textbf{Case 2}:  $\Im(Z'_0(E))=0$.  In this case, $E$ fits in a short exact sequence of coherent sheaves
		$$0\rightarrow T\rightarrow E\rightarrow F\rightarrow 0$$
        where $T \in \F_\A$ and $F \in \T_\A[-1]$.  Since we must have $\Im Z_0'(T)=0$ and $T \in \F_\A \cap \Coh (X) \subseteq \T_{\omega_0', B_0'}$, the sheaf $T$ must be torsion.

		If $T\neq 0$, then $T$ supports on dimension $0$, and so $E$ has a nonzero 0-dimensional subsheaf; however, this is impossible since  $H^{-1}(Q)\in \F_{\omega'_0, B'_0}$ is torsion-free while $\OO_\Theta (m)$ is pure 1-dimensional.
		Hence $T=0$, then $E\in \T_\A[-1]$. Then $\phi_{\sigma_a}(E)=\frac{1}{4}$. On the other hand 
		$$Z'_0(\cO_\Theta(m))=-(m+1)-D_\alpha-\frac{5}{2}+\frac{1}{2}i.$$
		When $D_\alpha+1+m\geq -1$, $\phi_{\sigma_a}(Q)=\phi_{\sigma_a}(\cO_\Theta(m))>\frac{1}{2}$. Hence $E$ does not destabilize $\cO_\Theta(m)$. 
	\end{proof}

\begin{prop}
	\label{Prop:CorrKer}
	We have the following correspondence involving objects in the kernels of $Z_{V, H}$ and $Z_D$:
		\begin{equation*}
		\Phi(\cO_\Theta(-1)[1]\otimes \cO_X(-\alpha))=L_0[1],
	\end{equation*} and 
	\begin{equation*}
		\Phi(\cO_X[1]\otimes \cO_X(-\alpha))=L_1.
	\end{equation*}
	\end{prop}

\begin{proof}

By  \cite[Corollary 5.9]{Lo7}, we have $\Phi(\cO_\Theta(m))$ is a line bundle on $X$ with 
\[
\ch_1(\Phi(\cO_\Theta(m)))=(\ch_2(\cO_\Theta(m))-1)f.
\]
Since $\ch_2(\cO_\Theta(m))=m+1$, we have 
\begin{equation}
	\begin{split}
	\Phi(\cO_\Theta(-1)[1]\otimes \cO_X(-\alpha))&=\Phi(\cO_\Theta(-D_\alpha-1)[1])\\
	&\simeq \cO_X(-(D_\alpha+1)f)[1].
	\end{split}
\end{equation}
For the other equality, we have 
\begin{equation*}
	\begin{split}
	\Phi(\cO_X[1]\otimes \cO_X(-\alpha))&=\Phi(\cO_X(-\Theta-(D_\alpha+e)f)[1])\\
	&\simeq \Phi(\cO(-\Theta))[1]\otimes \cO_X(-(D_\alpha+e)f)\\
	&\simeq \cO_X(\Theta)\otimes \cO_X(-(D_\alpha+e)f) \\
	&\simeq \cO_X(\Theta-(D_\alpha+e)f),
	\end{split}
\end{equation*}
	where in the second line, $\Phi$ commutes with $\otimes \cO(f)$ because of the projection formula and $\Phi$ being a relative integral functor, while 
  the third line follows from  Proposition \ref{Prop:ImageO}.
\end{proof}

\begin{prop}
	\label{Prop:Ofstableweak}
	Given an arbitrary integer $m$, we have 
	
	(i) $\cO_f\otimes\cO(m\Theta)$ is stable in $\sigma_{V, H}$ for $V>0$,
	
	(ii) $\cO_f\otimes\cO(m\Theta)$ is stable in $\sigma^L_b$ for any $b\in\mathbb{P}^1$.
\end{prop}
\begin{proof}
We first prove (i).
	Let $C_H=\text{min}\{H\cdot \ch_1(E)| E\in \mathcal{D}, H\cdot \ch_1(E)>0\}$. Since $H$ is integral, $C_H=1$. 
	
	Note that $H\cdot \ch_1(\cO_f\otimes\cO(m\Theta))=1$ for any $m$. Assuming that there is a short exact sequence in $\B^0_{H, -2}$:
	\begin{equation*}
		0\rightarrow E\rightarrow \cO_f\otimes\cO(m\Theta)\rightarrow Q\rightarrow 0
	\end{equation*}
such that $\phi_{V, H}(E)\geq\phi_{V, H}(Q)$, 
We have $E$ fits into a short exact sequence in $\B^0_{H, -2}$
\begin{equation*}
	0\rightarrow F\rightarrow E\rightarrow T\rightarrow 0
\end{equation*}  
where $F\in \F^0_{H, -2}[1]$, and $T\in \T^0_{H, -2}$. Since $E$ is a coherent sheaf, $F$ must vanish and so  we have $E\in \T^0_{H, -2} \cap \Coh (X)$, and hence $E\in \T^0_H$. 

Now we have
 either $H\cdot \ch_1(E)>0$ or  $H\cdot \ch_1(E)=0$.  In the latter case, $E$ must be a torsion sheaf; but it cannot be supported on dimension $0$ or else $\Hom (E,\OO_f \otimes \OO(m\Theta))=0$, so $E$ must be a pure 1-dimensional sheaf; in fact, $E$ must be supported on $\Theta$ by Lemma 
 \cite[Lemma 5.2]{CLSY1},
 and so  $E\in \langle \cO_\Theta(i)\rangle$.  However, we also have $\Hom(\cO_\Theta(i), \cO_f\otimes\cO(m\Theta))=0$, and so we must have $H\cdot \ch_1(E)>0$ instead, in which case  $H\cdot \ch_1(E)\geq 1$.


Consider the case  $H\cdot \ch_1(E)>1$.  As a morphism in $\Coh (X)$, the image of the morphism $E \to \OO_f \otimes \OO (m\Theta)$ must be nonzero and supported on $f$, and so  $H\cdot \ch_1(H^{-1}(Q))>0$, contradicting that $$H\cdot \ch_1(H^{-1}(Q))\leq 0$$ 
which comes from $H^{-1}(Q) \in \F^0_H$.

We are left to consider $H\cdot \ch_1(E)=1$. Then $H\cdot \ch_1(Q)=0$. If $Q\notin \ker(Z_{V, H})\cap \B^0_{H, k}$, then $Q$ does not destabilize $\cO_f\otimes\cO(m\Theta)$. 

Assume that $Q\in \ker(Z_{V, H})\cap \B^0_{H, -2}$. Then $Q\in\langle \cO_\Theta(-1)\rangle$, then  \[
\Hom(\cO_f\otimes \cO(m\Theta), \cO_\Theta(-1))=0
\]
 implies that $Q$ is zero.

For (ii), the same argument carries over until the last step. Assume that $Q\in \ker(Z_{H, -2})\cap \B^0_{H, -2}$, then by the same reasoning in (i), we have $H\cdot \ch_1(Q)=0$, the image of $E \to \OO_f \otimes \OO (m\Theta)$ as a morphism in $\Coh (X)$ must be supported on $f$, implying $H^0(Q)$ is supported in dimension 0.  From our description of $\kernel (Z_{H,-2})\cap \B^0_{H,-2}$ in Proposition 
 \cite[5.4]{CLSY1}, it follows that  $H^0(Q)=0$ and $Q\in\langle \cO_X[1]\rangle$. But $\phi_b(\cO_X[1])=1$, hence $Q$ does not destabilize $\cO_f\otimes\cO(m\Theta)$.
\end{proof}

\begin{rem}
\label{rem:Qstableweak}
The argument of Proposition \ref{Prop:Ofstableweak} only depends on $\cO_f$ being pure and the Chern character of $\cO_f\otimes \cO(m\Theta)$, hence we can replace $\cO_f$ by any pure sheaf $F$ such that $\ch(F)=\ch(\cO_f)$.
\end{rem}

We denote the category  $\langle O_\Theta(j)|j\leq k\rangle$ by $\mC^0_k$, and the category $\langle L_1\rangle$ by $\mC_1$. Recall that $\A^0_H$ is the abelian category obtained by tilting $\Coh(X)$ at the standard torsion pair $(\T^0_H, \F^0_H)$.
\begin{prop}
	$\B^0_{H, k}$ is obtained by tilting $\Coh(X)$ at the torsion pair
	\begin{equation*}
		\begin{split}
			\T_1&:=\{E\in \Coh(X)|E\in \T^0_H \text{ and } \Hom_{\A^0_H}(E, \mC^0_k)=0\}\\
			\F_1&:=\{E\in \Coh(X)|E\in \langle\F^0_H, \mC^0_k\rangle\}.
		\end{split}
	\end{equation*}
	\end{prop}
\begin{proof}
	We first show that $(\T_1, \F_1)$ indeed defines a torsion pair on $\Coh(X)$. It is obvious that 
 \[
 \Hom(\T_1, \F_1)=0.
 \]
	We have that any $E$ in $\Coh (X)$ fits into a short exact sequence
	\begin{equation*}
		0\rightarrow T\rightarrow E\rightarrow F\rightarrow 0
	\end{equation*}
	such that $T\in \T^0_H$ and $F\in \F^0_H$. Then $T\in \A^0_H$. Since $\mC^0_k$ is a torsion-free class in $\A^0_H$, we have $T$ fits into a short exact sequence in $\A^0_H$
	\begin{equation*}
		0\rightarrow T'\rightarrow T\rightarrow C\rightarrow 0
	\end{equation*}
	where $T'\in \T^0_H$ because $T$ is a sheaf, and $C\in \mC^0_k$ and $\Hom(T', \mC^0_k)=0$. Then  we have the following short exact sequence:
	\begin{equation*}
		0\rightarrow T'\rightarrow E\rightarrow F'\rightarrow 0
	\end{equation*}
	where $F'$ is an extension of $C$ and $F$. Hence $T'\in \T_1$ and $F'\in \F_1$.
	
	Let $E\in \B^0_{H, k}$, then we have 
	\begin{equation*}
		0\rightarrow C\rightarrow E\rightarrow A\rightarrow 0
	\end{equation*}
where $C\in \mC^0_k[1]$, $A\in \A^0_H$ and $\Hom_{\A^0_H}(A, \mC^0_k)=0$.
Note that we have the short exact sequence in $\A^0_H$:
\begin{equation*}
	0\rightarrow F_A\rightarrow A\rightarrow T_A\rightarrow 0
\end{equation*}
where $F_A\in\F^0_H[1]$ and $T_A\in \T^0_H$. Furthermore since $\Hom_{\A^0_H}(A, \mC^0_k)=0$, we have $\Hom_{\A^0_H}(T_A, \mC^0_k)=0$. This implies that $T_A\in\T_1$. Then we have $$E\in \langle\F_1[1], \T_1\rangle.$$ Hence $\B^0_{H, k}=\langle\F_1[1], \T_1\rangle$.
	\end{proof}

Recall that $\A$ is the abelian category obtained by tilting $\Coh(X)$ at the standard torsion pair $(\T_{\omega'_0, B'_0}, \F_{\omega'_0, B'_0})$, and that $\B$ is the tilt of $\A$ at the torsion pair $(\T_\A, \F_\A)$ in $\A$ from  \cite[Proposition 5.25]{CLSY1} 
but with a shift, so that $\B = \langle \F_\A, \T_\A [-1]\rangle$.  Here, $\mC_1 = \langle L_1 \rangle=\T_\A[-1]$.

\begin{prop}
	$\B$ is obtained by tilting $\Coh(X)$ at the torsion pair
		\begin{equation*}
		\begin{split}
			\T'_1&:=\{E\in \Coh(X)|E\in\langle \T_{\omega'_0, B'_0}, \mC_1\rangle\}\\
			\F'_1&:=\{E\in \Coh(X)|E\in \F_{\omega'_0, B'_0} \text{ and } \Hom_{\A}(\mC_1, E)=0\}.
		\end{split}
	\end{equation*}
\end{prop}
\begin{proof}
	We first show that $(\T'_1, \F'_1)$ defines a torsion pair in $\Coh (X)$. It is easy to see that $\Hom(\T'_1, \F'_1)=0$. Given $E\in \Coh(X)$, we have $E$ fits into the short exact sequence 
	\begin{equation*}
		0\rightarrow T\rightarrow E\rightarrow F\rightarrow 0
		\end{equation*}
	with $T\in \T_{\omega'_0, B'_0}$ and $F\in \F_{\omega'_0, B'_0}$. Since $\mC_1[1]=\T_\A$ is a torsion class in $\A$, we have $F[1]$ fits in a short exact sequence 
	\begin{equation*}
		0\rightarrow C[1]\rightarrow F[1]\rightarrow F'[1]\rightarrow 0
	\end{equation*}
where $C\in \mC_1$ and $F'[1]\in \A$ such that $\Hom(\mC_1, F')=0$.  Note that $F \in \Coh (X)$ implies $F' \in \Coh (X)$ and hence $F' \in  \F_{\omega'_0, B'_0}$.  Then we have the short exact sequence 
$$0\rightarrow E'\rightarrow E\rightarrow F'\rightarrow 0,$$
where $E'$ is an extension of $T$ and $C$.
Hence $E'\in \T'_1$, and $F'\in \F'_1$, and we have $(\T'_1, F'_1)$ defines a torsion pair in $\Coh(X)$. 

Let $E\in \B$, there exists a short exact sequence 
\begin{equation*}
	0\rightarrow A\rightarrow E\rightarrow C\rightarrow 0
\end{equation*}
where $A\in \F_{\A}$ and $C\in\T_\A[-1]=\mC_1$ from the definition of $\B$.  Also, $A$ fits into a short exact sequence in $\A$
\begin{equation*}
	0\rightarrow F_A\rightarrow A\rightarrow T_A\rightarrow 0
	\end{equation*}
where $F_A\in\F_{\omega'_0, B'_0}[1]$ and $T_A\in \T_{\omega'_0, B'_0}$. Since $\Hom_{\A}(\mC_1[1], A)=0$, we have $\Hom_{\A}(\mC_1[1], F_A)=0$. Hence $E\in \langle\F'_1[1], \T'_1\rangle$, which implies that $$\B= \langle\F'_1[1], \T'_1\rangle.$$
\end{proof}

Recall that $\Psi=\Phi(\_\otimes \cO_X(-\alpha))$. Let $g\in \mathrm{GL}^+(2, \mathbb{R})$ be defined by 
	\begin{equation*}
	g:=\begin{pmatrix}
		1 & 0\\ 
		0 & R_{\omega'_0}
	\end{pmatrix}		
	\begin{pmatrix}
		1 & R_{B'_0}\\ 
		0 & 1
	\end{pmatrix}	
	\begin{pmatrix}
		0 & 1\\ 
		-1 & 0
	\end{pmatrix}
	\begin{pmatrix}
		1 & R_\alpha\\ 
		0 & 1
	\end{pmatrix}.
\end{equation*}
Since $R_{\omega'_0}=\frac{1}{2}$, $R_{B'_0}=\frac{1}{2}$ and $R_\alpha=1$, we have
	\begin{equation*}
	g:=\begin{pmatrix}
		-1/2 & 1/2\\ 
		-1/2 & -1/2
	\end{pmatrix},	
\end{equation*}
and
\begin{equation}\label{eq:AG52-17-2}
	Z'_0(\Psi(E))=g Z_{H}(E).
\end{equation}
Note that $g$ can be view as a clockwise rotation of $\frac{3}{4}\pi$.
To apply Proposition \ref{prop:weakstabconcorr-var1}, we first explore the relation between $\B^0_{H, -2}$ and $\B$. 

Recall that if $\Pc$ is a slicing on a triangulated category $\Tc$ in the sense of \cite[Definition 3.3]{StabTC} and $\Lambda$ is an autoequivalence of $\Tc$, then we can define a new slicing $\Lambda \cdot \Pc$ by setting $(\Lambda \cdot \Pc)(t) = \Lambda (\Pc (t))$ for any $t \in \mathbb{R}$.

\begin{thm}
	\label{Thm:CorrHeart1}
	Give $b\in \mathbb{P}^1$, let $\cP_b$ denote the slicing for the weak stability condition $\sigma_b^L$, and let $\Pc = \Psi \cdot \Pc_b$.  Denote the heart $\cP((\frac{3}{4}, \frac{7}{4}])$ by $\B'_{H, -2}$. Then we have  $$\B'_{H, -2}=\B.$$
	\end{thm}
	
\begin{rem}\label{rem:AG52-17-1}
From the relation \eqref{eq:AG52-17-2}, we have 
\[
  Z'_0(E)= gZ_H(\Psi^{-1}(E))
\]
for all $E \in D^b(X)$.  Since the matrix $g$ decreases the phases of complex numbers by $\tfrac{3}{4}\pi$, it follows that for any $E \in \Bc'_{H,-2}$, the complex number $Z'_0(E)$ lies on the upper-half complex plane.  Consequently, any $E \in \Bc'_{H,-2}$ satisfies \[
\Im Z'_0(E) = \omega'_0 \ch_1^{B'_0}(E) \geq 0.
\]
\end{rem}

\begin{lem}\label{lem:AG52-17-2}
We have that $(\kernel (Z_H \circ \Psi^{-1})) \cap \Bc'_{H,-2}$ is generated by $L_1$ and $L_0[1]$ via direct sum.
\end{lem}

\begin{proof}
Recall from 
\cite[Section 5.6]{CLSY1} that $\OO_X[1]$ and $\OO_\Theta (-1)$, which are the generators of $(\kernel Z_H) \cap \Bc^0_{H,-2}$, have phases $1$ and $\tfrac{1}{2}$ in $\sigma^L_b$ respectively.  Since $\OO_X[1]$ and  $\OO_\Theta (-1)$ are both $\sigma^L_b$-stable by  
 \cite[Remark 5.9]{CLSY1}, we have 
\begin{align*}
    L_1 &= \Psi(\OO_X[1]) \in \Pc (1), \\
    L_0 &=\Psi (\OO_\Theta (-1)) \in \Pc (\tfrac{1}{2})
\end{align*}
and so  $L_0[1], L_1 \in \Bc'_{H,-2}$.

Now, take any nonzero object $E \in \Bc'_{H,-2}=\Pc (\tfrac{3}{4}, \tfrac{7}{4}]$ such that $Z_H(\Psi^{-1} (E))=0$.  Let $\Pc_b$ denote the slicing for $\sigma_b^L$.  From the definition of $\Bc'_{H,-2}$, we know $\Psi^{-1}(E)$ lies in $\Pc_b (\tfrac{3}{4}, \tfrac{7}{4}]$, which is contained in $\langle \Bc^0_{H,-2}[1], \Bc^0_{H,-2}\rangle$.  That is, $\Psi^{-1} E$ fits in an exact triangle
\[
F[1] \to \Psi^{-1}E \to T \to F[2]
\]
where $F, T \in \Bc^0_{H,-2}$.  Since $Z_H(\Psi^{-1}(E))=0$ and $Z_H$ takes $\Bc^0_{H,-2}$ into a non-strict half plane, it follows that $F, T$ both lie in $\Bc^0_{H,-2} \cap \kernel (Z_H)$.  However, we have $F \in \Pc_b (0, \tfrac{3}{4}]$ and $T \in \Pc_b (\tfrac{3}{4},1]$, and so by \cite[Remark 5.9]{CLSY1} 
and the phase computations in 
\cite[Section 5.6]{CLSY1}, we have that $F \in \langle \OO_\Theta (-1)\rangle$ and $T \in \langle \OO_X [1] \rangle$.  So $\Psi^{-1}E \in \langle \OO_\Theta (-1)[1], \OO_X [1] \rangle$ which implies $E \in \langle L_0[1], L_1 \rangle$.

Moreover, since $L_0[1] \in \Pc (\tfrac{3}{2})$ and $L_1 \in\Pc (1)$, and
\begin{align*}
    \Ext^1(L_0[1], L_1)& \cong \Ext^1( \OO_\Theta (-1),\OO_X) = 0, \\
    \Ext^1(L_1, L_0[1]) &\cong \Ext^1 (\OO_X, \OO_\Theta (-1))=0,
\end{align*}
it follows that $E$ satisfies $Z_H (\Psi^{-1}(E))=0$ if and only if $E$ is a direct sum of copies of $L_0[1]$ and $L_1$.
\end{proof}

\begin{proof}[Proof of Theorem \ref{Thm:CorrHeart1}]

	 We follow the argument in the proof of \cite[Lemma 6.20]{MSlec}.
  
 
\textbf{Part 1.} We have $$\C(x)=\Psi(Q\otimes\cO_X(\alpha)[1]),$$ 
	where $Q$ is a pure sheaf supported in dimension $1$ with $\ch(Q)=\ch(\cO_f)$. For any $k\in \mathbb{Z}$, $Q\otimes \cO_X(kf)$ is also a pure sheaf supported in dimension $1$ with $\ch(Q \otimes \OO_X (kf))=\ch(\cO_f)$. From Proposition \ref{Prop:Ofstableweak} and Remark \ref{rem:Qstableweak}, we know that $Q \otimes\OO_X(\alpha)[1]$ lies in $\Pc_b (\tfrac{7}{4})$ and is $\sigma_b^L$-stable.
 Recall in Remark \ref{Rem:stableness}, we say an object is $\sigma^L_b$-stable in $\cP_b(t)$ if it is a shift of stable object in $\cP_b(t')$ for some $0<t'\leq 1$. We denote $Q\otimes \cO_X(\alpha)$ by $Q'$. 
 
Claim: $Q'[1]$ is also stable in $\cP_b((\frac{3}{4}, \frac{7}{4}])$. 

Note that the objects in $\ker(Z_H)\cap \cP_b((\frac{3}{4}, \frac{7}{4}])$ are direct sums of $\cO_X[1]$ and $\cO_\Theta(-1)[1]$ by Lemma \ref{lem:AG52-68-1b} below. We can define the phases of these objects by taking the limits of $\phi_{V_\omega, D_\omega}$ as in Section 9.1. Then the phases of objects in $\ker(Z_H)\cap\cP_b((\frac{3}{4}, \frac{7}{4}])$ are in $[1, \frac{3}{2}]$.

Assume the claim is not true, we have a short exact sequence 
\begin{equation}
    \label{Equ:12.80}
0\to E\to Q'[1]\to F\to 0
\end{equation}
in $\cP_b((\frac{3}{4}, \frac{7}{4}])$ such that $\phi_b (E) \geq \phi_b (F)$.  Since the phases of objects in $\Pc_b ( (\tfrac{3}{4},\tfrac{7}{4}]) \cap \kernel Z_H$ are defined as limits of $\phi_{V_\omega, D_\omega}$, the weak seesaw property still holds for all objects in $\Pc_b ((\tfrac{3}{4}, \tfrac{7}{4}])$. It follows that $\phi_b(E) \geq \phi_b (Q'[1])=\tfrac{7}{4}$.    By Lemma \ref{lem:AG52-68-1}, we must have $\phi_b(E)=\frac{7}{4}$.

Since $\cP_b$ defines a slicing, there exists a subobject $E'$ of $E$ in $\cP_b((\frac{3}{4}, \frac{7}{4}])$ such that $E'\in \cP_b(\frac{7}{4})$.
We form the short exact sequence in $\cP_b((\frac{3}{4}, \frac{7}{4}])$:
\[
0\to E'\to Q'[1]\to F'\to 0.
\]

Suppose $\phi_b(F')\neq \frac{7}{4}$.  Note that $E' \in \Pc_b(\tfrac{7}{4})$ implies $E'$ is the shift of a semistable object in $\Pc_b (0,1]$ of phase $\tfrac{3}{4}$, and so $E' \notin \kernel Z_H$.  Hence  $F'$ must lie in $\ker(Z_H)\cap \cP_b((\frac{3}{4}, \frac{7}{4}])$ (or else we would have $\phi_b (E') \neq \phi_b (Q'[1])$).  By Lemma \ref{lem:AG52-68-1}, there must be a nonzero surjection $F' \to A$ in $\Pc_b ((\tfrac{3}{4}, \tfrac{7}{4}])$ for some $A \in \Pc_b (t)$ with $t < \tfrac{7}{4}$; this induces a nonzero map $Q'[1] \to A$, contradicting $\Pc_b$ being a slicing.

Hence we must have $\phi_b(F')=\frac{7}{4}$.  Then  either (i) $F'\in \cP_b(\frac{7}{4})$, or (ii) $F'\in \langle \ker{Z_H}\cap \cP_b((\frac{3}{4}, \frac{7}{4}]), \cP_b(\frac{7}{4})\rangle$. 
If we are in case (ii), there is a nonzero map from $Q'[1]$ to $\langle \ker(Z_H)\cap \cP_b((3/4, 7/4])\rangle$. On the other hand, 
\[
\langle \ker(Z_H)\cap \cP_b((3/4, 7/4])\rangle=\langle \cO_X[1], \cO_\Theta(-1)[1]\rangle,
\]
and there is no nonzero map from $Q'[1]$ to $\cO_X[1]$ or $\cO_\Theta(-1)[1]$, hence we must have $F'\in \cP_b(7/4)$.
Then the short exact sequence 
\[
0\to E[-1]\to Q'\to F[-1]\to 0
\]
is a destabilizing short exact sequence in $\cP_b((0, 1])$, which contradicts  $Q'$ being $\sigma_b^L$-stable.

Let $S(t)\subset \cP(t)$ be the subset consisting of objects in $\cP(t)$ which are images of $\sigma_b^L$-stable objects in $\cP_b(t)$. 

Claim: let $A\in S(t)$ for some $\frac{3}{4}<t\leq \frac{7}{4}$, we have 
\[
\Hom(\mathbb{C}(x), A)=0.
\]
 Let $A=\Psi(A')$ where $A'$ is $\sigma_b^L$-stable. If we have a nonzero map $\mathbb{C}(x)\to A$, this corresponds to a nonzero map $f:Q'[1]\rightarrow A'$. Consider the image of $f$ in $\cP_b((\frac{3}{4}, \frac{7}{4}])$. Since $Q'[1]$ is stable in $\cP_b((\frac{3}{4}, \frac{7}{4}])$, we have $\phi_b(\text{im}f)=\frac{7}{4}$. Then by Lemma \ref{lem:AG52-68-1}, there exists a subobject $C$ of $\text{im}f$ such that $C\in \cP_b(\frac{7}{4})$. We form the short exact sequence 
\begin{equation}
    \label{Equ:12.8equ1}
0\to C\to A'\to B\to 0
\end{equation}
in $\cP_b((\frac{3}{4}, \frac{7}{4}])$. If $t<\frac{7}{4}$, then $\phi_b(C)>\phi_b(A')$. Since $C\in \cP_b(\frac{7}{4})$, $A'\in \cP_b(t)$, this contradicts $\cP_b$ is a slicing. If $t=\frac{7}{4}$, then  \eqref{Equ:12.8equ1} is a short exact sequence with $C$ and $A'$ both in $\cP_b(\frac{7}{4})$. Then we have either (i) $B\in \cP_b(\frac{7}{4})$ which implies that $A'$ strictly $\sigma_b^L$-semistable, or (ii) $B\in\langle\cP_b(\frac{7}{4}), \ker(Z_H)\cap \cP_b((\frac{3}{4}, \frac{7}{4}])\rangle$. Case (ii) implies that there is a nonzero map from $A'$ to $\ker(Z_H)\cap \cP_b((\frac{3}{4}, \frac{7}{4}])$. Since the objects in $\ker(Z_H)\cap \cP_b((\frac{3}{4}, \frac{7}{4}])$ are direct sums of $ \cO_X[1]$ and $\cO_\Theta(-1)[1]$, and $\phi_b(\cO_X[1])=1$, $\phi_b(\cO_\Theta(-1)[1])=\frac{3}{2}$, this again contradicts $\cP_b$ being a slicing. 



On the other hand, if $K\in \ker(Z_H)\cap \cP_b(t)$ for $0<t\leq 1$, then $K\simeq \oplus \cO_\Theta(-1)$ or $K\simeq \oplus \cO_X[1]$. By Proposition \ref{Prop:CorrKer}, we know that $\Psi(\cO_\Theta(-1))=L_0$, $\Psi(\cO_X[1])=L_1$.
We have $\Hom(\C(x), L_i)=0$ for $i=0, 1$, and $\Hom(\C(x), L_i[1])=0$ for $i=0, 1$. 
Hence we have 
\[
\Hom(Q'[1], K)=0
\]
for any $K\in \ker(Z_H)\cap \cP_b((0, 1])$, or any $K\in \ker(Z_H)\cap \cP_b((1, 2])$. 
If $t\leq 1$, objects in $\cP(t)$ are iterated extensions of objects in $S(t)$ and objects in $\Psi(\ker(Z_H)\cap \cP_b((0, 1]))$. If $1<t\leq \frac{7}{4}$, objects in $\cP(t)$ are iterated extensions of objects in $S(t)$ and objects in $\Psi(\ker(Z_H)\cap \cP_b((1, 2]))$. Hence we have $\Hom(\C(x), E)=0$, for any $E\in\B'_{H, -2}$.

Combining the arguments in the last two paragraphs, and using  Serre duality, we now have that for any $E \in \Bc'_{H,-2}$ 
	\begin{equation*}
		\Ext^i(E, \C(x))=\Ext^{2-i}(\C(x), E)=0
	\end{equation*}
for all $i\neq 0, 1$. Then by \cite[Proposition 5.4]{BMef}, we have that $E$ is quasi-isomorphic to a complex of locally free sheaves $F^\bullet$ such that $F^i=0$ for $i>0$ or $i<-1$. This also implies that $H^{-1}(F^\bullet)$ is torsion-free.


As a result, we have $\B'_{H, -2}\subset \langle \Coh(X), \Coh(X)[1] \rangle$. Set
\begin{align*}
    \Tc_{\B'}&=\Coh(X)\cap \B'_{H, -2},\\
    \Fc_{\B'}&=\Coh(X)\cap (\B'_{H, -2}[-1]).
\end{align*}
Then we have $\Coh(X)=\langle \T_{\B'}, \F_{\B'}\rangle$, and $\B'_{H, -2}=\langle \F_{\B'}[1], \T_{\B'} \rangle$. We need to show that $\T_{\B'}=\T'_1$ and $\F_{\B'}=\F'_1$. It is enough to show that $\T'_1\subseteq \T_{\B'}$ and $\F'_1\subseteq \F_{\B'}$.  
 
 Note that, since $H^{-1}(F)$ is a torsion-free sheaf for any $F \in \Bc'_{H,-2}$ from above, every torsion sheaf on $X$ must be contained in $\Tc_{\Bc'}$.  

\textbf{Part 2.} In this part, we show that $\T'_1\subseteq\T_{B'}$, and for any $E$ which is $\mu_{\omega'_0, B'_0}$-stable torsion-free sheaf and $\mu_{\omega'_0, B'_0}(E)<0$, $E$ lies in $\F_{B'}$. Let $E\in \Coh(X)$ be a $\mu_{\omega'_0, B'_0}$-stable torsion-free sheaf. Then $E$ fits into a short exact sequence of sheaves
\begin{equation*}
	0\rightarrow T\rightarrow E\rightarrow F\rightarrow 0
\end{equation*}
where $T\in \T_{\B'}$ and $F\in \F_{\B'}$. Note that $T, F$ are both torsion-free, while we have  $\mu_{\omega'_0, B'_0}(T)\geq 0$ and $\mu_{\omega'_0, B'_0}(F)\leq 0$ from Remark \ref{rem:AG52-17-1}. Then  by the $\mu_{\omega'_0, B'_0}$-stability of $E$, either $T$ or $F$ vanishes, i.e.\ either  $E\in \T_{\B'}$ or $E\in \F_{\B'}$.   In particular, if  $\mu_{\omega'_0, B'_0}(E)>0$ (resp.\ $\mu_{\omega'_0, B'_0}(E)<0$), then $E$ must lie in $\T_{\B'}$ (resp.\ $\Fc_{B'}$). 



By Proposition \ref{Prop:CorrKer}, we have $L_1=\Psi(O_X[1])$.  Since $\OO_X[1]$ is $\sigma^L_b$-semistable of phase 1, it follows that 
\ $L_1 \in \Pc (1)$.  Hence  $L_1\in \T_{\B'}$.  We have now shown $\Tc_{\omega'_0, B'_0}$ and $\mathcal{C}_1$ are both contained in $\Tc_{\Bc'}$, and so  $\T'_1\subseteq \T_{\B'}$.

\textbf{Part 3.} Now consider the case where $E$ is $\mu_{\omega'_0, B'_0}$-semistable, $\mu_{\omega'_0, B'_0}(E)=0$ and $E\in \F'_1$. We claim that  $E\in \F_{\B'}$. To show this, we will denote the following as condition $(*)$ for a sheaf $I$:
\[
(*)\text{\quad} \text{$I$ is  $\mu_{\omega'_0, B'_0}$-semistable}, \mu_{\omega'_0, B'_0}(I)=0, I\in\F'_1, \text{ and } I\in\T_{\B'}.
\]
Suppose $E\notin \F_{\B'}$; then we have a short exact sequence in $\Coh(X)$
\begin{equation*}
	0\rightarrow E_T\rightarrow E\rightarrow E_F\rightarrow 0
\end{equation*}
with $0\neq E_T\in \T_{\B'}$ and $E_F\in \F_{\B'}$.
Since $E\in \F'_1$, we have $E_T\in \F'_1$. 
Since $E$ is $\mu_{\omega'_0, B'_0}$-semistable with $\mu_{\omega'_0,B'_0}(E)=0$ while $\mu_{\omega'_0, B'_0}(E_T)\geq 0$ by Remark \ref{rem:AG52-17-1}, we must have  $\mu_{\omega'_0, B'_0}(E_T)=0$ and that $E_T$ is $\mu_{\omega'_0, B'_0}$-semistable. Hence $E_T$ satisfies $(*)$.  

Denote $E_T$ by $E^{(0)}$, and consider any nonzero  map $\gamma^{(0)} : E^{(0)}\rightarrow \C(x^{(0)})$ in $\Coh(X)$, for some $x^{(0)} \in X$. If $\gamma^{(0)}$ is not surjective in the abelian category $\B'_{H,-2}$, then taking its image and cokernel in $\B'_{H,-2}$,  we can consider the short exact sequence in $\B'_{H,-2}$
\begin{equation}\label{eq:igcxcg}
0\to \image \gamma^{(0)} \to \mathbb{C}(x^{(0)}) \to \cokernel \gamma^{(0)} \to 0.
\end{equation}
Note that here $\image \gamma^{(0)}$ is a sheaf.  

Recall from the start of the proof that every skyscraper sheaf $\mathbb{C}(x)$ is of the form $\Psi (Q'[1])$ for some $Z_H$-stable object $Q'[1]$ in $\Pc_b (\tfrac{3}{4}, \tfrac{7}{4}]$ with phase $\tfrac{7}{4}$.  Since $\Bc'_{H,-2}=\Pc (\tfrac{3}{4}, \tfrac{7}{4}]$, the short exact sequence \eqref{eq:igcxcg} in $\Pc (\tfrac{3}{4}, \tfrac{7}{4}]$ corresponds to some short exact sequence in $\Pc_b (\tfrac{3}{4}, \tfrac{7}{4}]$
\[
0 \to M \to Q'[1] \to R \to 0.
\]
The stability of $Q'[1]$ now implies $\phi_b (R)=\tfrac{7}{4}$ and that $\phi_b (M) < \tfrac{7}{4}$.  Filtering $R$ using the slicing $\Pc_b$, we see that $R$ cannot lie in $\kernel Z_H$; it follows that  $M$  lies in $\kernel Z_H$.  It follows that $\image \gamma^{(0)}$ lies in $\Bc'_{H,-2} \cap (\kernel Z_H (\Psi^{-1}(-)))$.  Then  $\gamma^{(0)}$ must  either 


\begin{itemize}
    \item[(i)] be a surjection in $\B'_{H, -2}$, or 
    \item[(ii)] factor through a sheaf $F$ (i.e.\ $\image \gamma^{(0)}$ above) such that 
    \[
    F\in \Bc'_{H,-2} \cap (\kernel Z_H (\Psi^{-1}(-)))
    \]
    in which case $F\in \langle L_1 \rangle$ by Lemma \ref{lem:AG52-17-2}.
\end{itemize}
For any $i \geq 0$, we will set $E^{(i+1)}$ to be the kernel of $\gamma^{(i)}$ in $\Bc'_{H,-2}$, and define $\gamma^{(i+1)}$ to be any nonzero map $E^{(i+1)} \to \mathbb{C}(x^{(i+1)})$ in $\Coh (X)$ for some $x^{(i+1)} \in X$.  By the same argument as above, for each $i \geq 0$, the map $\gamma^{(i)}$ must be in either case (i) or case (ii).

We now show by induction that for each $i \geq 0$, regardless of whether we are in case (i) or case (ii), $E^{(i+1)}$ always satisfies (*).  So assume $E^{(i)}$ satisfies (*) and consider the short exact sequence in $\Bc'_{H,-2}$ 
\[
0 \to E^{(i+1)} \to E^{(i)} \overset{\gamma^{(i)}}{\longrightarrow}  \image \gamma^{(i)} \to 0.
\]
Note that this is also a short exact sequence in $\Coh (X)$.
If $\gamma^{(i)}$ is in case (i), then $\ch_j(E^{(i+1)})=\ch_j(E^{(i)})$ for $j=0,1$; if $\gamma^{(i)}$ is in case (ii), then $Z_H(\Psi^{-1}(\image \gamma^{(i)}))=0$ which implies that 
\[
\omega'_0\cdot\ch_1^{B'_0}(\image \gamma^{(i)})=0.
\]  In either case, we have  $\mu_{\omega'_0, B'_0}(E^{(i+1)})=\mu_{\omega'_0, B'_0}(E^{(i)})=0$, and so  $E^{(i+1)}$ is again a torsion-free sheaf that is $\mu_{\omega'_0, B'_0}$-semistable. Since $\Fc_1'$ is a torsion-free class in  $\Coh (X)$, we also have $E^{(i+1)}$ lies in $\Fc_1'$.  In addition, since $E^{(i+1)}$ is a coherent sheaf that lies in $\Bc'_{H,-2}$, it lies in $\Tc_{\Bc'}$ by definition.  Hence $E^{(i+1)}$ also satisfies (*).

Next, we show that as $i$ increases from $0$, case (ii) can only happen a finite number of times to the $\gamma^{(i)}$ before (i) happens.  
If $\gamma^{(i)}$ is in case (ii), we have $\ch_0(E^{(i+1)})<\ch_0(E^{(i)})$.  
Hence after (ii) happens $m$ times for $m$ large enough, we have $E^{(m)}\in \langle L_1\rangle$.
Since $E^{(i)}$ lies in $\Fc_1'$, we have   $\Hom (L_1,E^{(i)})=0$ which is a contradiction.  

Therefore,  every $\gamma^{(i)}$ is in case (i) for $i \gg 0$, in which case
\[ 
Z_{\omega'_0, B'_0}(E^{(i+1)})=Z_{\omega'_0, B'_0}(E^{(i)})+1.
\]
It then follows that for $i \gg 0$, the image of $E^{(i)}$ under $Z_H \circ \Psi^{-1}$, 
cannot lie in the half plane 
\[
\{ re^{i\phi \pi} : r \in \mathbb{R}_{\geq 0}, \phi \in (\tfrac{3}{4}, \tfrac{7}{4}]\}
\]
which contradicts $E^{(i)} \in \Bc'_{H,-2}=\Pc(\tfrac{3}{4}, \tfrac{7}{4}]$.  This completes the proof that $E \in \Fc_{\Bc'}$.
 

We can now finish the proof of this theorem: $\F_1'\subseteq \F_{\B'}$. For any $E \in \Fc_1'$, we consider its $\mu_{\omega'_0,B'_0}$-HN filtration
\[
 0 \neq E_1 \subsetneq E_2 \subsetneq \cdots \subsetneq E_m=E
\]
 in $\Coh (X)$. Then $E_1\in \Fc_1'$. If $\mu_{\omega'_0,B'_0}(E_1)= 0$, by our argument in Part 3, $E_1$ lies in $\Fc_{\Bc'}$.  Also, any HN factor $E_i/E_{i-1}$ satisfying $\mu_{\omega'_0, B'_0}(E_i/E_{i-1})<0$ can be further filtered by $\mu_{\omega'_0,B'_0}$-stable torsion-free sheaves with $\mu_{\omega'_0,B'_0}<0$, each of which must lie in $\Fc_{\Bc'}$ by the argument in Part 2.  Overall, $E$ has a filtration by objects in $\Fc_{\Bc'}$ and hence itself lies in $\Fc_{\Bc'}$.
	\end{proof}


\begin{lem}\label{lem:AG52-68-1}
For any real numbers $c, d$ satisfying $c\leq d \leq c+1$, if $E \in \Pc_b (c,d)$ then $\phi_b (E) \in (c,d)$.
\end{lem}

\begin{proof}
    Suppose $E \in \Pc_b (c,d)$.  Then there is a filtration of $E$ by exact triangles
    \[
0 = E_0 \to E_1 \to \cdots \to E_m = E
    \]
    such that each mapping cone $C_i := \mathrm{cone}(E_{i-1} \to E_i)$ lies in $\Pc_b (t_i)$ for some $t_i \in (b,c)$, and $c> t_1 > \cdots > t_m > b$.  If $C_i$ does not lie in $\kernel (Z_H)$ for some $1 \leq i \leq m$, then $Z_H (E)\neq 0$ and we have $\phi_b (E) \in (c,d)$.  If $C_i \in \kernel (Z_H)$ for all $i$, then since the phases of objects in $\kernel (Z_H) \cap \Pc_b (c,d)$ are defined as limits of $\phi_{V_\omega, D_\omega}$,   we  still  have $\phi_b (E) \in (c,d)$.
\end{proof}

\begin{lem}\label{lem:AG52-67-1}
Suppose $a \in (0,1]$.  Then for any $t \in (a,a+1]$, every object in $\Pc_b (t)$ is $Z_H$-semistable in $\Pc_b (a,a+1]$.
\end{lem}

\begin{proof}
Suppose $E \in \Pc_b (t)$.  Take any short exact sequence 
  \[
  0 \to E' \to E \to E'' \to 0
  \]
  in $\Pc_b (a,a+1]$.  Since $\Pc_b$ is a slicing, we must have $E' \in \Pc_b (a,t]$ and $E'' \in \Pc_b (t,a+1]$.  By Lemma \ref{lem:AG52-68-1}, we have $\phi (E') \leq t \leq \phi (E'')$.
\end{proof}

\begin{lem}\label{lem:AG52-68-1b}
Every object in $\Pc_b (\tfrac{3}{4}, \tfrac{7}{4}]\cap \kernel(Z_H)$ is a direct sum of copies of $\OO_X[1]$ and $\OO_\Theta (-1)[1]$.
\end{lem}

\begin{proof}
  Suppose $E \in \Pc_b (\tfrac{3}{4}, \tfrac{7}{4}]\cap \kernel(Z_H)$.  Then we can fit $E$ in a short exact sequence in the heart $\Pc_b (a,a+1]$
  \[
  0 \to E' \to E \to E'' \to 0
  \]
  where $E' \in \Pc_b (1,\tfrac{7}{4}]$ and $E'' \in \Pc_b (\tfrac{3}{4},1]$.  Since $Z_H$ is still a weak stability function on $\Pc_b (\tfrac{3}{4},\tfrac{7}{4}]$, we must have $E'[-1] \in \Pc_b(0,\tfrac{3}{4}] \cap \kernel(Z_H)$ and $E'' \in \Pc_b (\tfrac{3}{4},1] \cap \kernel(Z_H)$.  From the definition of a slicing, we know that $E'[-1]$ has a filtration by objects in $\Pc_b (t) \cap \kernel(Z_H)$ where $t \in (0, \tfrac{3}{4}]$, each of which must be semistable.  Hence $E''[-1]$ can only be a direct sum of copies of $\OO_\Theta (-1)$.  Similarly, $E''$ can only be a direct sum of copies of $\OO_X[1]$.  Since we know $\Ext^1 (\OO_X[1], \OO_\Theta (-1)[1])=0$ from  
  \cite[Lemma 5.5]{CLSY1}, the claim follows.
\end{proof}

We define  $$\Psi'=\cO_X(\alpha)\otimes(\Phi^{-1}(\_)[-1])$$ so that we have $\Psi\circ\Psi'=\mathrm{id}_\mathcal{D}[-1] = \Psi'\circ\Psi$.
\begin{thm}
\label{Cor:corrstab1}
	Given $b\in\mathbb{P}^1$, we have 
the line bundle $\cO_X(\alpha)$ is $\sigma^L_b$-stable.

\end{thm}

\begin{proof}
	We first check that we are in Configuration III-w:
 
 (a) We take the autoequivalences $\Psi'$ and $\Psi$, then by definition we have $\Psi'\circ\Psi=\Psi\circ\Psi'=\mathrm{id}_\mathcal{D}[-1]$. 
 
 (b) We take the two abelian categories to be $\B$ and $\B^0_{H, -2}$.
  From Theorem \ref{Thm:CorrHeart1}, we have $\B=\B'_{H, -2}=\Pc ((\tfrac{3}{4},\tfrac{7}{4}])$ where $\Pc$ denotes the slicing of $\Phi \cdot \cP_b$. Hence we have $$\Psi'(\B)=\cP_b(-\tfrac{1}{4}, \tfrac{3}{4}].$$ 
  This implies that $$\Psi'(\B)\subset \langle \B^0_{H, -2}, \B^0_{H, -2}[-1]\rangle.$$

 (c) We take $T=-g^{-1}$.
Then  \eqref{eq:AG52-17-2} implies that we have 
\[
Z_{H}(\Psi'(E))=TZ'_0(E).
\]
Since the action of $g$ on $\mathbb{R}^2$ decreases the phase of an object by $\frac{3}{4}\pi$, the action of $T$ decreases the phase of an object by $\frac{1}{4}\pi$. We define $\Gamma_T:\phi\to \phi-\frac{1}{4}$. Using the notation in Configuration III-w, we can take $a=0$ and $b=0$, hence we also have
\[
0<\Gamma_T^{-1}(0)<1.
\]

Now we check the condition in Proposition \ref{prop:weakstabconcorr-var1}. We know that $\Psi'(\cO_\Theta(-2))=\cO_X(\alpha)$ and $\cO_X(\alpha)\in \T^0_H$. Since $$\Hom(\cO_X(\alpha), \cO_\Theta(m))=H^0(\cO_\Theta(-D_\alpha+m)),$$
	we have $\Hom(\cO_X(\alpha), \cO_\Theta(m))=0$ for $D_\alpha\geq 0$ and $m\leq -2$. Hence $\cO_X(\alpha)\in \B^0_{H, -2}$, and $\cO_\Theta(-2)$ is $\Psi'_{\B^0_{H, -2}}\text{-WIT}_0$. 
 Recall that 
	$$\ker(Z_{H})\cap \B^0_{H, -2}=\langle \cO_\Theta(-1), \cO_X[1]\rangle.$$ Hence there is no short exact sequence in $\B^0_{H, -2}$ 
 \begin{equation}
 \label{Equ:corstab1}
	0\rightarrow M\rightarrow \Psi'(\cO_\Theta(-2))\rightarrow N\rightarrow 0
 \end{equation}
	such that $M\in \ker(Z_{H})\cap \B^0_{H, -2}$.
 
 On the other hand, assume that $N\in \ker(Z_{H})\cap \B^0_{H, -2}$. We have 
 \[
 \Hom(\cO_X(\alpha), \cO_\Theta(-1))\simeq H^0(\cO_\Theta(-D_\alpha-1))=0
 \] 
 for $D_\alpha\geq 0$. Since every object in $\Bc^0_{H,-2} \cap \kernel (Z_H)$ is a direct sum of copies of $\OO_\Theta (-1)$ and $\OO_X[1]$ by  
 \cite[Lemma 5.5]{CLSY1}, it follows that $N\in \langle \cO_X[1]\rangle$.  Note that $\Im Z_H (\OO_X(\alpha))= H\alpha = D_\alpha + e>0$, and so $\phi_{H,-2}(\OO_X(\alpha))<1$.  Note also that $\phi_{H, -2}(\cO_X[1])=1$.  Therefore, by the weak seesaw property, we must have $\phi_{H,-2}(M) < \phi_{H,-2}(N)$. 
By Theorem \ref{Thm:OTheta(-2)stable}, we know that $\cO_\Theta(-2)$ is $\sigma^R_a$-stable for any $a\in \mathbb{R}$. Then the stable case of Proposition \ref{prop:weakstabconcorr-var1} implies that $\cO_X(\alpha)=\Psi'(O_\Theta(-2))$ is $\sigma^L_b$-stable for any $b\in\mathbb{R}$. 
\end{proof}

\begin{thm}
	\label{Thm:goingupV}
	$L:=\cO_X(\alpha)$ is $\sigma_{V, H}$-stable for all $V>0$.
\end{thm}
\begin{proof}
	Recall from \cite[Proposition 5.14]{CLSY1}, 
we have for any $V\in \mathbb{R}_{\geq0}$, $\sigma_{V, H}$ have the same heart, and $Z_{V, H}=Z_{H}$ when $V=0$. 
	Consider a weak stability condition $\sigma_{V, H}$.  A short exact sequence in $\B^0_{H, -2}$
	\begin{equation}
		\label{Equ:Thmgoingup1-a}
		0\rightarrow E\rightarrow L\rightarrow F\rightarrow 0
	\end{equation}
	 is a destabilizing short exact sequence for $L$ if $\phi_{V,H}(E)> \phi_{V,H}(F)$.  By the weak seesaw property, this occurs if  either 
	
	(i) $\phi_{V, H}(E)> \phi_{V, H}(L)$, or
	
	(ii) $\phi_{V, H}(E)= \phi_{V, H}(L)>\phi_{V, H}(F)$, in which case $E \notin \kernel Z_{V,H}$ and so $F$ must lie in $\kernel Z_{V,H}$, and so  $F\in \langle\cO_\Theta(-1)\rangle$.
 
	But when $D_\alpha\geq0$, $\Hom(L, \cO_\Theta(-1))=0$, we only need to consider case (i).  Note that if $\Im Z_{V,H}(E)=0$, then since $\Im Z_{V,H}(L)\neq 0$, we have $\phi_b^L(E)> \phi_b^L(L)$, contradicting the $\sigma_b^L$-stability of $L$.  So we must have $\Im Z_{V,H}(E)\neq 0$.
 
 Since $E\notin \ker Z_{V, H}\cap B^0_{H, -2}$, it makes sense to talk about the Bridgeland slope of $E$. Let $\rho_{V, H}=-\frac{\Re(Z_{V, H})}{\Im(Z_{V, H})}$.
	Then $\rho_{V, H}(E)> \rho_{V, H}(L)$ is equivalent to 
	$$-\frac{\Re(Z_{V, H}(E))}{\Im(Z_{V, H}(E))}>-\frac{\Re(Z_{V, H}(L))}{\Im(Z_{V, H}(L))}$$
	which is equivalent to 
	\begin{equation}\label{eq:thm12-15-1}
 (\mu_{H}(E)-\mu_H(L))V> \frac{(H\cdot \ch_1(E))\ch_2(L)-(H\cdot \ch_1(L))\ch_2(E)}{\ch_0(L)\ch_0(E)}.
 \end{equation}
Also note that $L$ is a $\mu_H$-semistable sheaf, and so from the long exact sequence of sheaves of \eqref{Equ:Thmgoingup1-a}, we obtain $Hc_1(E) \leq Hc_1(L)$ while $\ch_0(E) \geq 1 = \ch_0(L)$, giving us  $\mu_H(E)\leq \mu_H(L)$.  As a result, we have   $\rho_{V,H}(E) > \rho_{V,H}(F)$   in  \eqref{Equ:Thmgoingup1-a} if one of the following cases holds:


(i) $\mu_H(E)=\mu_H(L)$ and \eqref{eq:thm12-15-1} holds for all $V\in \mathbb{R}_{\geq 0}$; 
	

 (ii) $\mu_H(E) < \mu_H (L)$ and 
 \[V< \frac{(H\cdot \ch_1(E))\ch_2(L)-(H\cdot \ch_1(L))\ch_2(E)}{\ch_0(L)\ch_0(E)(\mu_H(E)- \mu_H(L))}.
 \]

	Hence if there exists $E$ 
	that destabilizes $L$ at $\sigma_{V, H}$ for some $V >0$, then $E$ also destabilizes $L$ at $\sigma^L_b$ for  
	any $b\in\mathbb{P}^1$, which contradicts $L$ is $\sigma^L_b$-stable for any $b\in\mathbb{P}^1$. 
	

 This shows that $L$ is $\sigma_{V, H}$-semistable for all $V>0$. Assume that $L$ is strictly semistable at $\sigma_{V', H}$ for some $V'>0$. Then there is a short exact sequence as in \eqref{Equ:Thmgoingup1-a} such that 
 \[
 \phi_{V', H}(E)=\phi_{V', H}(L)=\phi_{V', H}(F).
 \]

 Then we have 
 \begin{equation}\label{eq:thm12-15-1'}
 (\mu_{H}(E)-\mu_H(L))V'= \frac{(H\cdot \ch_1(E))\ch_2(L)-(H\cdot \ch_1(L))\ch_2(E)}{\ch_0(L)\ch_0(E)},
 \end{equation}
 which implies that the right hand side of equation \eqref{eq:thm12-15-1'} must be smaller or equal to zero.
 Since $L$ is $\sigma_b^L$-stable for any $b\in \mathbb{P}^1$, we must have $\phi_b^L(E)\leq \phi_b^L(L)$, and so overall we must have
 \[
 \frac{(H\cdot \ch_1(E))\ch_2(L)-(H\cdot \ch_1(L))\ch_2(E)}{\ch_0(L)\ch_0(E)}= 0.
 \]
Then $\phi_b^L(E)=\phi_b^L(L)<\phi_b^L(F)$ for any $b\in \mathbb{P}^1$.  Note that $E$ cannot lie in $\kernel Z_H$ since $\kernel (Z_H) \cap \Bc^0_{H,-2} = \langle \OO_\Theta (-1) \rangle$ while $\Hom (\OO_\Theta (-1),L)=0$.   
So we must have   $F\in \ker(Z_H)\cap\B^0_{H, -2}$ (or else $\phi_b^L(E)<\phi_b^L(L)$, a contradiction). Since $\Hom(L, \cO_\Theta(-1))=0$, we have $F\simeq \oplus \cO_X[1]$. But $\phi_{V, H}(\cO_X[1])=1$ for all $V\geq 0$, contradicting $\phi_{V', H}(L)=\phi_{V', H}(F)$.

Hence $L$ is $\sigma_{V, H}$-stable for all $V>0$.
\end{proof}

Recall that $\Upsilon=\Phi(\cdot)\otimes \cO(f)$. We define 
\[
\Upsilon'=\Phi^{-1}(\_\otimes \cO(-f))[-1].
\]
Then by definition we have $\Upsilon'\circ\Upsilon=\Upsilon\circ\Upsilon'=\mathrm{id}_\mathcal{D}[-1]$.
Let $h\in \text{GL}^+(2, \mathbb{R})$ be defined by 
\begin{equation}\label{eq:matrixhdef}
	h:=
	\begin{pmatrix}
		0 & 1\\ 
		-1 & 0
	\end{pmatrix}
\end{equation}
A quick computation shows 
\begin{equation}
\label{Equ:CentralCharforZ_D}
	Z_{D}(\Upsilon(E))=h\cdot Z_{V, H}(E)
\end{equation}
when $D=V$.
To apply Proposition \ref{prop:weakstabconcorr-var1}, we first explore the relation between $\B^0_{H, -2}$ and $\Coh^{\omega, 0}$, where $\omega=\Theta+(e+V)f$.
\begin{thm}
	\label{Thm:CorrHeart2}
 Given $V\in\mathbb{R}_{>0}$, let $\cP_{V, H}$ denote the slicing for the weak stability condition $\sigma_{V, H}$. 
	Let $\cQ=\Upsilon\cdot\cP_{V, H}$, and denote the heart $\cQ((\frac{1}{2}, \frac{3}{2}])$ by $\B_\Upsilon$, then we have 
$$\B_\Upsilon=\Coh^{\omega, 0},$$
where $\omega=\Theta+(e+V)f$.
\end{thm}
\begin{proof}
	The proof is similar to the proof of Theorem \ref{Thm:CorrHeart1}. We only write down the arguments where the proofs are different. 
	
We have $$\C(x)=\Upsilon(Q[1])$$ where $Q$ is a pure sheaf supported on dimension $1$ with $\ch(Q)=\ch(\cO_f)$. By Proposition \ref{Prop:Ofstableweak} and Remark \ref{rem:Qstableweak}, we have $Q[1]$ lies in $\cP_{V, H}(\frac{3}{2})$ and is $\sigma_{V, H}$-stable. Note that $\ker(Z_{V, H})\cap\cP_{V, H}((\frac{1}{2}, \frac{3}{2}])=\langle\cO_\Theta(-1)[1]\rangle$. Since $\phi_{V, H}(\cO_\Theta(-1))=\frac{1}{2}$, we have objects in $\ker(Z_{V, H})\cap\cP_{V, H}((\frac{1}{2}, \frac{3}{2}])$ always have phase $\phi_{V, H}=\frac{3}{2}$.

Claim: $Q[1]$ is also stable in $\cP_{V, H}((\frac{1}{2}, \frac{3}{2}])$.

By the similar argument as in Theorem \ref{Thm:CorrHeart1}, we have a short exact sequence in $\cP_{V, H}((\frac{1}{2}, \frac{3}{2}])$:
\[
0\to E'\to Q[1]\to F'\to 0,
\]
with $E'\in \cP_{V, H}(\frac{3}{2})$ and $Q[1]\in \cP_{V, H}(\frac{3}{2})$. Since $\phi_{V, H}(\cO_\Theta(-1)[1])=\frac{3}{2}$, we have $F'\in \cP_{V, H}(\frac{3}{2})$. Then the short exact sequence 
\[
0\to E'[-1]\to Q\to F'[-1]\to 0
\]
implies that $Q$ is strictly semistable in $\sigma_{V, H}$, which leads to a contradiction. 

 Let $A$ be an object such that $A=\Upsilon(A')$, where $A'$ is a $\sigma_{V, H}$-stable object in $\cP_{V, H}(t)$. Note that $\ker(Z_{V,H}) \cap P_{V,H}(\frac{1}{2},\frac{3}{2}])=\cP_{V, H}(\frac{3}{2})$, then a similar argument in Theorem \ref{Thm:CorrHeart1} also implies \[
 \Hom(\C(x), A)=0.
 \] 
 Since $\Upsilon(\cO_\Theta(-1))=\cO_X$, we have 
 $\Hom(\C(x), \Upsilon(\cO_\Theta(-1)))=0$, and $\Hom(\C(x), \Upsilon(\cO_\Theta(-1)[1]))=0$.
Let $E$ be an arbitrary object in $\B_\Upsilon$, then $E$ is an extension of objects $\Upsilon(\cO_\Theta(-1))$, $\Upsilon(\cO_\Theta(-1)[1])$ and $A=\Upsilon(A')$ where $A'$ is $\sigma_{V, H}$-stable in $\cP_{V, H}(t)$ for $\frac{1}{2}<t\leq \frac{3}{2}$.
Hence \cite[Proposition 5.4]{BMef} implies that $E$ is quasi-isomorphic to a two-term complex of locally free sheaves $F^\bullet$ such that $F^i=0$ for $i>0$ or $i<-1$. This means we have
	$$\B_\Upsilon\subset\langle \Coh(X), \Coh(X)[1]\rangle.$$
	We denote $\Coh(X)\cap\B_\Upsilon$ by $\T_\Upsilon$, and $\Coh(X)\cap \B_\Upsilon[-1]$ by $\F_\Upsilon$.
	We need to show that $\T_{\omega, 0}\subset \T_\Upsilon$, and $\F_{\omega, 0}\subset \F_\Upsilon$. 
	
	Let $E\in \Coh(X)$ be $\mu_{\omega, 0}$- stable. By the same argument of Parts 1 and 2 of Theorem \ref{Thm:CorrHeart1}, we have if $E$ is a torsion sheaf, then $E\in \T_\Upsilon$. If $\mu_{\omega, 0}(E)>0$, then $E\in \T_\Upsilon$. If $\mu_{\omega, 0}(E)<0$, then $E\in \F_\Upsilon$. 
	
	Finally we consider $E$ is $\mu_{\omega, 0}$-semistable and $\mu_{\omega, 0}(E)=0$. We want to show $E\in\F_\Upsilon$. If not, then there exists a short exact sequence in $\Coh(X)$:
	$$0\rightarrow E_T\rightarrow E\rightarrow E_F\rightarrow 0$$
	where $0\neq E_T\in\T_\Upsilon$ and $E_F\in\F_\Upsilon$.
  Denote $E_T$ by $E^{(0)}$. Similar to the argument in Theorem \ref{Thm:CorrHeart1}, 
	consider any nonzero map $\gamma^{(0)}: E^{(0)}\rightarrow \C(x^{(0)})$ for some $x^{(0)}\in X$. 
 If $\gamma^{(0)}$ is not surjective in the abelian category $\B_\Upsilon$, then taking its image and cokernel in $\B_\Upsilon$,  we can consider the short exact sequence in $\B_\Upsilon$
\begin{equation}\label{eq:igcxcg'}
0\to \image \gamma^{(0)} \to \mathbb{C}(x^{(0)}) \to \cokernel \gamma^{(0)} \to 0.
\end{equation}
  This corresponds to a short exact sequence in $\cP_{V, H}((\frac{1}{2}, \frac{3}{2}])$
  \[
0 \to M \to Q[1] \to R \to 0,
\]
where $\C(x^{(0)})=\Upsilon(Q[1])$ and $Q[1]$ is stable in $\cP_{V, H}((\frac{1}{2}, \frac{3}{2}])$ with $\phi_{V, H}(Q[1])=\frac{3}{2}$. The stability of $Q[1]$ in $\cP_{V, H}((\frac{1}{2}, \frac{3}{2}])$ then implies that $\phi_{V, H}(R)=\frac{3}{2}$. 

Now suppose $M\notin \kernel (Z_{V,H})$.  If $R \in \kernel(Z_{V,H})$, then $\phi_{V,H}(M)=\phi_{V,H}(Q[1])=\tfrac{3}{2}$ making it impossible for $\phi_{V,H}(M) < \phi_{V,H}(R)$ to hold. If $R \notin \kernel(Z_{V,H})$ instead, then we obtain $\phi_{V,H}(M) <  \phi_{V,H}(Q[1])=\tfrac{3}{2} < \phi_{V,H}(R)$ which again is impossible.  Hence we must have $M \in \kernel(Z_{V,H})$.  Since objects in $\ker(Z_{V, H})\cap \cP_{V, H}((\frac{1}{2}, \frac{3}{2}])$ always have phase $\frac{3}{2}$, we have $\phi_{V, H}(M)=\frac{3}{2}$. This contradicts $Q[1]$ is stable in $\cP_{V, H}((\frac{1}{2}, \frac{3}{2}])$.

 Hence we have $E^{(0)}\rightarrow \C(x^{(0)})$ is also a surjection in $\B_\Upsilon$. Recall that for any $i \geq 0$, we set $E^{(i+1)}$ to be the kernel of $\gamma^{(i)}$ in $\B_\Upsilon$, and define $\gamma^{(i+1)}$ to be any nonzero map $E^{(i+1)} \to \mathbb{C}(x^{(i+1)})$ in $\Coh (X)$ for some $x^{(i+1)} \in X$. The same argument as in the last two paragraphs shows that $\gamma^{(i+1)}$ is surjective for all $i \geq 0$.
 Consider the short exact sequences in $\B_\Upsilon$
 \[
 0\to E^{(i+1)}\to E^{(i)}\to \C(x^{(i)})\to 0.
 \]
 
 We have $Z_D(E^{(i+1)})=Z_D(E^{(i)})+1$.
For $i\gg 0$, $Z_{V, H}(\Upsilon^{-1}(E^{(i)}))$ does not lie in the half plane 
\[
\{re^{i\phi\pi}|r\in \mathbb{R}_{\geq 0}, \phi\in(\frac{1}{2}, \frac{3}{2}]\}
\]
which leads to a contradiction. The rest of the argument is the same as Theorem \ref{Thm:CorrHeart1}.
\end{proof}

\begin{cor}
\label{Cor:stableDaxis}
	$\Upsilon(L)$ is $\sigma_D$-stable for all $D>0$.
	\end{cor}
\begin{proof}
Given $D\in\mathbb{R}_{>0}$, let $V=D$.
	We first check that we are in Configuration III-w:

 (a) We take the autoequivalences $\Upsilon$ and $\Upsilon'$, then by definition we have $\Upsilon\circ\Upsilon'=\Upsilon'\circ\Upsilon=\mathrm{id}_\mathcal{D}[-1]$.
 
(b) We take the two abelian categories to be $\B^0_{H, -2}$ and $\Coh^{\omega, 0}$, where $\omega=\Theta+(V+e)f$. Denote the slicing of $\sigma_{V, H}$ by $\cP_{V, H}$, we have
	\[
 \B^0_{H, -2}=\cP_{V, H}(0,1])\subset\Big\langle\cP_{V, H}((\frac{1}{2}, \frac{3}{2}]), \cP_{V, H}((\frac{1}{2}, \frac{3}{2}])[-1]\Big\rangle
 \]
	Since $\B_\Upsilon=\Coh^{\omega, 0}$, we have 
 \[
 \Upsilon(\cP_{V, H}(0,1])\subset \langle \Coh^{\omega, 0}, \Coh^{\omega, 0}[-1]\rangle.
 \]

 (c) We take $T=h$ where $h$ is as in \eqref{eq:matrixhdef}, then  \eqref{Equ:CentralCharforZ_D} implies that
\[
Z_{D}(\Upsilon(E))=T\cdot Z_{V, H}(E).
\]
	Since the action of $T$ decreases the phase of objects by $\pi/2$, we define 
 \[
 \Gamma_T:\phi\to\phi-\frac{1}{2}.
 \]
Using the notation in Configuration III-w, we take $a=0$, $b=0$, then we have 
\[
0<\Gamma_T^{-1}(0)<1.
\]
Now we check the conditions in Proposition \ref{prop:weakstabconcorr-var1}.
	 
By the computation in Section \ref{para:spcase1}, if $D_\alpha+1-V>0$, then $L$ is $\Upsilon_{\Coh^{\omega, 0}}\text{-WIT}_0$. 
 Since $\ker(Z_D)\cap \Coh^{\omega, 0}=\langle\cO_X[1]\rangle$ and $\Upsilon(L)$ is a line bundle from \ref{para:spcase1}, There is no short exact sequence in $\Coh^{\omega, 0}$
	$$0\rightarrow M\rightarrow \Upsilon(L)\rightarrow N\rightarrow 0$$
	with $M\in \ker(Z_D)\cap \Coh^{\omega, 0}$. We denote the phase of $\sigma_D$ by $\phi_D$. If $N\in \ker(Z_D)\cap \Coh^{\omega, 0}$, we have $\phi_D(N)=1$; since $\Im Z_D (\Upsilon (L))=D_\alpha + 1-V>0$, it follows that $\phi_D (\Upsilon(L))<1$, and so by the weak seesaw property we have  $\phi_D(M) < \phi_D(N)$.
	
	If $D_\alpha+1-V\leq0$, then $L$ is $\Upsilon_{\Coh^{\omega, 0}}\text{-WIT}_1$. 
	Assume we have a short exact sequence in $\Coh^{\omega, 0}$
	$$0\rightarrow M\rightarrow \Upsilon(L)[1]\rightarrow N\rightarrow 0.$$
	Since $\ch_1(\Upsilon(L))=-\Theta+(D_\alpha+1)f$, any morphism of sheaves $M[-1] \to \Upsilon (L)$ must vanish on the generic fiber of the elliptic fibration (by considering fiber degree).  Since $\Upsilon (L)$ is  a line bundle, this forces  $\Hom(M, \Upsilon(L))=0$, hence $M\notin \ker(Z_D)\cap \Coh^{\omega, 0}$. Now suppose $N\in \ker(Z_D)\cap \Coh^{\omega, 0}$.  In this case, we have  $\phi_D(N)=\phi_D(\cO_X[1])=1$.  If $D_\alpha + 1 -V <0$ then $\Im Z_D (\Upsilon (L)[1])>0$ and so, by the weak seesaw property as in the previous paragraph, we have $\phi_D (M) < \phi_D (N)$.  So suppose $D_\alpha+1-V=0$ from now on.  Since $N$ is a direct sum of copies of $\OO_X[1]$,  the morphism $\Upsilon(L)[1] \to N$ induces a nonzero  morphism of sheaves $\Upsilon (L) \to \OO_X$ that must be injective (since $\Upsilon (L)$ is a line bundle), and so its cokernel is a torsion sheaf with $c_1 = \Theta - (D_\alpha + 1)f$ which is not effective, a contradiction.

 Since $L$ is $\sigma_{V, H}$-stable for all $V>0$, the stable case of Proposition \ref{prop:weakstabconcorr-var1} implies that $\Upsilon(L)$ is $\sigma_D$-stable for all $D>0$.
	\end{proof}
 
Consider the central charge given by
	$$Z_{V, D}(E)=-\ch_2(E)+V\ch_0(E)+i(\Theta+(D+e)f)\cdot \ch_1(E).$$
	Denote $\Theta+(D+e)f$ by $\omega$, and let $\sigma_{V, D}=(Z_{V, D}, \Coh^{\omega, 0})$.
\begin{thm}
	\label{Thm:goingupD}
	$\Upsilon(L)$ is $\sigma_{V, D}$-stable for all $D>0$, $V>0$.
\end{thm}
\begin{proof}
Fix $D\in \mathbb{R}_{>0}$. Let $\omega=\Theta+(D+e)f$.
	If $L$ is $\Upsilon_{\Coh^{\omega, 0}}\text{-WIT}_0$, then $\Upsilon (L)$ is $\sigma_D$-stable for all $D>0$ by Corollary \ref{Cor:stableDaxis}, and the argument of Theorem \ref{Thm:goingupV} shows the $\sigma_{V,D}$-stability of $\Upsilon(L)$ for all $V>0$ with only two minor differences:
  (1) we don't actually get case (ii) 
 for the short exact sequence \eqref{Equ:Thmgoingup1-a} since $\sigma_{V,D}$'s are Bridgeland stability conditions (and so $Z_{V,D}$ has zero kernel). (2) since $L$ is $\Upsilon_{\Coh^{\omega, 0}}\text{-WIT}_0$, we have $\Upsilon(L)\in \Coh^{\omega, 0}$. Then $\Im (Z_{V, D}(\Upsilon(L)))>0$, hence we still get the contradiction in last part of the proof of Theorem \ref{Thm:goingupV}.

 We consider the case when $L$ is $\Upsilon_{\Coh^{\omega, 0}}\text{-WIT}_1$. 
	
	Note that we have $Z_{V, D}=Z_D$ when $V=0$, and the heart for $\sigma_{V, D}$ is $\Coh^{\omega, 0}$ for any $V\geq 0$.
	Consider $\sigma_{V, D}$ for some $V>0$. Let $\rho_{V, D}=-\frac{\Re(Z_{V, D})}{\Im(Z_{V, D})}$. Then a short exact sequence in $\Coh^{\omega, 0}$
	\begin{equation}
		\label{Equ:Thmgoingup1}
		0\rightarrow E\rightarrow \Upsilon(L)[1]\rightarrow F\rightarrow 0
	\end{equation}
	 is a destabilizing short exact sequence for $\Upsilon(L)[1]$ if $\rho_{V, D}(E)> \rho_{V, D}(\Upsilon(L)[1])$. Note that by the same argument in Theorem \ref{Thm:goingupV}, we have $\Im(Z_{V, D}(E))\neq 0$. Then the inequality $\rho_{V, D}(E)> \rho_{V, D}(\Upsilon(L)[1])$ is equivalent to 
	\begin{align}
	    (\omega\cdot \ch_1(E)\ch_0(\Upsilon(L)[1])-\omega\cdot \ch_1(\Upsilon(L)[1])&\ch_0(E))V \notag\\
 &> \omega\cdot \ch_1(E)\ch_2(\Upsilon(L)[1])-\omega\cdot \ch_1(\Upsilon(L)[1])\ch_2(E). \label{eq:Thm12d8d1-1}
 \end{align}
 Note that $F \in \Coh (X)[1]$ and $F [-1]$ is a torsion-free sheaf.  Since $\ch_0(\Upsilon(L)[1])=-1$, we have $\ch_0(E)=\ch_0(\Upsilon(L)[1])-\ch_0(F)\geq 0$. Then $\rho_{V, D}(E)>\rho_{V, D}(\Upsilon(L)[1])$ if either

(i) $\omega\cdot \ch_1(E)\ch_0(\Upsilon(L)[1])-\omega\cdot \ch_1(\Upsilon(L)[1])\ch_0(E)=0$, in which case \eqref{eq:Thm12d8d1-1} holds for all $V\in \mathbb{R}_{\geq 0}$, or


(ii)  $\omega\cdot \ch_1(E)\ch_0(\Upsilon(L)[1])-\omega\cdot \ch_1(\Upsilon(L)[1])\ch_0(E)<0$, in which case \eqref{eq:Thm12d8d1-1} is in turn equivalent to  $$V< \frac{\omega\cdot \ch_1(E)\ch_2(\Upsilon(L)[1])-\omega\cdot \ch_1(\Upsilon(L)[1])\ch_2(E)}{\omega\cdot \ch_1(E)\ch_0(\Upsilon(L)[1])-\omega\cdot \ch_1(\Upsilon(L)[1])\ch_0(E))}.$$

Thus if there exists $E$ 
destabilizing $\Upsilon(L)[1]$ at $\sigma_{V, D}$, then $E$ destabilizes $\Upsilon(L)[1]$ at $\sigma_D$, which contradicts Corollary \ref{Cor:stableDaxis}. 

Hence we have $\Upsilon(L)$ is $\sigma_{V, D}$-semistable for all $V$. Assume that $\Upsilon(L)$ is strictly $\sigma_{V', D}$-semistable for some $V'>0$. Then there is a short exact sequence as in \eqref{Equ:Thmgoingup1} such that $\phi_{V', D}(E)=\phi_{V', D}(\Upsilon(L)[1])=\phi_{V', D}(F)$. By the same argument in Theorem \ref{Thm:goingupV}, we have $\phi_D(E)=\phi_D(\Upsilon(L)[1])$ and $F\in \ker(Z_D)\cap \Coh^{\omega, 0}$. But since $\ker(Z_D)\cap \Coh^{\omega, 0}=\langle\cO_X[1]\rangle$ and $\Hom(\Upsilon(L), \cO_X)=0$, such a short exact sequence cannot exist.

Hence $\Upsilon(L)$ is $\sigma_{V, D}$-stable for all $V>0$.
	\end{proof}

\begin{thm}
The line bundle $L=\cO_X(\alpha)$ is
$\sigma_{\omega, 0}$-stable for all ample $\omega\in\langle\Theta, f\rangle$. 
\end{thm}
\begin{proof}
By \eqref{eq:AG52-21-2}, we have 
\[
\Phi\cdot\sigma_{\omega, f}\cdot \Tilde{g}=\sigma_{\omega', 0},
\]
where $V_\omega=D_{\omega'}$ and $D_\omega=V_{\omega'}$. This implies that
\[
\Upsilon\cdot\sigma_{\omega, 0}\cdot \Tilde{g}=\sigma_{\omega', 0}.
\]
Then the result follows from Theorem \ref{Thm:goingupD}.
\end{proof}

For the rest of this section, we prove Theorem 1.4.
Let $F$ be a torsion free coherent sheaf on $X$. We denote \begin{equation}\label{eq:chE}
  \ch_0(F)=n, \text{\quad} f\ch_1(F)=d, \text{\quad}\Theta \ch_1(F)=c, \text{\quad} \ch_2(F)=s,
\end{equation}
then 
\begin{align*}
    \ch_0 (\Phi F) &= d \\
    \ch_1 (\Phi F) &\equiv -c_1(F) + (d-n)\Theta +(c+\tfrac{e}{2}d+s)f \\
    \ch_2(\Phi F) &= -(c+ed-\tfrac{e}{2}n).
 \end{align*}
 Recall $\Psi=\Phi(\_\otimes\cO_X(-\alpha))$. To emphasize the dependence of $\Psi$ on $\alpha$, we denote this exact functor by $\Psi_\alpha$ for the rest of this paper.
 Then we have the following result generalizing Theorem \ref{Thm:OTheta(-2)stable}:

\begin{thm}
\label{thm:torsionstablestrong}
Let $L$ be a line bundle such that $d=1$, $-c+s=1$ and $c>-2$. 
then $\Psi_\alpha(L)$ is $\sigma_a^R$-stable for all $a\in \mathbb{P}^1$. 
\end{thm}
\begin{proof}
Note that when $d=1$, $L$ is $\Psi\text{-WIT}_1$ and $\Phi\text{-WIT}_0$.
Furthermore, taking $e=2$ we have the following Chern characters:
\begin{align*}
    \ch_0 (\Psi_\alpha L) &= 0 \\
    \ch_1 (\Psi_\alpha L) &\equiv -c_1(L) +(1+s)f \\
    \ch_2(\Psi_\alpha L) &= -(c-D_\alpha-1).
 \end{align*}
Then $\Psi_\alpha(L)[1]$ is a torsion sheaf, we denote it by $P$. 
 The condition $c>-2$ implies that $L\in \B^0_{H, -2}$. The condition $-c+s=1$ implies that $\Im(Z'_0(P))=\frac{1}{2}=C_{\omega'_0}$. These two conditions also imply $\phi_{\sigma_a}(P)>\frac{1}{4}$.
The argument of Theorem \ref{Thm:OTheta(-2)stable} carries over except for the following statement in Case 1 (ii): 
assuming there exists a s.e.s in $\B$
\[
0\rightarrow E\rightarrow P\rightarrow Q\rightarrow 0,
\]
if $H^{-1}(Q)\in\ker(Z'_0)$, then $H^0(Q)\neq 0$.

We prove this statement directly. 
Assuming $H^0(Q)=0$, then we have the following s.e.s. in $\Coh(X)$:
\begin{equation}
\label{Equ:H0Q=0}
0\rightarrow L_0^{\oplus m}\rightarrow E\rightarrow P\rightarrow 0.
\end{equation}
 Since $\Hom(L_0, L_1)=0$, we have $E\notin \B_{\ker(Z'_0)}:=\ker(Z'_0)\cap\B$. Assuming there exists $E_i$ in the HN filtration of $E$ such that $E_i\notin \B_{\ker(Z'_0)}$, then we form the following diagram of s.e.s. in $\Coh(X)$:
 	\begin{equation}
  \label{Equ:diag}
			\begin{tikzcd}
				& 0\ar{d}& 0\ar{d} & 0\ar{d} &\\
				0\ar{r}& A	\ar{r}\ar{d}& E_i\ar{r}\ar{d} &P' \ar{r}\ar{d} & 0 \\
				0\ar{r} & L_0^{m}\ar{r}\ar{d}&E \ar{r} \ar{d}&P\ar{r}\ar{d} &0\\
				0\ar{r}&S\ar{r}\ar{d} & E/E_i \ar{r}\ar{d} & T\ar{r}\ar{d}& 0\\
				&0 & 0&0 &
			\end{tikzcd}
		\end{equation} 
where $P'$ is the image of $E_i\rightarrow P$. 
Since $P$ is pure sheaf supported in dimension $1$, $P'$ is also a pure sheaf supported in dimension $1$. Since $\omega'_0$ is ample, we have $0<\Im Z'_0(P')\leq \Im Z'_0(P)=\frac{1}{2}$, hence $\Im Z'_0(P')=\frac{1}{2}$. Furthermore, since $E_i\notin \B_{\ker(Z'_0)}$, we must have $\Im Z'_0(E_i)>0$, which implies that $\Im Z'_0(E_i)\geq\frac{1}{2}$. We have two cases:

{\bf Case(i)}  $\Im Z'_0(E_i)>\frac{1}{2}$. Then the s.e.s. in the first row implies that $\Im Z'_0(A)>0$, hence $\mu_{\omega'_0, B'_0}(A)>0$. Since $\mu_{\omega'_0, B'_0}(L_0)=0$, $A$ destabilizes $L_0$ contradicting that $L_0$ is $\mu_{\omega'_0, B'_0}$-stable. 

{\bf Case(ii)}  $\Im Z'_0(E_i)=\frac{1}{2}$. This implies that $\Im Z'_0(E/E_i)=0$, then $E/E_i\simeq L_1^{\oplus n}$ for some $n$. Since $\Hom(L_0, L_1)=0$, we have $S\simeq 0$ and $E/E_i\simeq T$. But $T$ is a torsion sheaf, again we obtain a contradiction. 

The above discussion forces such an $E_i$ does not exists. Then $E$ fits in a s.e.s.
\[
0\rightarrow L_1^n\rightarrow E\rightarrow J\rightarrow 0
\]
where $J$ is the last HN factor of $E$. Then $E\notin \B_{\ker Z'_0}$ implies that $J\notin \B_{\ker Z'_0}$. Take the JH filtration of $J$, we obtain a surjection $E\rightarrow W$ where $W$ is $\mu_{\omega'_0, B'_0}$-stable. Then we obtain the same diagram as \ref {Equ:diag} by replacing $E/E_i$ by $W$. Furthermore $W\not\simeq L_1$, since otherwise we obtain the same contradiction as in Case (ii). Then $\mu_{\omega'_0, B'_0}(W)>0$. Since $\Im Z'_0(P')=\Im Z'_0(P)=\frac{1}{2}$, we have $\Im Z'_0(T)=0$. Then $\Im Z'_0(S)=\Im Z'_0(W)>0$. Since $T$ is torsion, we have $\ch_0(S)=\ch_0(W)$, which implies that $\mu_{\omega'_0, B'_0}(S)=\mu_{\omega'_0, B'_0}(W)$, contradicting $W$ is $\mu_{\omega'_0, B'_0}$-stable. 
Hence the s.e.s. \ref{Equ:H0Q=0} cannot exist.
\end{proof}

\begin{thm}
\label{thm:mainstrong}
Let $L$ be a line bundle such that $d=1$, $-c+s=1$ and $c>-1$. 
Then $L$ is $\sigma_{\omega, 0}$-stable for all $\omega\in\langle\Theta, f\rangle$.
\end{thm}
\begin{proof}

We check the proof of Theorem \ref{Cor:corrstab1}, Theorem \ref{Thm:goingupV}, Corollary \ref{Cor:stableDaxis} and Theorem \ref{Thm:goingupD} goes through under the above assumptions. 

For Theorem \ref{Cor:corrstab1}, we only need to check the condition in Proposition 9.9. Denoting the torsion sheaf $\Psi(L)[1]$ by $P$, then $\Psi'(P)=L$. Since every object in $\Bc^0_{H,-2} \cap \kernel (Z_H)$ is a direct sum of copies of $\OO_\Theta (-1)$ and $\OO_X[1]$ by  
 \cite[Lemma 5.5]{CLSY1}, there is no short exact sequence in $\B^0_{H, -2}$ 
 \begin{equation}
 \label{Equ:corstab1}
	0\rightarrow M\rightarrow \Psi'(P)\rightarrow N\rightarrow 0
 \end{equation}
	such that $M\in \ker(Z_{H})\cap \B^0_{H, -2}$.
 Assume that $N\in \ker(Z_{H})\cap \B^0_{H, -2}$. By assumption, we can write $\ch_1(L)=\Theta+(c+2)f+\delta$ where $\delta\in\langle\Theta, f\rangle^\perp$. Since $c>-1$, we have 
$\Hom(L, \cO_\Theta(-1))=H^0(\cO_\Theta(-1-c))=0$, it follows that $N\in \langle \cO_X[1]\rangle$.  Also $c>-2$ implies that $\Im Z_H (L)>0$. Then the rest of the argument carries over. 

For Theorem \ref{Thm:goingupV}, going up lemma along $V$-axis, as above we have 
$\Hom(L, \cO_\Theta(-1))=0$, then the arguments completely carries over. 

For Corollary \ref{Cor:stableDaxis}, again we only need to check the condition in Proposition 9.9. First by Lemma \ref{lem:lbfdonetr}, we have $\Upsilon(L)$ is a line bundle. Assume $\Im Z_D(\Upsilon(L))>0$, then $L$ is $\Upsilon_{\Coh^{\omega, 0}}\text{-WIT}_0$. Since $\ker(Z_D)\cap \Coh^{\omega, 0}=\langle\cO_X[1]\rangle$, there is no short exact sequence in $\Coh^{\omega, 0}$
\begin{equation*}
 0\rightarrow M\rightarrow \Upsilon(L)\rightarrow N\rightarrow 0
 \end{equation*}
	with $M\in \ker(Z_D)\cap \Coh^{\omega, 0}$.
If $N\in \ker(Z_D)\cap \Coh^{\omega, 0}$, we have $\phi_D(N)=1$; since $\Im Z_D (\Upsilon (L))>0$, it follows that $\phi_D (\Upsilon(L))<1$, and so by the weak seesaw property we have  $\phi_D(M) < \phi_D(N)$. 
Assume $\Im Z_D(\Upsilon(L))\leq 0$, then $L$ is $\Upsilon_{\Coh^{\omega, 0}}\text{-WIT}_1$.
	Assume we have a short exact sequence in $\Coh^{\omega, 0}$
	\begin{equation}
 \label{Equ:ses9.9}
	    0\rightarrow M\rightarrow \Upsilon(L)[1]\rightarrow N\rightarrow 0.
     \end{equation}
Since $\ch_1(L)\cdot f=1$, we have $\ch_1(\Upsilon(L))\cdot f=-1$. Then $\Hom(\cO_X, \Upsilon(L))=0$, since otherwise there exists a s.e.s. in $\Coh(X)$ 
\[
0\rightarrow \cO_X\rightarrow \Upsilon(L)\rightarrow R\rightarrow 0
\]
where $R$ is a torsion sheaf such that $c_1(R)\cdot f=-1$, which leads to a contradiction. 
This implies that $M\notin \ker(Z_D)\cap \Coh^{\omega, 0}$. Now suppose $N\in \ker(Z_D)\cap \Coh^{\omega, 0}$.  In this case, we have  $\phi_D(N)=\phi_D(\cO_X[1])=1$.  If $\Im Z_D (\Upsilon (L)[1])>0$, by the weak seesaw property, we have $\phi_D (M) < \phi_D (N)$.  So suppose $\Im Z_D (\Upsilon (L)[1])=0$ from now on.  Since $N$ is a direct sum of copies of $\OO_X[1]$,  the morphism $\Upsilon(L)[1] \to N$ induces a nonzero morphism $\Upsilon (L) \to \OO_X$ which fits into a s.e.s.
\[
0\rightarrow \Upsilon(L)\rightarrow \cO_X\rightarrow R'\rightarrow 0.
\]
  By assumption we have $-c+s=1$ and $c\geq -1$, which implies that $s\geq 0$.
Then $\ch_1(\Upsilon(L))\cdot\Theta=s+2\geq 2$, which implies that $c_1(R')\cdot\Theta\leq -2$. Since $R'$ is a torsion sheaf, $c_1(R')$ is effective. We write $c_1(R')=k\Theta+bf+\sum_i a_iC_i$, where $C_i$'s are other irreducible curves on $X$ and all the coefficients are greater or equal to zero.  Since $c_1(R')\cdot \Theta$ is negative, we must have  $k>0$ . Since $c_1(R')\cdot f=1$, we can only have $k=1$ and $a_i=0$ for all $i$.
Then $c_1(R')\cdot\Theta\leq -2$ further implies that $b\leq 0$, hence $b=0$, i.e. 
$\ch_1(\Upsilon(L))=-\Theta$. 
Since $\Hom(\cO_X(-\Theta), \cO_X)$ is $1$-dimensional, the map from $\cO_X(-\Theta)$ to $\cO_X$ is injective, and  the s.e.s \ref{Equ:ses9.9} is of the from 
\begin{equation}
\label{Equ:Ninker}
0\rightarrow G\rightarrow \cO_X(-\Theta)[1]\rightarrow \cO_X^{\oplus n}[1]\rightarrow 0,
\end{equation}
where $G\in \Coh(X)$. Then $G\notin \ker(Z_D)\cap\Coh^{\omega, 0}$, and $\Im(Z_D(G))>0$.
This implies that $\phi_D(G)<\phi_D(\cO_X[1])=1$, and the condition in Proposition \ref{prop:weakstabconcorr-var1} is satisfied.

For Theorem \ref{Thm:goingupD}, 
the argument carries over till the end, assume that $\Upsilon(L)$ is strictly $\sigma_{V,D}$-stable for some $V>0$. Denoting the distabilizing s.e.s. in $\Coh^{\omega, 0}$ by
\[
0\rightarrow E\rightarrow \Upsilon(L)[1]\rightarrow F\rightarrow 0.
\]
 Then $F\in\ker(Z_D)\cap \Coh^{\omega, 0}$, and this s.e.s. must be of the form in \ref{Equ:Ninker}. Then $\phi_{V, D}\cO_X[1]=1$ for any $V>0$. Since $\phi_{V, D}(E)<1$ for any $V>0$, we can never have $\phi_{V,D}(E)=\phi_{V, D}(\Upsilon(L)[1])=\phi_{V, D}(F)$, hence $\Upsilon(L)$ is $\sigma_{V, D}$-stable for all $V>0$.
\end{proof}

\bibliography{refs}{}

\begin{thebibliography}{10}

\bibitem{ABL}
D.~Arcara, A.~Bertram, and M.~Lieblich.
\newblock {Bridgeland-stable moduli spaces for K-trivial surfaces}.
\newblock {\em J. Eur. Math. Soc.}, 15(1):1--38, 2013.

\bibitem{ARCARA20161655}
D.~Arcara and E.~Miles.
\newblock Bridgeland stability of line bundles on surfaces.
\newblock {\em J.\ Pure Appl.\ Algebra}, 220(4):1655--1677, 2016.

\bibitem{FMNT}
C.~Bartocci, U.~Bruzzo, and D.~Hern\'{a}ndez-Ruip\'{e}rez.
\newblock {\em Fourier-Mukai and Nahm Transforms in Geometry and Mathematical Physics}, volume 276.
\newblock Birkh\"{a}user, 2009.
\newblock Progress in Mathematics.

\bibitem{BayerPBSC}
A.~Bayer.
\newblock {Polynomial Bridgeland stability conditions and the large volume limit}.
\newblock {\em Geom. Topol.}, 13:2389--2425, 2009.

\bibitem{StabTC}
T.~Bridgeland.
\newblock Stability conditions on triangulated categories.
\newblock {\em Ann. Math.}, 166:317--345, 2007.

\bibitem{BMef}
T.~Bridgeland and A.~Maciocia.
\newblock {Fourier-Mukai transforms for K3 and elliptic fibrations}.
\newblock {\em J. Algebraic Geom.}, 11:629--657, 2002.

\bibitem{GaoChen}
Gao Chen.
\newblock The {J}-equation and the supercritical deformed {H}ermitian-{Y}ang-{M}ills equation.
\newblock {\em Invent. Math.}, 225(2):529--602, 2021.

\bibitem{ChuLee}
Jianchun Chu and Man-Chun Lee.
\newblock Hypercritical deformed hermitian-yang-mills equation revisited, 2022.

\bibitem{ChuLeeTak}
Jianchun Chu, Man-Chun Lee, and Ryosuke Takahashi.
\newblock A nakai-moishezon type criterion for supercritical deformed hermitian-yang-mills equation, 2021.

\bibitem{CLSY1}
T.~Collins, J.~Lo, Y.~Shi, and S.-T. Yau.
\newblock Weak stability conditions as limits of {B}ridgeland stability conditions.
\newblock Preprint., 2024.

\bibitem{CJY}
T.~C. Collins, A.~Jacob, and S.-T. Yau.
\newblock $(1,1)$ forms with specified {L}agrangian phase: A priori estimates and algebraic obstructions.
\newblock {\em Camb. J. Math.}, 8(2), 2020.

\bibitem{CoShdHYM}
T.~C. Collins and Y.~Shi.
\newblock {Stability and the deformed Hermitian–Yang–Mills equation}.
\newblock In {\em Surv. Differ. Geom.}, volume~24, pages 1--38. 2019.

\bibitem{CollinsYau}
Tristan~C. Collins and Shing-Tung Yau.
\newblock Moment maps, nonlinear pde, and stability in mirror symmetry, 2018.

\bibitem{CollinsYaugeo}
Tristan~C. Collins and Shing-Tung Yau.
\newblock Moment maps, nonlinear {PDE} and stability in mirror symmetry, {I}: geodesics.
\newblock {\em Ann. PDE}, 7(1):Paper No. 11, 73, 2021.

\bibitem{DatarPing}
Ved~V. Datar and Vamsi~Pritham Pingali.
\newblock A numerical criterion for generalised {M}onge-{A}mp\`ere equations on projective manifolds.
\newblock {\em Geom. Funct. Anal.}, 31(4):767--814, 2021.

\bibitem{DP}
Jean-Pierre Demailly and Mihai Paun.
\newblock Numerical characterization of the {K}\"{a}hler cone of a compact {K}\"{a}hler manifold.
\newblock {\em Ann. of Math. (2)}, 159(3):1247--1274, 2004.

\bibitem{DonaldsonASDYMC}
S.~K. Donaldson.
\newblock Anti self-dual yang-mills connections over complex algebraic surfaces and stable vector bundles.
\newblock {\em Proc. Lond. Math. Soc.}, S3-50(1):1--26, 1985.

\bibitem{DFR}
Michael~R. Douglas, Bartomeu Fiol, and Christian R\"{o}melsberger.
\newblock Stability and {BPS} branes.
\newblock {\em J. High Energy Phys.}, (9):006, 15, 2005.

\bibitem{LYZ}
Naichung~Conan Leung, Shing-Tung Yau, and Eric Zaslow.
\newblock From special {L}agrangian to {H}ermitian-{Y}ang-{M}ills via {F}ourier-{M}ukai transform.
\newblock {\em Adv. Theor. Math. Phys.}, 4(6):1319--1341, 2000.

\bibitem{LLM}
W.~Liu, J.~Lo, and C.~Martinez.
\newblock {Fourier-Mukai transforms and stable sheaves on Weierstrass elliptic surfaces}.
\newblock Preprint. arXiv:1910.02477 [math.AG], 2019.

\bibitem{LLMII}
W.~Liu, J.~Lo, and C.~Martinez.
\newblock {Fourier-Mukai transforms and stable sheaves on Weierstrass elliptic surfaces, II}.
\newblock In preparation, 2022.

\bibitem{Lo7}
J.~Lo.
\newblock {Stability and Fourier-Mukai transforms on elliptic fibrations}.
\newblock {\em Adv. Math.}, 255:86--118, 2014.

\bibitem{Lo11}
J.~Lo.
\newblock t-structures on elliptic fibrations.
\newblock {\em Kyoto J. Math.}, 56(4):701--735, 2016.

\bibitem{Lo20}
J.~Lo.
\newblock Weight functions, tilts, and stability conditions.
\newblock Preprint. arXiv:2007.06857 [math.AG], 2020.

\bibitem{LM2}
J.~Lo and C.~Martinez.
\newblock Geometric stability conditions under autoequivalences and applications: Elliptic surfaces.
\newblock Preprint. arXiv:2210.01261 [math.AG], 2022.

\bibitem{LZ3}
J.~Lo and Z.~Zhang.
\newblock Duality spectral sequences for {W}eierstrass fibrations and applications.
\newblock {\em J. Geom. Phys.}, 123:362--371, 2018.

\bibitem{Mac14}
A.~Maciocia.
\newblock Computing the walls associated to bridgeland stability conditions on projective surfaces.
\newblock {\em Asian J. Math.}, 18(2):263--280, 2014.

\bibitem{MSlec}
E.~Macr\'{i} and B.~Schmidt.
\newblock Lectures on {B}ridgeland stability.
\newblock In {\em Moduli of Curves}, volume~21 of {\em Lecture Notes of the Unione Matematica Italiana}, pages 139--211. Springer, 2017.

\bibitem{MMMS}
Marcos Mari\~{n}o, Ruben Minasian, Gregory Moore, and Andrew Strominger.
\newblock Nonlinear instantons from supersymmetric {$p$}-branes.
\newblock {\em J. High Energy Phys.}, (1):Paper 5, 32, 2000.

\bibitem{Martinez2017}
C.~Martinez.
\newblock Duality, {B}ridgeland wall-crossing and flips of secant varieties.
\newblock {\em Internat. J. Math.}, 28(2):1750011, 2017.

\bibitem{piyaratne2015moduli}
D.~Piyaratne and Y.~Toda.
\newblock Moduli of {B}ridgeland semistable objects on 3-folds and {D}onaldson-{T}homas invariants.
\newblock {\em J. Reine Angew. Math.}, 2015.

\bibitem{UYHYMC}
K.~Uhlenbeck and S.~T. Yau.
\newblock On the existence of hermitian-yang-mills connections in stable vector bundles.
\newblock {\em Comm. Pure Appl. Math.}, 39(S1):S257--S293, 1986.

\bibitem{YauMA}
Shing~Tung Yau.
\newblock On the {R}icci curvature of a compact {K}\"{a}hler manifold and the complex {M}onge-{A}mp\`ere equation. {I}.
\newblock {\em Comm. Pure Appl. Math.}, 31(3):339--411, 1978.

\end{thebibliography}
\bibliographystyle{plain}

\end{document}